%% file: spec_prop_paper.tex
\documentclass[10pt]{amsart}
\usepackage{amssymb}
\usepackage{amsmath}
\usepackage{graphicx}
\usepackage{subfigure}
\usepackage{url}
\usepackage{hyperref}

\usepackage{cite}

\numberwithin{equation}{section}

\include{gsmacros}

\begin{document}

\title[Eigenvalues and NLS]{Embedded Eigenvalues and the Nonlinear
  Schr\"odinger Equation}

\author[Asad]{Reza Asad}\email{mehrdad.asad@utoronto.ca}
\address{Department of Mathematics\\ University of Toronto\\Toronto,
  Canada}

\author[Simpson]{Gideon Simpson}\email{simpson@math.toronto.edu}
\address{Department of Mathematics\\ University of Toronto\\Toronto,
  Canada}

\date{\today}

\maketitle

\begin{abstract}
  A common challenge to proving asymptotic stability of solitary waves
  is understanding the spectrum of the operator associated with the
  linearized flow.  The existence of eigenvalues can inhibit the
  dispersive estimates key to proving stability.  Following the work
  of Marzuola \& Simpson, we prove the absence of embedded eigenvalues
  for a collection of nonlinear Schr\"odinger equations, including
  some one and three dimensional supercritical equations, and the
  three dimensional cubic-quintic equation.  Our results also rule out
  nonzero eigenvalues within the spectral gap and, in 3D, endpoint
  resonances.

  The proof is computer assisted as it depends on the sign of certain
  inner products which do not readily admit analytic representations.
  Our source code is available for verification at
  \url{http://www.math.toronto.edu/simpson/files/spec_prop_asad_simpson_code.zip}.
\end{abstract}

\section{Introduction}
\label{s:intro}
The nonlinear Schr\"odinger equation (NLS) in $\R^{d+1}$ dimensions,
\begin{equation}
  \label{e:nls}
  i \psi_t + \nabla^2 \psi + f (|\psi|^2) \psi =  0, \quad \psi (0,\bx) = \psi_0 (\bx), 
\end{equation}
appears as a leading order approximation to a wide variety of natural
and engineered systems, including optics, plasma physics, and fluid
mechanics.  It is also a model equation for the studying the
competition between nonlinear and dispersive effects.  See Sulem \&
Sulem, \cite{sulem1999nse}, for details and examples.

\subsection{Solitons}

For appropriate choices of the nonlinearity, $f: \R \to \R$, the
equation will possess solitary wave solutions, smooth localized
solutions of \eqref{e:nls} satisfying the ansatz
\begin{equation}
  \label{e:soliton_ansatz}
  \psi(t,\bx) = e^{i \omega t} R(\bx; \omega), \quad \omega > 0.
\end{equation}
$R$ then solves the nonlinear elliptic equation
\begin{equation}
  \label{e:soliton}
  -\omega R + \nabla^2 R + f(\abs{R}^2)R = 0.
\end{equation}
The solutions are assumed to have finite $L^2$ norm and vanish at
infinity.

These nonlinear bound states, deemed solitary waves or solitons, are
of interest both as mathematical objects and in applications.  This
raises the question of whether or not these solutions are stable to
the flow of \eqref{e:nls}.  Developing tools for assessing stability
is the main purpose of this paper.

\subsection{Stability}
With regard to solitary waves, the two types of stability typically
discussed are {\it orbital} stability and {\it asymptotic} stability.
Orbital stability, alternatively called Lyapunov or modulational
stability, asserts that if the data is close to a solitary wave, it
remains close, modulo a group of symmetries associated with the
equation.

In \cite{Weinstein:1985p66,Weinstein:1986p65}, M.I. Weinstein proved
that orbital stability of a solution to \eqref{e:soliton} could be
assessed by computing the sign of
\begin{equation}
  \label{e:slope_condition}
  \frac{d}{d\omega}\int \abs{R({\bf x} ;\omega)}^2 d {\bf x}.
\end{equation}
When it is positive, the solitons are stable; when it is negative,
they are unstable.  This work was subsequently generalized for
Hamiltonian equations by Grillakis, Shatah, \& Strauss
\cite{Grillakis:1987p115,Grillakis:1990p116}.

For NLS with a power, or monomial, nonlinearity of the form $f(s) =
s^\sigma$, the transition between stable and unstable solitons is
determined by the product $\sigma d$.  If $\sigma d < 2$, the solitary
waves are stable, otherwise they are unstable.  These regimes, $\sigma
d < 2, = 2, >2$ correspond to the subcritical, critical, and
supercritical forms of \eqref{e:nls}, and are intimately related to
the well-posedness of NLS, \cite{sulem1999nse}.

Of course, orbital stability does not tell us of the limiting behavior
of the solution.  An orbitally stable soliton could perpetually
oscillate through the symmetry group of the equation.  It is our
expectation that perturbations of stable solitons diminish as $t\to
\infty$, and the solution relaxes to a particular solitary wave with a
fixed set of parameters.  We turn to asymptotic stability to
understand the dynamics as $t\to \infty$.

Asymptotic stability is usually proven by expanding \eqref{e:nls}
about the solitary wave solution, $\psi = e^{i \omega t} (R(\cdot;
\omega) + u + i v)$ to arrive at an equation for the perturbation, $p
= u+iv$,
\begin{equation}
  \label{e:linearize_nls}
  \begin{split}
    \partial_t \begin{pmatrix} u \\ v \end{pmatrix} &=
    JL \begin{pmatrix}
      u\\ v \end{pmatrix}  + {\bf F}(u,v)\\
    &= \begin{pmatrix} 0 & 1 \\ -1 & 0 \end{pmatrix}\begin{pmatrix}
      L_+ & 0 \\ 0 & L_-\end{pmatrix}\begin{pmatrix} u \\
      v \end{pmatrix} + {\bf F}.
  \end{split}
\end{equation}
${\bf F}$ contains terms that are nonlinear in the perturbation, and
the scalar operators $L_\pm$ are given by
\begin{subequations}
  \label{e:linear_ops}
  \begin{align}
    L_+ &= - \nabla^2 + \omega -f(R^2) - 2 f'(R^2)R^2= -\nabla^2 +
    \omega + V_+,\\
    L_- & = - \nabla^2 + \omega- f(R^2)= - \nabla^2 +\omega + V_-.
  \end{align}
\end{subequations}
Proving asymptotic stability of the solitary wave will require
Strichartz estimates of the form
\begin{equation}
  \label{e:strichartz}
  \norm{e^{JLt} f}_{L^p_t L^q_x} \lesssim \norm{f}_{L^r_x}.
\end{equation}
With such an estimate, one can conclude linear stability of the
soliton from the temporal decay of a solution to the linear problem
\[ {\bf p}_t = JL {\bf p}, \quad {\bf p} = (u,v)^T.
\]
This decay can then be used to show that the nonlinear part of the
flow, ${\bf F}$ in \eqref{e:linearize_nls}, is dominated by the linear
part can can be treated perturbatively.  Successful implementations
include Buslaev \& Perelman, \cite{BusPer}, Buslaev \& Sulem,
\cite{Buslaev:2003p5760}, Cuccagna and Rodnianski, Soffer \& Schlag,
\cite{Cuc, RodSchSof} and more recently to Schlag and Krieger \&
Schlag, \cite{KS1,Schlag:2009p346}.  These last two works, where the
authors show the asymptotic stability of a constrained soliton which
is orbitally {\it unstable}, are closely related to the present
results.


\subsection{The Spectrum \& Embedded Eigenvalues}

Estimates of the form \eqref{e:strichartz} require adequate knowledge
of the spectrum of $JL$, $\sigma(JL)$.  Since the solitons are highly
localized, $JL$ is easily shown to be a {\it relatively compact
  perturbation} of
\begin{equation}
  \begin{pmatrix} 0 & - \nabla^2 + \omega\\ \nabla^2 - \omega & 0\end{pmatrix}
\end{equation}
which has as its essential spectrum
\[
(-i\infty,-i\omega] \cup [i\omega, i\infty)
\]
This can easily be computed by the Fourier transform.  Thus
\begin{equation}
  \label{e:spectral_estimate}
  \sigma_\ess(JL) = i(-\infty, -\omega]\cup i[\omega, \infty) \subseteq \sigma(JL) 
\end{equation}
See \cite{Erdogan:2006p347} for additional details.  By direct
computations, one can show that the origin is an eigenvalue of $JL$ of
algebraic multiplicity {\it at least} $2d+2$.  See section
\ref{s:discrete} below for the elements of the kernel, and Figure
\ref{f:spec} for a visualization of the spectrum of the problems
considered here.

\begin{figure}
  \includegraphics[width=2.4in]{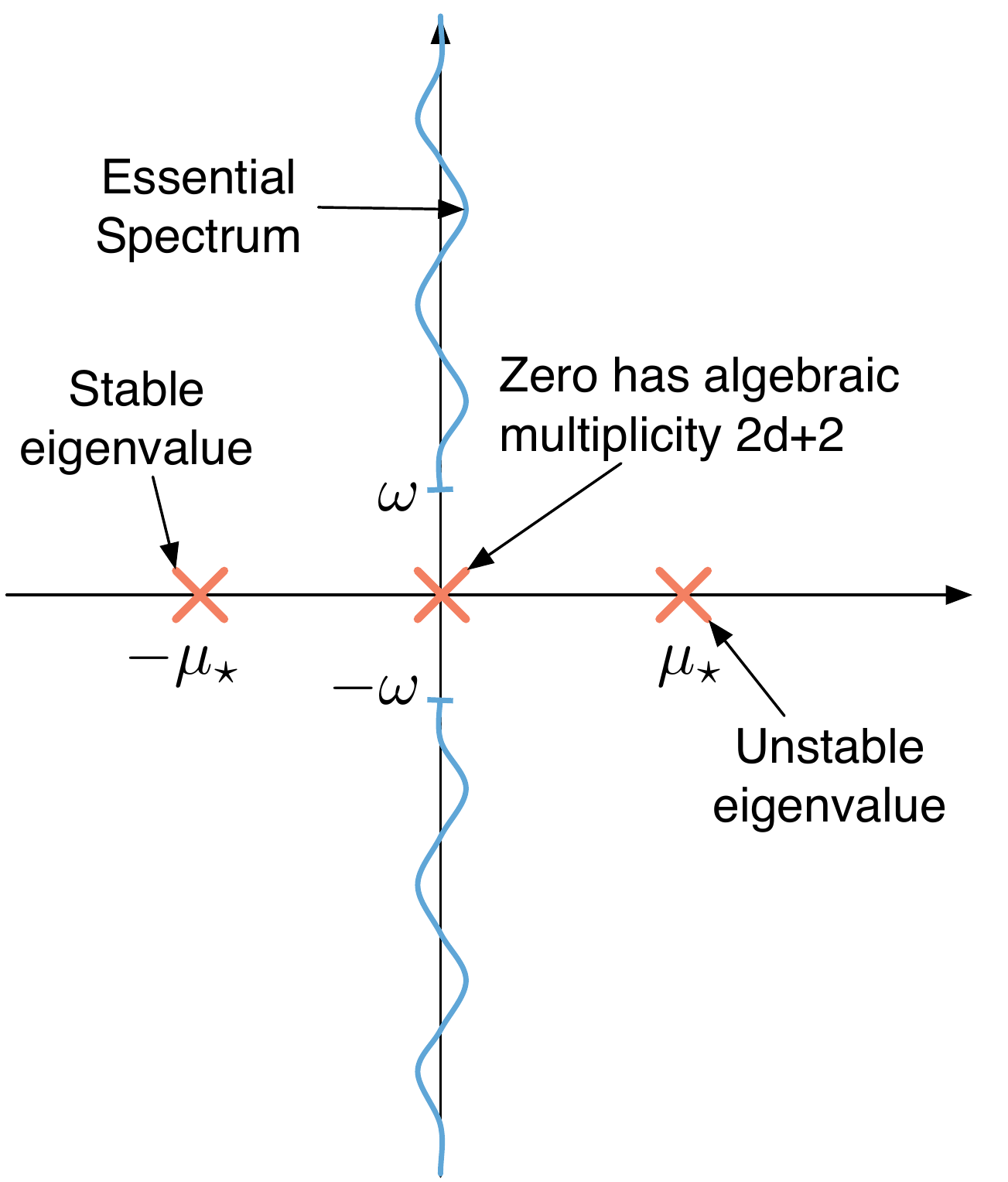}
  \caption{The generic spectrum of $JL$ for an orbitally unstable
    soliton.  The instability is due to the positive eigenvalue,
    $\mu_\star$. }
  \label{f:spec}
\end{figure}

The existence of purely imaginary eigenvalues both within the {\it
  spectral gap}, $[-i\omega, i\omega]$, and embedded in the essential
spectrum and resonances can obstruct the Strichartz estimates.  Many
results on asymptotic stability of NLS solitons, including
\cite{Buslaev:2003p5760,Erdogan:2006p347,Schlag:2009p346}, have
explicitly assumed:
\begin{enumerate}
\item $JL$ has no embedded eigenvalues;
\item The only eigenvalue within the spectral gap is zero;
\item The endpoints of the essential spectrum, $\pm i \omega$, are not
  resonances.
\end{enumerate}
We are thus motivated to find ways of rigorously proving these
assumptions.

Some previous results on these spectral questions were addressed by
Krieger \& Schlag in 1D for $f(s)= s^\sigma$ with $\sigma > 4$
(supercritical), \cite{KS1}.  Alternatively, Cuccagna \& Pelinovksy
and Cuccagna, Pelinovsky, \& Vougalter \cite{CucPel, CucPelVou} showed
that in 1D, embedded eigenvalues with positive Krein signature can
exist and not disrupt asymptotic stability.  But an extension of their
approach, based on the Fermi Golden Rule, to higher dimensions remains
a challenge; it may be easier to prove that such states are absent.
In \cite{DeSc}, Demanet \& Schlag gave a numerically assisted proof
for the absence of non zero eigenvalues within the spectral gap for 3D
NLS with a power nonlinearity with
\begin{equation}
  \label{e:desc_bound}
  0.913958905\pm 10^{−8}<\sigma \leq 1
\end{equation}

More recently, Marzuola \& Simpson, \cite{Marzuola:2010p5770}, proved
that the 3D cubic problem has no purely imaginary eigenvalues or
endpoint resonances.  This proof required a modest amount of numerical
assistance to compute:
\begin{itemize}
\item The dimension of the negative subspace of several 1D
  Schr\"odinger operators;
\item The inner products of solutions of several 1D boundary value
  problems.
\end{itemize}
In this work, we apply the approach of \cite{Marzuola:2010p5770}, to
some other equations.

\subsection{Main Results}
Our main results are for the 1D and 3D supercritical NLS equations
with $f(s) = s^\sigma$, and the 3D cubic-quintic nonlinear
Schr\"odinger equation (CQNLS) and for the 1D supercritical NLS,
\begin{equation}
  \label{e:cqnls}
  i \psi_t + \nabla^2\psi + (\abs{\psi}^2 - \gamma \abs{\psi}^4)\psi = 0.
\end{equation}

\begin{thm}
  \label{t:3dnls}
  For 3D NLS with
  \begin{equation}
    \label{e:3dnls_bound}
    0.807425  <\sigma  <1.12092
  \end{equation}
  $JL$, the linearization about the soliton $R= R(\cdot;1)$, has no
  purely imaginary eigenvalues or endpoint resonances.
\end{thm}

\begin{thm}
  \label{t:3dcqnls}
  For 3D CQNLS with
  \begin{equation}
    \label{e:cqnls:bound}
    \gamma < 0.00989115
  \end{equation}
  $JL$, the linearization about the soliton $R= R(\cdot;1)$, has no
  purely imaginary eigenvalues or endpoint resonances.
\end{thm}

For the sake of exploring the limits of our approach, we also prove:
\begin{thm}
  \label{t:1dnls}
  For 1D NLS with
  \begin{equation}
    \label{e:1dnls_bound}
    2.4537956056< \sigma <6.1288520139
  \end{equation}
  $JL$, the linearization about the soliton $R= R(\cdot;1)$, has no
  purely imaginary eigenvalues.
\end{thm}
As previously mentioned, Theorem \ref{t:1dnls} has been rigorously
established for all $\sigma > 2$, \cite{KS1}.

To prove these theorems, we first show that a particular bilinear
form, which would vanish at any purely imaginary eigenvalue or
endpoint resonance, is coercive on an appropriate subspace.  This
coercivity allows us to immediately rule out such states.  To proceed
we must define distorted variants of $L_\pm$ which form the bilinear
form.
\begin{defn}
  \label{d:spec_prop}
  Given $L_\pm$ and a skew adjoint operator $\Lambda$, define the two
  Schr\"odinger operators via the commutator relations:
  \begin{equation}
    \mathcal{L}_\pm  = \frac{1}{2}[L_\pm, \Lambda] = -\nabla^2 + \mathcal{V}_\pm.
  \end{equation}
  For ${\bf z} = (u,v)^T \in L^2 \times L^2$, define the bilinear form
  \begin{equation}
    \begin{split}
      {\mathcal B}({\bf z}, {\bf z}) &= {\mathcal B}_+(u,u) +
      {\mathcal
        B}_-(v,v)\\
      &=\inner{\mathcal{L}_+ u}{u } + \inner{\mathcal{L}_-v}{v}
    \end{split}
  \end{equation}
  The operator $JL$ is said to satisfy the spectral property on the
  subspace $\mathcal{U} \subseteq L^2\times L^2$ if
  \begin{equation}
    \label{e:spec_prop}
    {\mathcal B} ({\bf z}, {\bf z}) \gtrsim
    \int \paren{ \abs{\nabla {\bf z} }^2 + e^{-\abs{\bf x}} \abs{{\bf
          z}}^2 }d {\bf x}
  \end{equation}
\end{defn}
The skew adjoint operator $\Lambda$ and the subspace $\mathcal{U}$
are, at this point, unspecified.  In our work, we use
\begin{equation}
  \label{e:skew_op}
  \Lambda = \frac{d}{2} + {\bf x } \cdot \nabla.
\end{equation}
The potentials, $\calV_\pm$, this induces from $L_\pm$ are
\begin{subequations}
  \label{e:skewed_potentials}
  \begin{align}
    \calV_+ & = \tfrac{1}{2}{\bf x} \cdot \nabla \bracket{f(R^2) + 2
      f'(R^2)R^2}=r(3f'(R^2) + 2 f''(R^2)R^2)RR' \\
    \calV_- & = \tfrac{1}{2}{\bf x} \cdot \nabla \bracket{f(R^2) }= 
    r f'(R^2)RR'
  \end{align}
\end{subequations}
For the 3D problems we consider, these correspond to:
\begin{align*}
f(s) = s^\sigma: & \calV_+  =\sigma
(2\sigma +1) R^{2\sigma-1} R' ,\quad \calV_- = \sigma r
R^{2\sigma -1 } R'\\
f(s) = s - \gamma s^2: & \calV_+ = r(3 R - 10 \gamma R^3)R',\quad \calV_- = r ( R - 2\gamma R^3) R'
\end{align*}
The particulars of this subspace will be discussed in Section
\ref{s:subspace}.  This Spectral Property appeared in the works of
Merle \& Rapha\"el, and Fibich, Merle \& Rapha\"el in their proofs of
the $\log-\log$ blowup of $L^2$ critical NLS, \cite{FMR,MR-Annals}.
It has similarly appeared in Simpson \& Zwiers, in a proof of the
$\log-\log$ blow up for vortex solitons in 2D cubic NLS,
\cite{Simpson:2010p8489}

The role of this Spectral Property was identified by G. Perelman,
\cite{GPer}, and leads directly to:
\begin{thm}
  If the Spectral Property holds for $JL$, then $JL$ has no imaginary
  eigenvalues on $\mathcal{U}$.
\end{thm}
The proof is quite simple and appears in \cite{Marzuola:2010p5770}.
Briefly, one shows by direct computation that $\mathcal{B}$ vanishes
at any eigenstates corresponding to imaginary eigenvalues.  In 3D,
this same analysis will rule out endpoint resonances.

Though all of the cases for which we have results correspond to
orbitally unstable solitons, the results remain of interest.  They can
be used to prove results on constrained asymptotic stability, as in
\cite{Schlag:2009p346}, where one projects away from the linearly
unstable direction of $JL$.  In addition they are intrinsically
interesting as results on the spectrum of Hamiltonian operators.
Finally, these theorems also map out the scope of success for this
approach.

Our paper is organized as follows.  In Section \ref{s:review} we
review some relevant algebraic properties of $JL$, and define the
subspace, $\calU$, that will be used to prove the Spectral Property.
In Section \ref{s:prelim}, we present some preliminary results used in
the proof.  Section\ref{s:results} presents our computations and the
main results.  Some discussion of the limits of this approach are
raised in Section \ref{s:discussion}.

{\bf Acknowledgements:} {The authors wish to thank J.L. Marzuola for
  some helpful comments, D.E. Pelinovsky for suggesting a comparison
  with the Demanet \& Schlag threshold, and C. Sulem for suggesting
  the extension to CQNLS.  This work was supported in part by NSERC.}

\section{Algebraic Structure}
\label{s:review}

As noted in the introduction, $\sigma(JL)$ includes the essential
spectrum, lying on a portion of the imaginary axis, along with some
eigenvalues.  Since the algebraic structure of $JL$ is used in our
proof, we review the discrete spectrum here.

\subsection{Discrete Spectrum}
\label{s:discrete}

Given that \eqref{e:nls} supports a solitary wave solution, we readily
observe:
\begin{thm}[Kernel]
  $JL$ has a kernel of algebraic multiplicity at least $2d + 2$:
  \begin{gather*}
    JL \begin{pmatrix} 0 \\ R \end{pmatrix} = 0, \quad
    JL \begin{pmatrix}
      \partial_{x_j}  R \\ 0 \end{pmatrix} = 0\\
    JL \begin{pmatrix} \partial_\omega R \\ 0 \end{pmatrix} =
    - \begin{pmatrix} 0 \\ R \end{pmatrix}\quad JL \begin{pmatrix}0\\
      x_j R \end{pmatrix} = -2 \begin{pmatrix} \partial_{x_j} R \\
      0\end{pmatrix}
  \end{gather*}
  for $j=1,\ldots d$.
\end{thm}
For power nonlinearities, we can, and do, use, that
\begin{equation}
  \label{e:analytic_deriv}
  2 \partial_\omega|_{\omega = 1} R = \tfrac{1}{\sigma} R+ {\bf x} \cdot \grad R.
\end{equation}

In addition to these elements, there are two additional eigenstates
that can reside on the imaginary axis, the origin, or the real axis,
depending on the slope condition, \eqref{e:slope_condition}.  For
orbitally unstable problems we have:
\begin{thm}[Off Axis Eigenvalues]
  There exists a pair of real, nonzero, eigenvalues located at $\pm
  \mu_\star$.  Corresponding to $\mu_\star>0$, is the eigenstate
  $\boldsymbol{\phi} = (\phi_1, \phi_2)^T$.  For $d=1$, the $\phi_j$
  are even functions, and for $d>1$, the $\phi_j$ are radially
  symmetric.
\end{thm}
These off axis eigenvalues appear in Figure \ref{f:spec}.  Their
existence follows form the results of \cite{Grillakis:1987p115}.
Alternatively, Schlag presented a continuity argument in
\cite{Schlag:2009p346} that relies on the knowledge of the two
additional eigenstates that appear at the origin in the critical case,
$\sigma d = 2$.
 
It is well known that for $f(s) = s^\sigma$, the slope is negative
when $\sigma d > 2$.  Due to scaling, this holds for all $\omega>0$.
For CQNLS, some solitons are stable and others are unstable, depending
on $\gamma$ and the soliton parameter $\omega$.  In this work, we
shall fix $\omega = 1$ for which solitons are known to exist provided
$\gamma < 3/16$, \cite{sulem1999nse}.  We now determine the range of
$0\leq \gamma< 3/16$ for which the solitons are
unstable.\footnote{Alternatively, we could set $\gamma=1$, and vary
  $\omega$.}

At $\gamma = 0$, this corresponds to the 3D cubic problem, which we
know to be unstable.  By continuity, we expect there is some open set
of $\gamma$ near zero for which the solitons are orbitally unstable.
To identify the threshold value of $\gamma$, we numerically compute
\eqref{e:slope_condition}, and find that it changes sign at
\begin{equation}
  \gamma_\star \approx 0.0255453.
\end{equation}
The slope condition values as a function of $\gamma$ are plotted in
Figure \ref{f:slope_condition}.  See Appendix \ref{s:numerics} for
details of this computation.

\begin{figure}
  \includegraphics[width=2.4in]{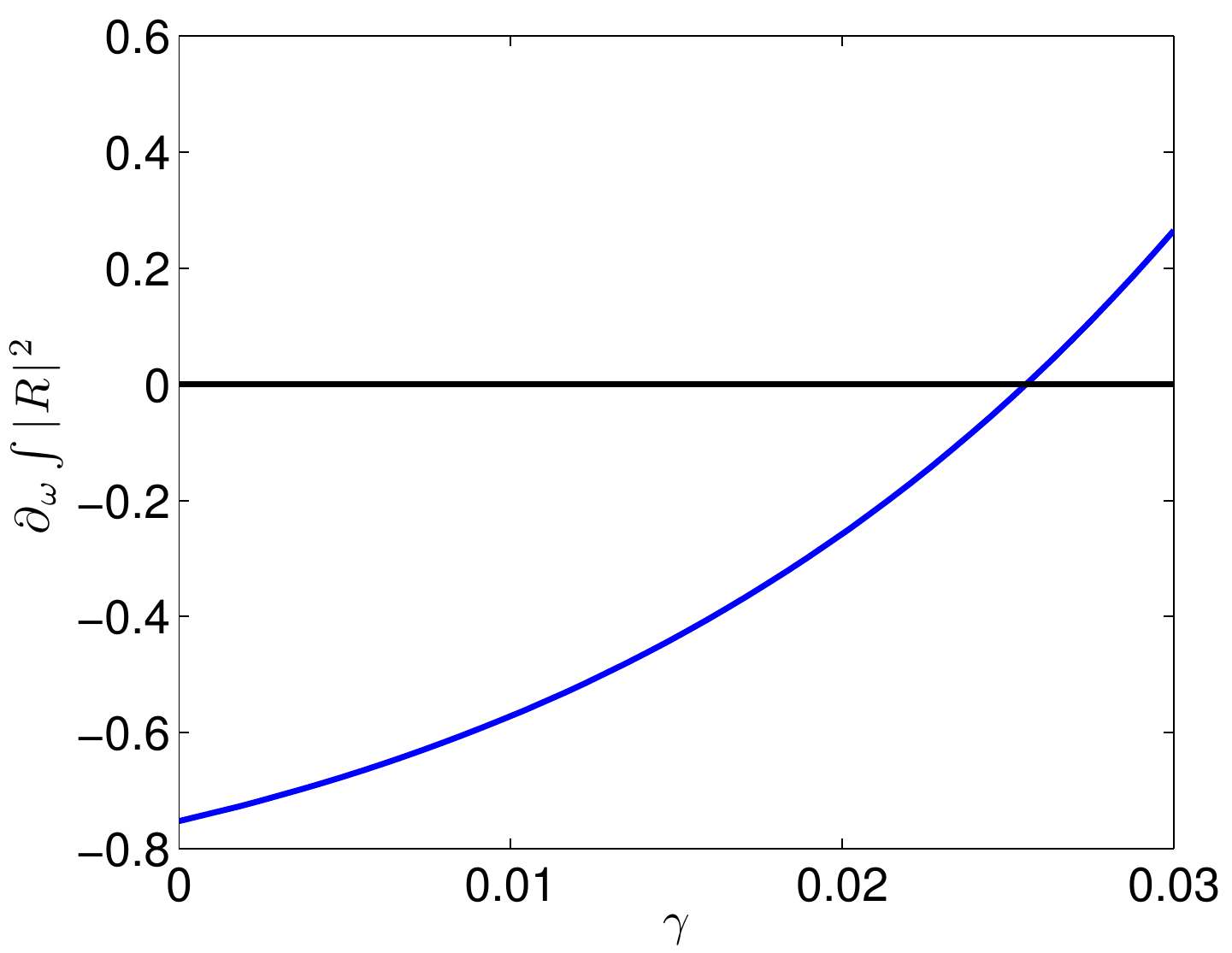}
  \caption{The numerically computed \eqref{e:slope_condition} at
    $\omega=1$ changes from negative to positive at the critical
    $\gamma_\star \approx 0.0255453$. }
  \label{f:slope_condition}
\end{figure}

\subsection{Subspaces}
\label{s:subspace}

As discussed in Section 2.2 of \cite{Marzuola:2010p5770}, the subspace
$\mathcal{U}$ for which we wish to establish the Spectral Property of
Definition \ref{d:spec_prop} must satisfy two properties:
\begin{itemize}
\item If they exist, eigenstates of purely imaginary eigenvalues must
  reside in $\mathcal{U}$;
\item $\mathcal{U}$ must be orthogonal to the negative subspaces of
  $\calL_\pm$.
\end{itemize}
Obviously, if both properties are satisfied, then there are no purely
imaginary eigenvalues.  But it may be possible to construct a
$\mathcal{U}$ that is orthogonal to the eigenstates of such
eigenvalues, and on which the operators are positive.

A practical subspace makes use of the eigenstates of the adjoint
operator, $(JL)^\ast$.  These are closely related to the elements of
$JL$.  Given $\boldsymbol{\psi}= (\psi_1, \psi_2)^T$, an eigenstate of
$JL$ with eigenvalue $\lambda$,
\begin{equation}
  \label{e:adjoint_eigenvalues}
  (JL)^\ast \begin{pmatrix} \psi_2\\ \psi_1 \end{pmatrix} =
  -\lambda \begin{pmatrix} \psi_2\\ \psi_1 \end{pmatrix}
\end{equation}
Consequently,
\begin{prop}
  \label{p:orhto}
  Let $\boldsymbol{\psi}= (\psi_1, \psi_2)^T$ be an eigenstate
  corresponding to a purely imaginary eigenvalue of $JL$.
  \begin{itemize}
  \item Any element of $\ker_g((JL)^\ast)$ is orthogonal to
    $\boldsymbol{\psi}$;
  \item The unstable eigenstate $\boldsymbol{\phi}$ of $JL$ satisfies
    the relations:
    \begin{equation*}
      \inner{\psi_1}{\phi_2} = 0,\quad \inner{\psi_2}{\phi_1}  = 0.
    \end{equation*}
  \end{itemize}
\end{prop}
See Section 2.2 of \cite{Marzuola:2010p5770} for a proof.  Using this,
we define $\mathcal{U}$ to be the orthogonal complement of
\begin{equation}
  \label{e:U_perp}
  \spn \set{\begin{pmatrix} R \\ 0\end{pmatrix}, \begin{pmatrix}
      0\\\partial_\omega R \end{pmatrix}, \begin{pmatrix}
      0   \\ \partial_{x_j}R \end{pmatrix}, \begin{pmatrix} {x_j}R \\
      0\end{pmatrix},\begin{pmatrix} \phi_2 \\ 0\end{pmatrix}, \begin{pmatrix} 0\\ \phi_1\end{pmatrix} }
\end{equation}
Indeed,
\begin{thm}
  If the Spectral Property can be established on $\calU$ given by the
  orthogonal complement to \eqref{e:U_perp}, or any subspace
  containing $\calU$, $JL$ will not have any purely imaginary
  eigenvalues on $L^2 \times L^2$.
\end{thm}
We succeed in proving:
\begin{thm}
  \label{t:spec_prop_ortho}
  Assume ${\bf z} = (f,g)^T \in L^2\times L^2$ satisfies the
  orthogonality conditions
  \begin{equation}
    \label{e:orhto}
    \inner{f}{R}=\inner{f}{\phi_2} = \inner{g}{\partial_\omega|_{\omega=1} R}
    = \inner{g}{\phi_1} = \inner{f}{{\bf x} R}=0
  \end{equation}
  Then the Spectral Property holds for:
  \begin{description}
  \item[3D NLS] Provided
    \[
    0.807425 < \sigma < 1.12092;
    \]
  \item[3D CQNLS] Provided
    \[
    \gamma < 0.00989115;
    \]
  \item[1D NLS] Provided
    \[
    2.4537956056 < \sigma < 6.1288520139.
    \]
  \end{description}
\end{thm}
For 1D and 3D NLS with power nonlinearities, we use
\eqref{e:analytic_deriv}.

\section{Preliminaries}
\label{s:prelim}

We shall prove the coercivity of $\mathcal{B}$ in three steps, closely
following \cite{FMR, Marzuola:2010p5770}, and omitting some details.
\begin{itemize}
\item In 1D, we decompose into even and odd functions and study
  \begin{equation}
    \label{e:1d_decomp}
    \begin{split}
      \calB_+(f,f) + \calB_-(g,g) &= \calB_+^{(e)}(f ^{(e)},f ^{(e)})
      +
      \calB_+^{(o)}(f ^{(o)},f ^{(o)}) \\
      &\quad + \calB_-^{(e)}(g^{(e)},g^{(e)})+
      \calB_-^{(o)}(g^{(o)},g^{(o)})
    \end{split}
  \end{equation}
  We then show that on $\mathcal{U}$, each of the four forms,
  $\calB_\pm^{(e/o)}$, is positive.

  Similarly, in 3D, we decompose into spherical harmonics,
  \begin{equation}
    \label{e:3d_decomp}
    \begin{split}
      \calB_+(f,f) + \calB_-(g,g) & = \sum_{k=0}^\infty
      \sum_{l=0}^{L_k}
      \calB_+^{(k)}(f^{(k,l)},f^{(k,l)}) \\
      &\quad + \sum_{k=0}^\infty \sum_{l=0}^{L_k}
      \calB_-^{(k)}(g^{(k,l)},g^{(k,l)})
    \end{split}
  \end{equation}
  where $f^{(k,l)}$ and $g^{(k,l)}$ are the components of $f$ and $g$
  in harmonic $k$ with angular component $l$.  As $\calL_\pm$ are
  radially symmetric operators, they are independent of $l$.  Here
  too, we shall prove positivity of each $\calB_\pm^{(k)}$ on $\calU$.

\item For each of these bilinear forms, we shall determine the
  codimension of a subspace on which they are positive, called its
  {\it index}.  This is problem specific and can vary with the
  dimension and choice of nonlinearity.

\item Lastly, we shall show that the $L^2$ orthogonality of $f$ and
  $g$ to $\mathcal{U}^\perp$ induces orthogonality, {\it with respect
    to to the bilinear form}, to the negative subspaces of the
  form. See Section \ref{s:nls3d_proof} for an example.


\end{itemize}

\subsection{Indexes and Eigenvalues}
\label{s:idx}

The index of a bilinear form $B$ on a vector space $V$ is defined as
\begin{equation}
  \begin{split}
    \ind_V(B) \equiv \max \{ k \in \mathbb{N}\mid & \text{there exists
      a
      subspace $\tilde V\subseteq V$} \\
    & \text{of codimension $k$ such that $B|_{\tilde V}$ is positive}
    \}
  \end{split}
\end{equation}
To compute the indexes, we rely on the following propositions:

\begin{prop}
  \label{p:1d_index}

  Let $U^{(e)}$ and $U^{(o)}$ solve the initial value problems
  \begin{align*}
    L U^{(e)} &= - \frac{d^2}{dx^2} u^{(e)} + V(x) U^{(e)} = 0, \quad
    U^{(e)}(0)=1,\quad \frac{d}{dx} U^{(e)}(0)=0\\
    L U^{(o)} &= - \frac{d^2}{dx^2} u^{(o)} + V(x) U^{(o)} = 0, \quad
    U^{(o)}(0)=0,\quad \frac{d}{dx} U^{(o)}(0)=1
  \end{align*}
  where $V$ is even, sufficiently smooth, and $\abs{V(x)} \lesssim
  e^{-\kappa\abs{x}}$, $\kappa>0$, and let
  \begin{equation*}
    B(\cdot, \cdot) \equiv \inner{L \cdot }{\cdot}.
  \end{equation*}
  Then the number of zeros of $U^{(e)}$ and $U^{(o)}$ is finite, and
  \begin{align*}
    \ind_{H^1_e} (B) &= \text{number of positive roots of $U^{(e)}$}
    \\
    &=\text{number of negative eigenvalues in $H^1_e$}\\
    \ind_{H^1_o} (B) &= \text{number of positive roots of
      $U^{(o)}$}\\
    &=\text{number of negative eigenvalues in $H^1_o$}
  \end{align*}
\end{prop}
$H^1_e$ and $H^1_o$ are the even and odd subspaces of $H^1$.

\begin{prop}
  \label{p:3d_index}

  Let $U^{(k)}$ solve the initial value problem
  \begin{gather*}
    L^{(k)} U^{(k)} = \bracket{-\frac{d^2}{dr^2} -
      \frac{2}{r}\frac{d}{dr} + V(r) +
      \frac{k^2}{r^2}} U^{(k)}=0\\
    \lim_{r\to 0} r^{-k} U^{(k)}(r) = 1, \quad \lim_{r\to
      0}\frac{d}{dr}\paren{ r^{-k} U^{(k)}(r)} = 0,
  \end{gather*}
  for $k = 0, 1,2,3, \ldots$ where $V$ is sufficiently smooth, and
  $\abs{V(r)} \lesssim e^{-\kappa\abs{r}}$,$\kappa>0$, and let
  \begin{equation*}
    B^{(k)}(\cdot, \cdot) \equiv \inner{L^{(k)} \cdot }{\cdot}.
  \end{equation*}
  Then the number of zeros of $U^{(k)}$ is finite, and
  \begin{align*}
    \ind_{H^1_{\rad}} (B^{(0)}) &= \text{number of positive roots of
      $U^{(0)}$}\\
    &=\text{number of negative eigenvalues in $H^1_\rad$}\\
    \ind_{H^1_{\rad+}}(B^{(k)}) & = \text{number of positive roots of
      $U^{(k)}$, for $k = 1, 2, \ldots$}\\
    &=\text{number of negative eigenvalues in $H^1_{\rad+}$}
  \end{align*}
\end{prop}
The space $H^1_{\rad+}$ is the subspace of radially symmetric $H^1$
functions for which
\begin{equation*}
  \int \abs{f}^2 \abs{{\bf x}}^{-2} d {\bf x} < \infty.
\end{equation*}

The indexes in the different harmonics have an invaluable monotonicity
property that permits us to restrict our attention to a finite number
of harmonics:
\begin{cor}
  \label{c:mono_index}
  Fixing the potential $V(r)$, the bilinear forms $B^{(k)}$ satisfy
  \begin{equation*}
    \ind(B^{(k+1)})\leq \ind(B^{(k)})
  \end{equation*}
\end{cor}

The last result that we shall make use of in computing the index of a
bilinear form is that it is stable to perturbation by a sufficiently
localized potential:
\begin{prop}
  \label{p:idx_stability}
  Fixing an operator $L$ from Propositions \ref{p:1d_index} or
  \ref{p:3d_index}, or Corollary \ref{c:mono_index}, there exists
  $\delta_0$ sufficiently small such that the bilinear form, $\bar B$,
  induced by the perturbed operator
  \begin{equation*}
    \bar L \equiv L - \delta_0 e^{-\abs{\bf x}}
  \end{equation*}
  satisfies
  \begin{equation*}
    \ind(\bar B) = \ind(B)
  \end{equation*}
\end{prop}

The proofs of the preceding results are given in
\cite{Marzuola:2010p5770}, and references therein.  As one might
suspect, this is closely related to Sturm oscillation theory.

\subsection{Boundary Value Problems}
\label{s:bvp_ip}

In the course of proving the spectral property, we will need to solve
a series of problems of the form
\begin{equation*}
  \calL u = f
\end{equation*}
for $u$ where $\calL$ is one of the above operators and $f$ is a
localized.  Though the solutions of these problems are found
numerically, the invertibility of the operators can be rigorously
justified, up to the index computation.  For the index, we continue to
rely on numerics.

\begin{prop}
  \label{p:1d_inverse}
  Let $f$ be a smooth function, with even or odd symmetry, satisfying
  the bound $\abs{f(x)}\lesssim e^{-\kappa\abs{x}}$ for some
  $\kappa>0$.  Then there exists a unique solution
  \begin{equation*}
    u \in L^\infty(\R)\cap C^2(\R)
  \end{equation*}
  to
  \begin{equation*}
    \calL u = f.
  \end{equation*}
  where $\calL$ is one of $\calL_\pm$.  $u$ possesses the same
  symmetry as $f$.
\end{prop}

\begin{prop}
  \label{p:3d_inverse}
  Let $f$ be a smooth, radially symmetric, function satisfying the
  bound $\abs{f(r)}\lesssim e^{-\kappa r}$ for some $\kappa > 0$.
  Then there exists a unique solution
  \begin{equation*}
    (1+r^{k+1})u \in L^\infty(\R^3)\cap C^2(\R^3)
  \end{equation*}
  to
  \begin{equation*}
    \calL u = f
  \end{equation*}
  where $\calL$ is one of $\calL^{(k)}_\pm$.  $u$ is radially
  symmetric.
\end{prop}

While the solutions of the 3D problems decay $\propto r^{-1-k}$, the
solutions of the 1D problems are merely bounded.  Knowledge of the
index of each of the operators is necessary in proving the uniqueness
of the solutions.  Additionally the invertibility of the operators is
stable to perturbation:

\begin{prop}
  \label{p:inverse_stability}
  For sufficiently small $\delta_0>0$, the results of Propositions
  \ref{p:1d_inverse} and \ref{p:3d_inverse} will also apply to
  \begin{equation*}
    \bar\calL  = \calL - \delta_0 e^{-\abs{\bf x}}
  \end{equation*}
\end{prop}

\section{Results}
\label{s:results}

We are now in the position to present our results and prove Theorem
\ref{t:spec_prop_ortho} for each NLS problem.  From this, we can
conclude Theorems \ref{t:3dnls}, \ref{t:3dcqnls}, and \ref{t:1dnls}.
Throughout, we shall make use of the orthogonality conditions
\eqref{e:orhto}.  See Appendix \ref{s:numerics} for details of our
numerical methods.

\subsection{3D Supercritical NLS}
\label{s:3dnls}

\subsubsection{Indexes}
\label{s:3dnls_idx}
\begin{prop}
  \label{p:nls3d_idx}
  For $.75\leq\sigma\leq 1.5$,
  \begin{gather*}
    \ind(\calB_+^{(0)})= \ind(\calB_-^{(0)})= \ind (\calB_+^{(1)}) = 1\\
    \ind(\calB_+^{(2)}) = \ind(\calB_-^{(1)}) = 0
  \end{gather*}
\end{prop}
\begin{proof}
  Examining Figures \ref{f:nls3d_idx}, we see that for six computed
  values of $.75\leq\sigma\leq 1.5$, inclusive, we have the
  corresponding number of zero crossings and apply Proposition
  \ref{p:3d_index}.  We argue by continuity that this should hold at
  all points in the interval.  These were computed using the approach
  described in Appendix \ref{s:idx_numerics}.
\end{proof}
Corollary \ref{c:mono_index} ensures positivity of the higher harmonic
forms.

\begin{figure}
  \subfigure[Harmonic
  $k=0$]{\includegraphics[width=2.4in]{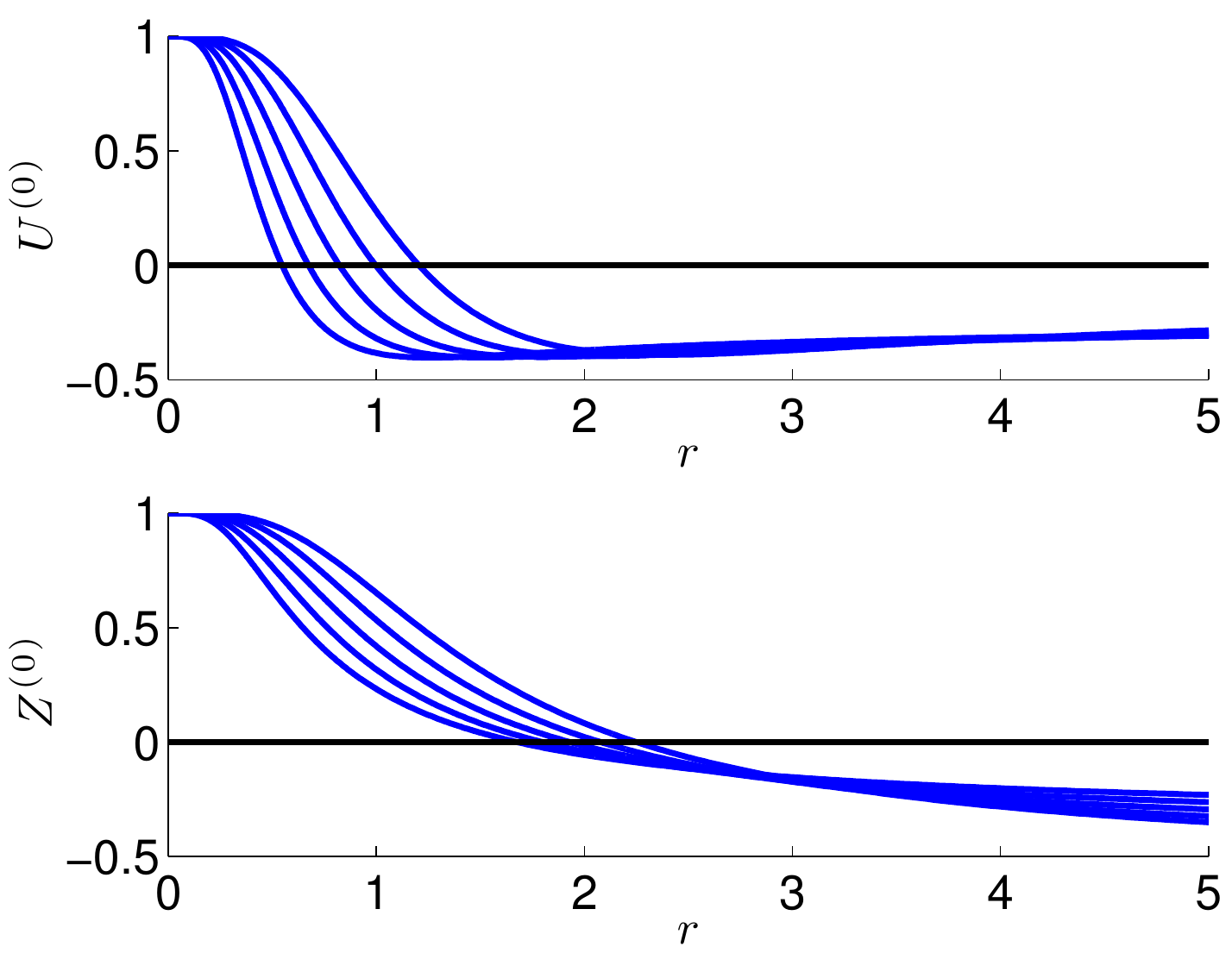}}
  \subfigure[Harmonic
  $k=1$]{\includegraphics[width=2.4in]{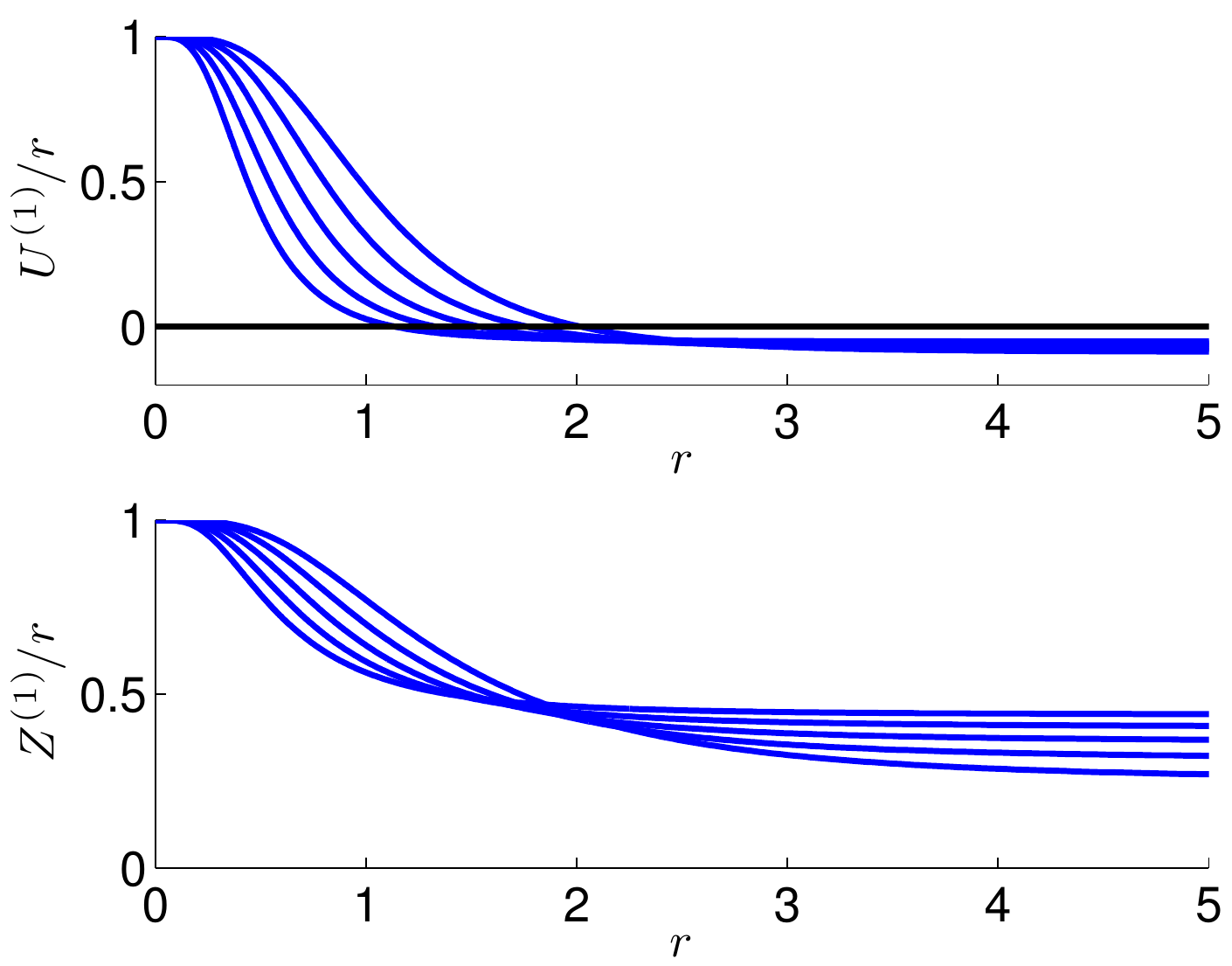}}

  \subfigure[Harmonic
  $k=2$]{\includegraphics[width=2.4in]{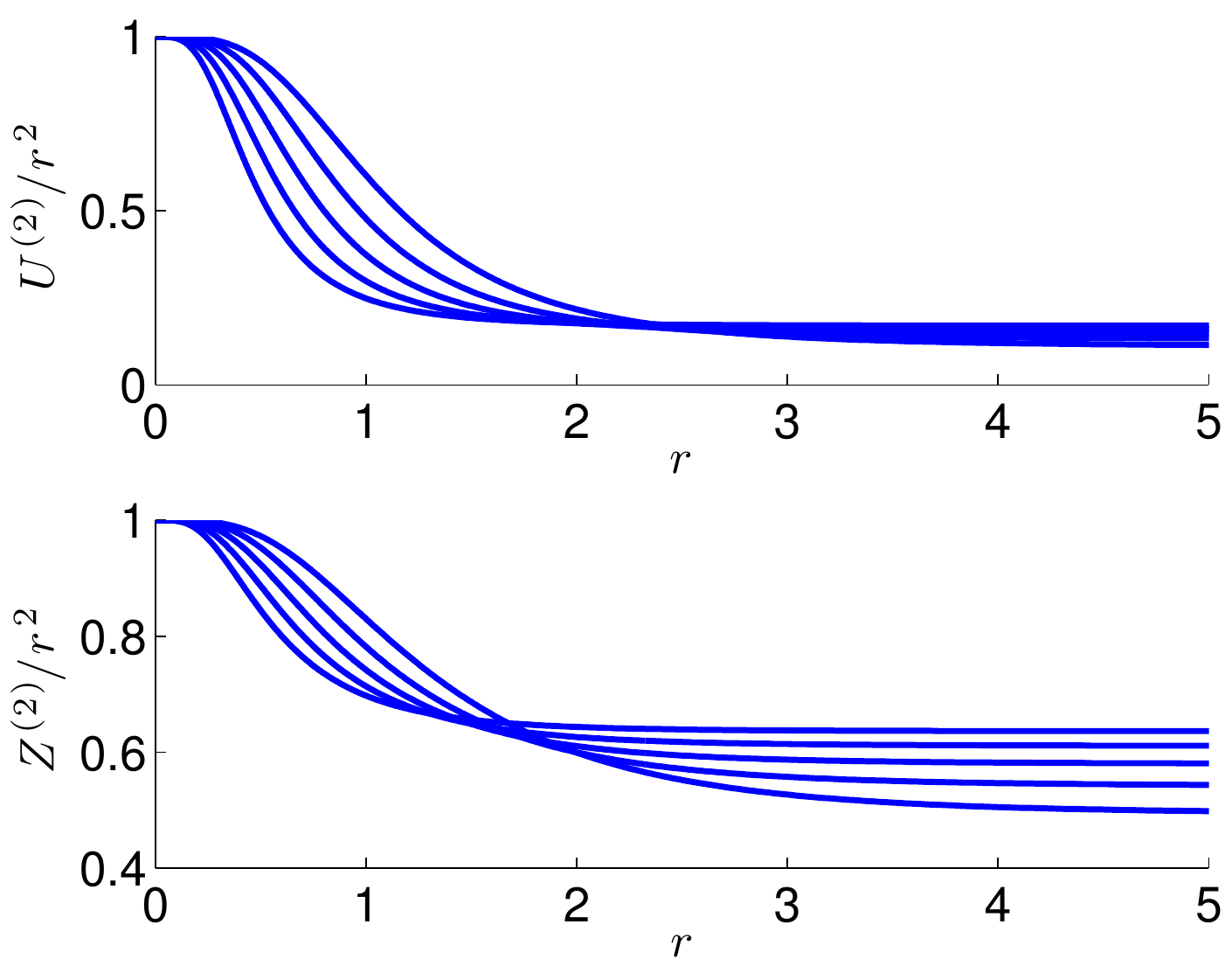}}
  \caption{Index functions for the 3D supercritical NLS equation
    computed at six values of $.75\leq\sigma\leq 1.5$, inclusive.}
  \label{f:nls3d_idx}
\end{figure}

\subsubsection{Inner Products}

\begin{prop}
  \label{p:nls3d_k0}
  Let $U_1^{(0)}$ and $U_2^{(0)}$, solve
  \begin{subequations}
    \begin{align}
      \calL_+^{(0)} U_1^{(0)} = R \\
      \calL_+^{(0)} U_2^{(0)} = \phi_2
    \end{align}
  \end{subequations}
  And define:
  \begin{subequations}
    \label{e:nls3d_k0_defs}
    \begin{align}
      K_1^{(0)} &\equiv  \inner{\calL_+^{(0)} U_1^{(0)}}{U_1^{(0)}},\\
      K_2^{(0)} &\equiv  \inner{\calL_+^{(0)} U_2^{(0)}}{U_2^{(0)}}, \\
      K_3^{(0)} &\equiv \inner{\calL_+^{(0)} U_1^{(0)}}{U_2^{(0)}}.
    \end{align}
  \end{subequations}
  The $K_j^{(0)}$ have the values indicated in Figure
  \ref{f:nls3d_k_0_var}.  Moreover, for $.8 \leq\sigma \leq 1.2$,
  \begin{equation}
    \label{e:nls3d_k0_prod2_def}
    (K_1^{(0)}  K_2^{(0)} - (K_3^{(0)})^2)/K_2^{(0)}<0
  \end{equation}
  as pictured in Figure \ref{f:nls3d_kprod2_0_var}.

\end{prop}
\begin{proof}
  Proposition \ref{p:3d_inverse} ensures that the solutions exist.
  The boundary value problems and inner products were computed using
  the techniques described in Appendix \ref{s:numerics}.  The
  \eqref{e:nls3d_k0_defs} and \eqref{e:nls3d_k0_prod2_def} were then
  computed at 41 uniformly spaced values of $\sigma$ between $.8$ and
  $1.2$.
\end{proof}

\begin{figure}
  \subfigure[]{\includegraphics[width=2.4in]{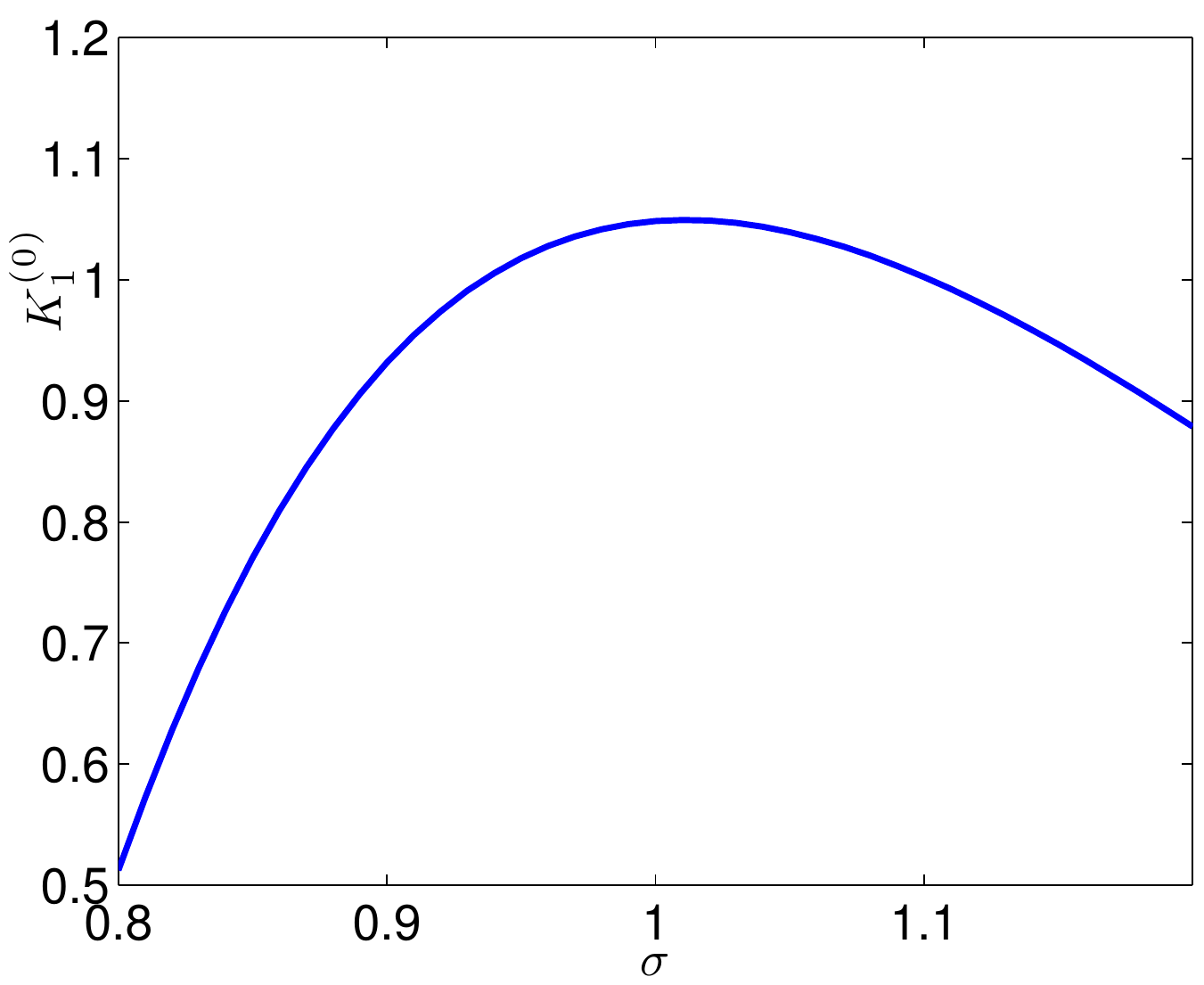}}
  \subfigure[]{\includegraphics[width=2.4in]{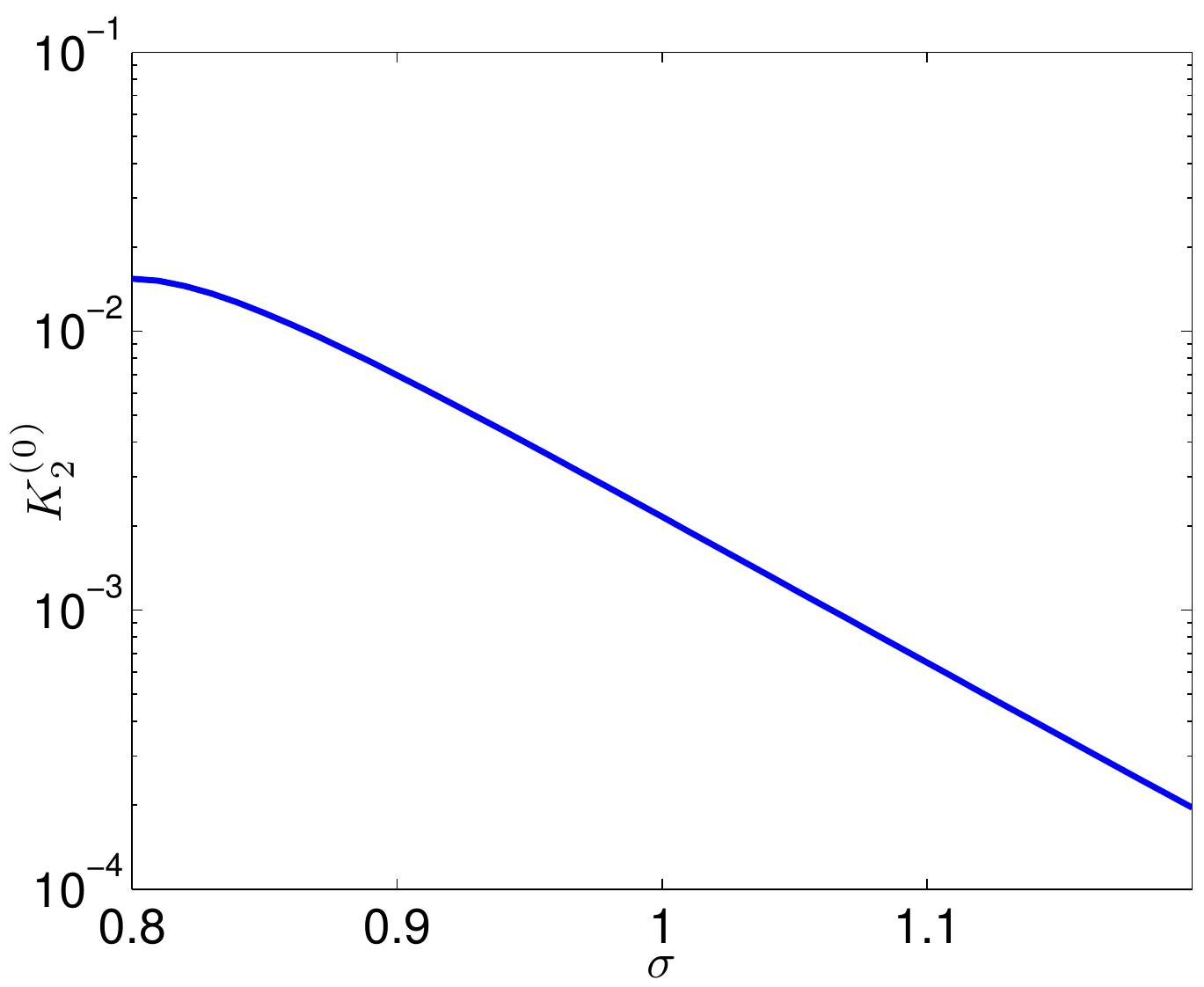}}
  \subfigure[]{\includegraphics[width=2.4in]{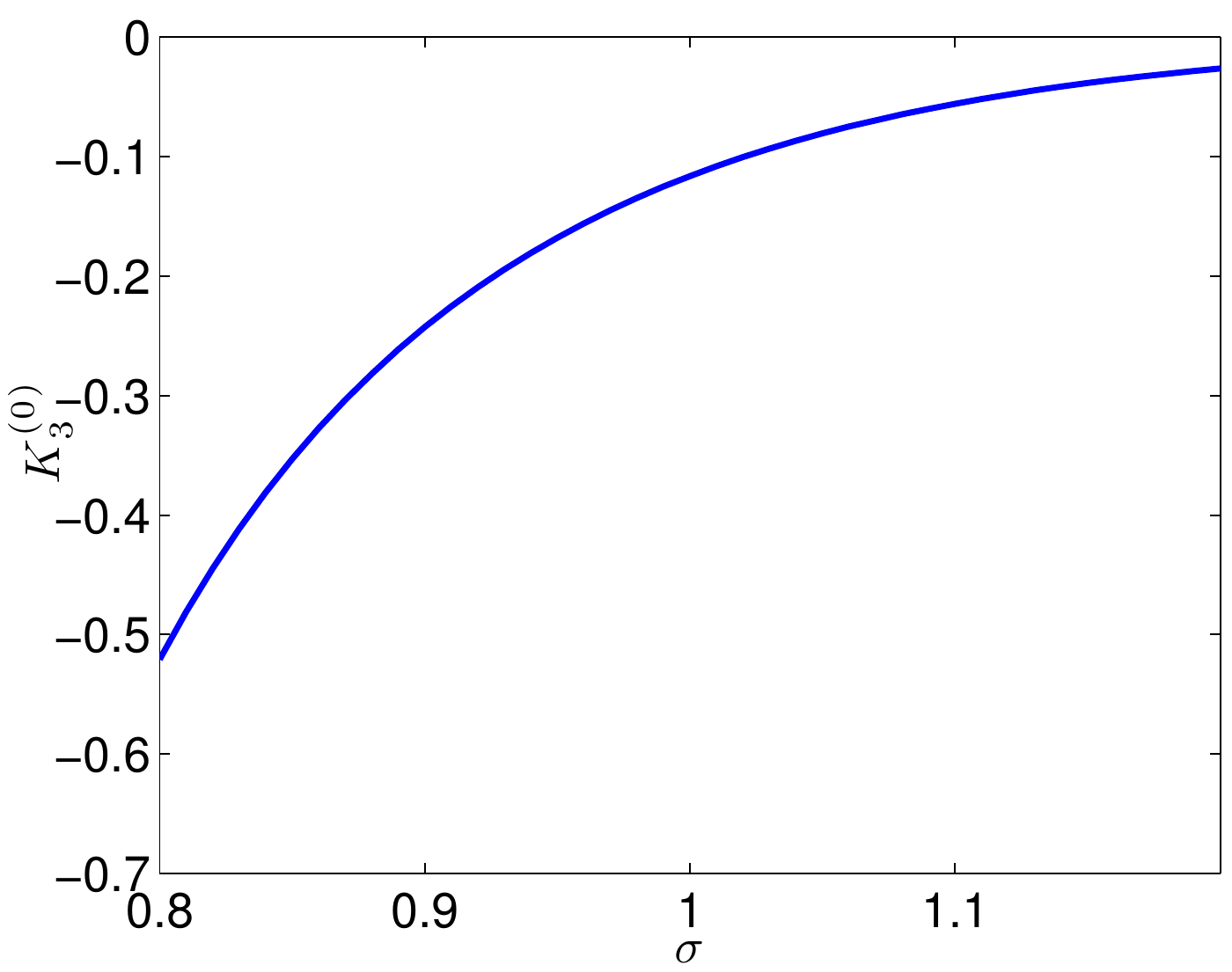}}
  \caption{$K_j^{(0)}$ as functions of $\sigma$ for 3D supercritical
    NLS.}
  \label{f:nls3d_k_0_var}
\end{figure}

\begin{figure}
  \includegraphics[width=2.4in]{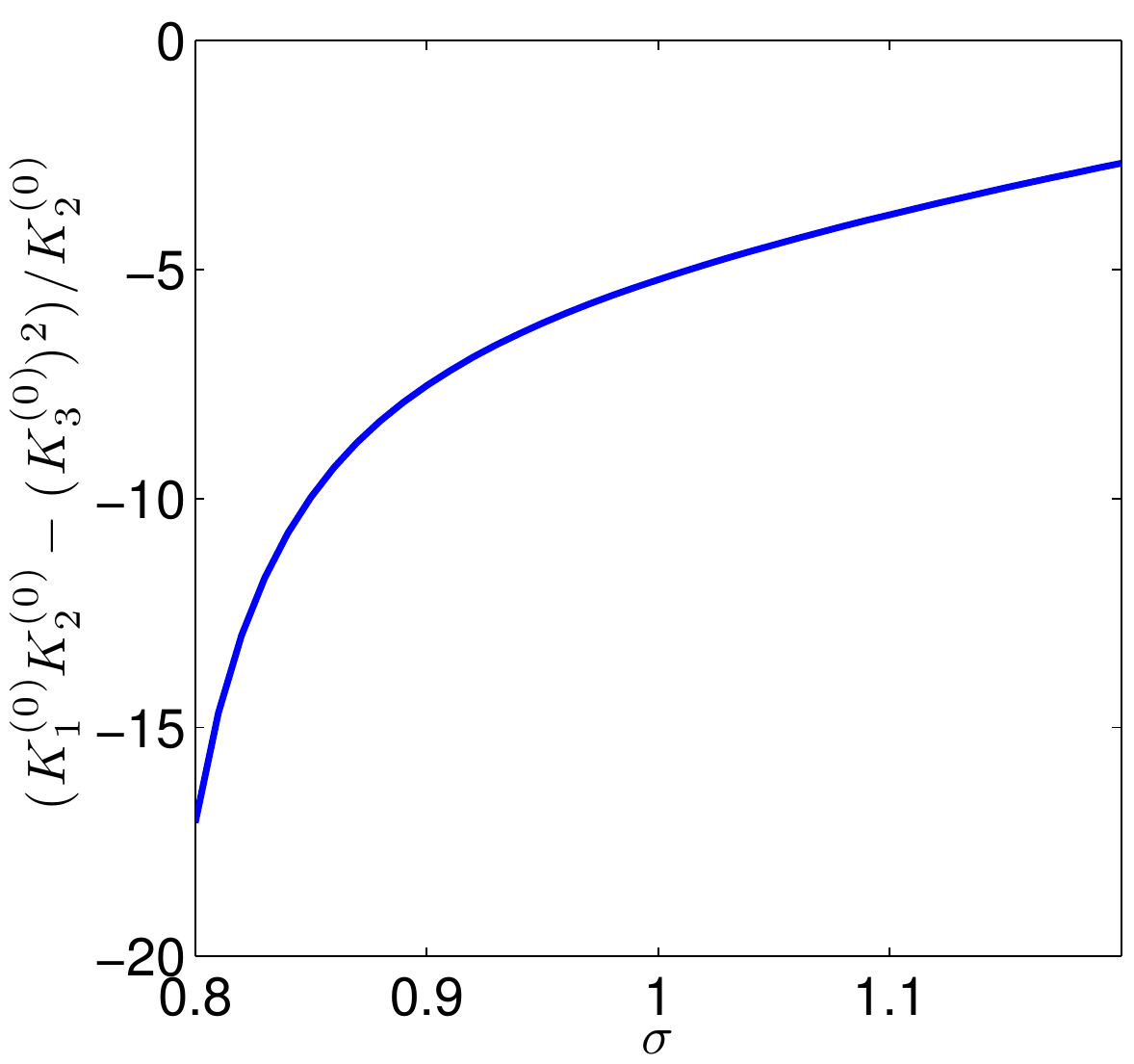}
  \caption{ \eqref{e:nls3d_k0_prod2_def} as a function of $\sigma$ for
    3D supercritical NLS.}
  \label{f:nls3d_kprod2_0_var}
\end{figure}

\begin{prop}
  \label{p:nls3d_j0}
  Let $Z_1^{(0)}$ and $Z_2^{(0)}$, solve
  \begin{subequations}
    \begin{align}
      \calL_-^{(0)} Z_1^{(0)} = \tfrac{1}{\sigma}R + r R' \\
      \calL_-^{(0)} Z_2^{(0)} = \phi_1
    \end{align}
  \end{subequations}
  and define:
  \begin{subequations}
    \label{e:nls3d_j0_defs}
    \begin{align}
      J_1^{(0)} &\equiv  \inner{\calL_-^{(0)} Z_1^{(0)}}{Z_1^{(0)}},\\
      J_2^{(0)} &\equiv  \inner{\calL_-^{(0)} Z_2^{(0)}}{Z_2^{(0)}}, \\
      J_3^{(0)} &\equiv \inner{\calL_-^{(0)} Z_1^{(0)}}{Z_2^{(0)}}.
    \end{align}
  \end{subequations}
  The $J_j^{(0)}$ have the values indicated in Figure
  \ref{f:nls3d_j_0_var}.  Moreover there exist $\sigma_1>.8$ and
  $\sigma_2>.8$ such that
  \begin{gather}
    J_1^{(0)} <0,\quad \text{for $\sigma_1 < \sigma< 1.2$}\\
    \label{e:nls3d_j0_prod2_def}
    (J_1^{(0)} J_2^{(0)} - (J_3^{(0)})^2)/J_2^{(0)}<0,\quad \text{for
      $\sigma_2 < \sigma< 1.2$}
  \end{gather}
  as pictured in Figure \ref{f:nls3d_jprod2_0_var}, where
  \begin{subequations}
    \begin{align}
      \label{e:nls3d_sigma1}
      \sigma_1 &= 0.807699 \\
      \label{e:nls3d_sigma2}
      \sigma_2 &= 0.807425
    \end{align}
  \end{subequations}

\end{prop}
\begin{proof}
  Like in the proof of Proposition \ref{p:nls3d_k0}, we numerically
  solve the boundary value problems at 41 uniformly spaced values of
  $\sigma$ between $.8$ and $1.2$.  Using the root finding approach,
  discussed in Appendix \ref{s:roots}, we found the zero crossings,
  $\sigma_1$ and $\sigma_2$, of $J_1^{(0)}$ and
  \eqref{e:nls3d_jprod2_0_var} to be at the indicated values.
\end{proof}

\begin{figure}
  \subfigure[]{\includegraphics[width=2.4in]{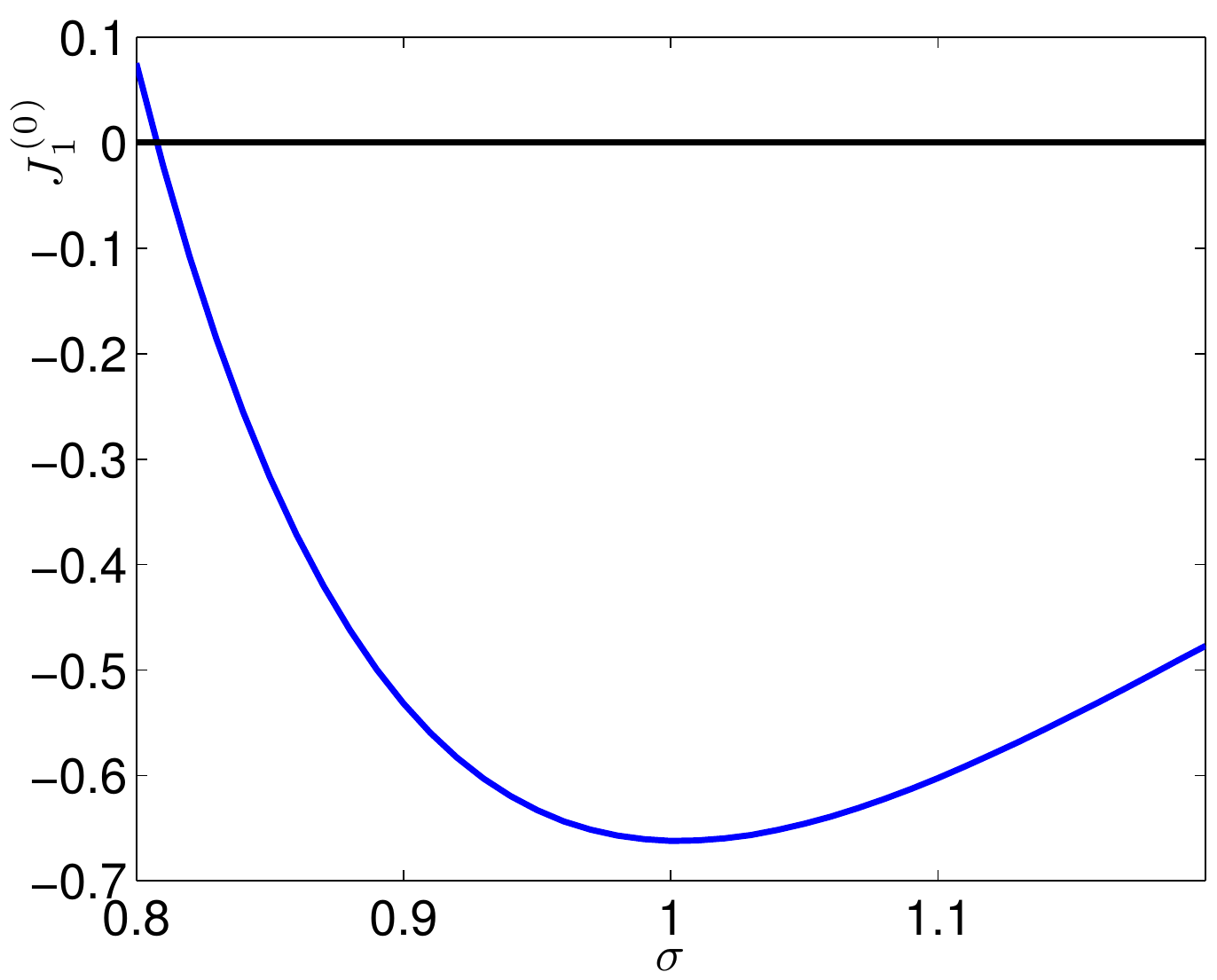}}
  \subfigure[]{\includegraphics[width=2.4in]{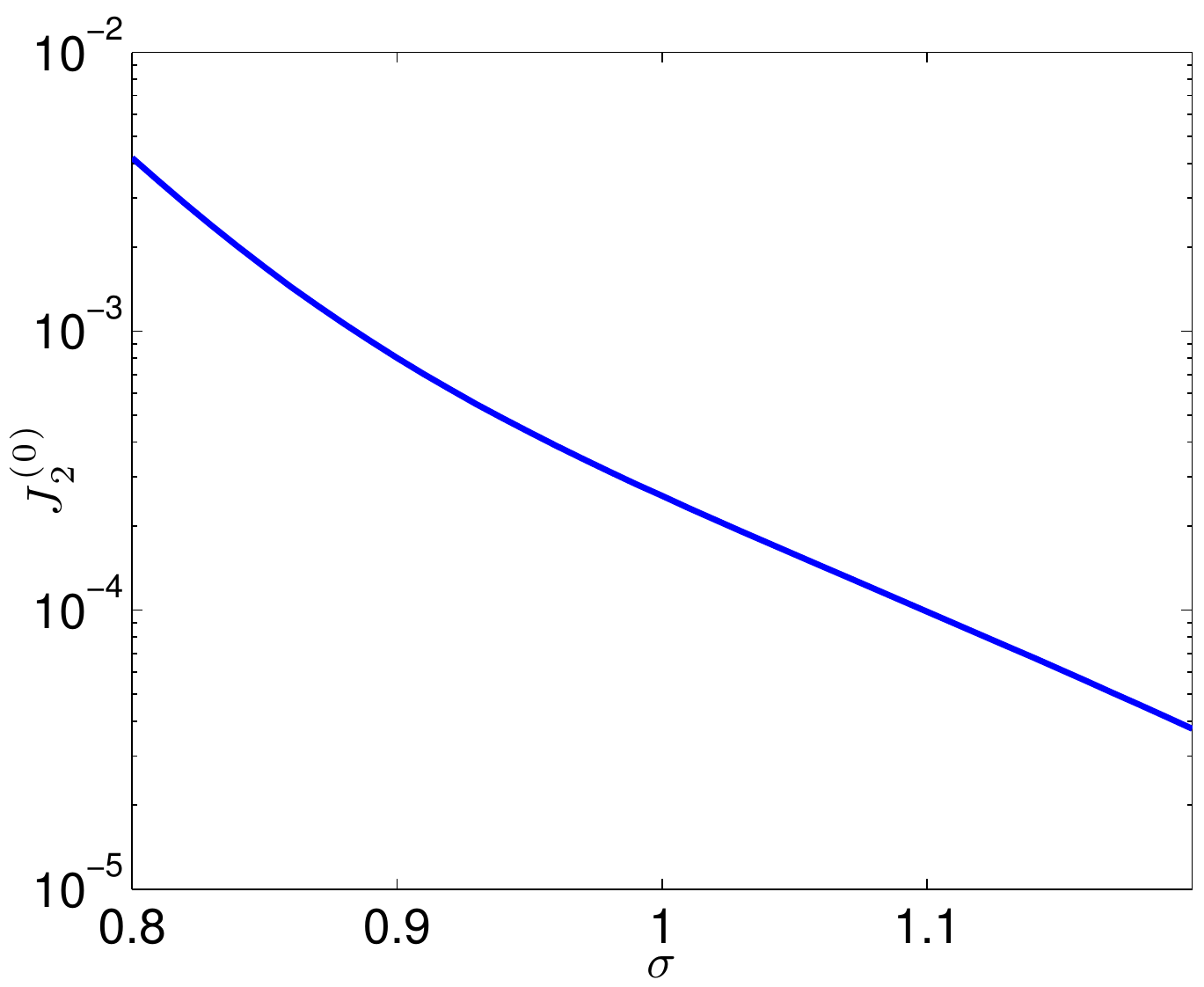}}
  \subfigure[]{\includegraphics[width=2.4in]{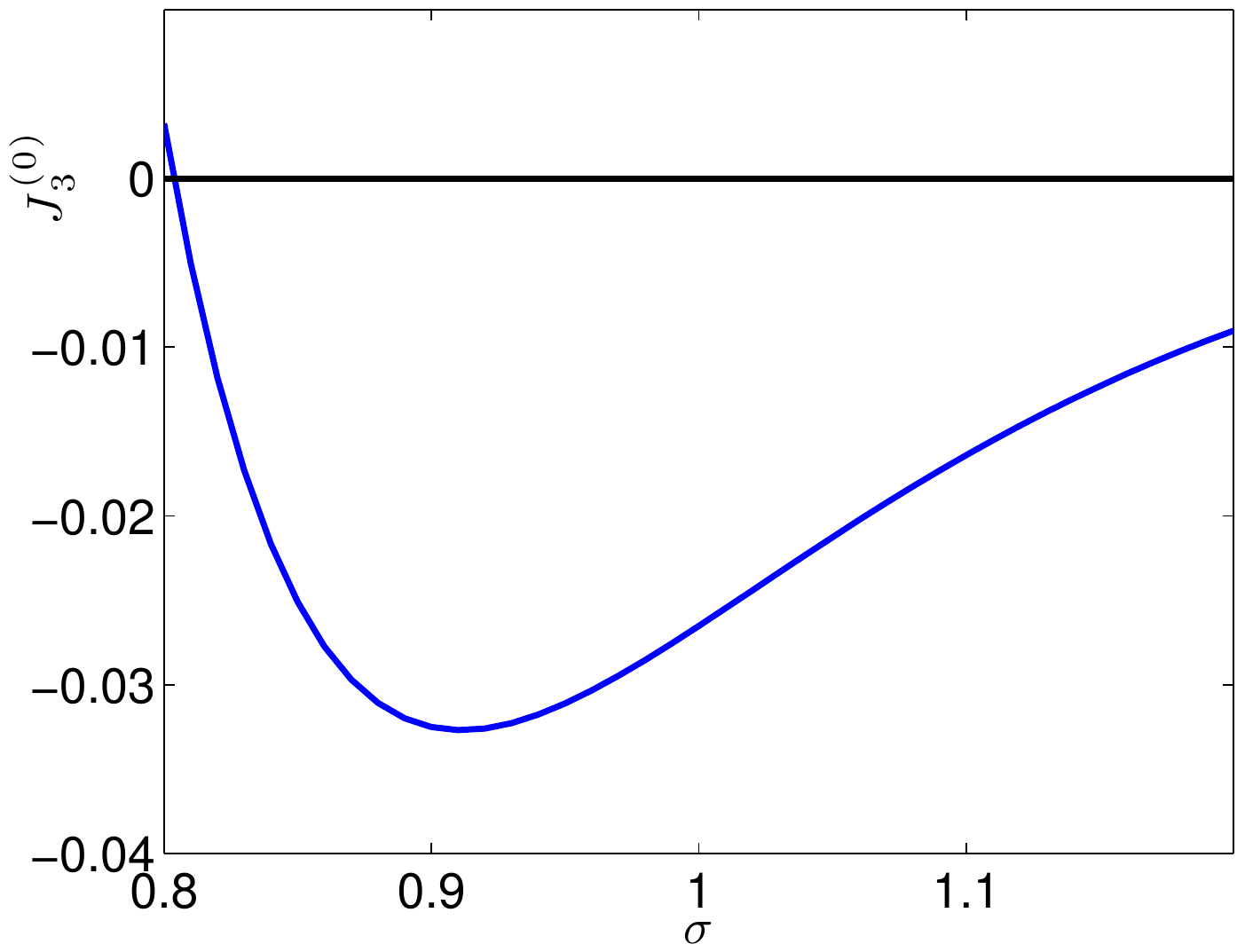}}
  \caption{$J_j^{(0)}$ as functions of $\sigma$ for 3D supercritical
    NLS.}
  \label{f:nls3d_j_0_var}
\end{figure}

\begin{figure}
  \includegraphics[width=2.4in]{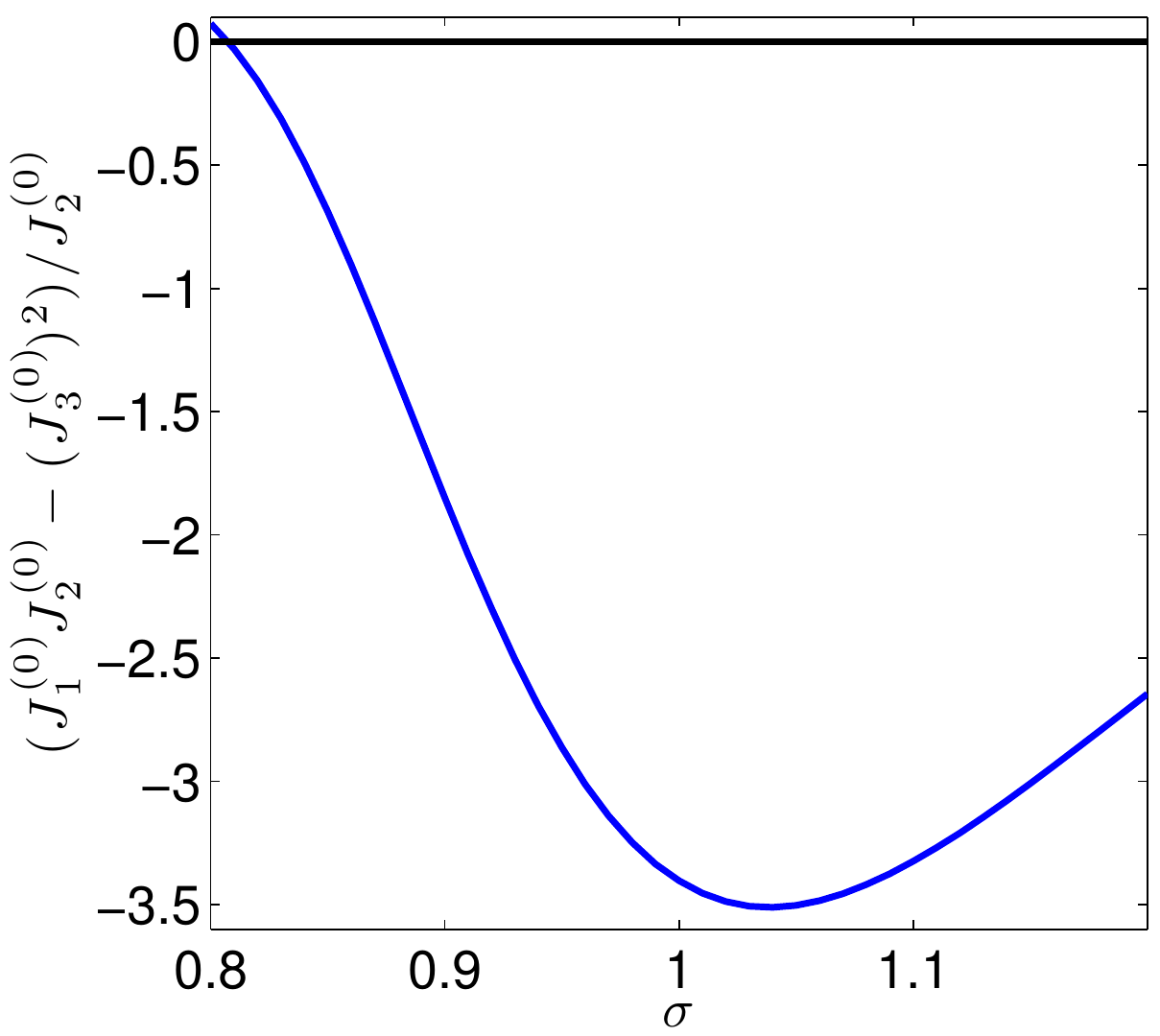}
  \caption{\eqref{e:nls3d_j0_prod2_def} as a function of $\sigma$ for
    3D supercritical NLS.}
  \label{f:nls3d_jprod2_0_var}
\end{figure}

\begin{prop}
  \label{p:nls3d_k1}
  Let $U_1^{(1)}$ solve
  \begin{equation}
    \calL_+^{(1)} U_1^{(1)} = r R, \quad U_1^{(1)} \in L^\infty
  \end{equation}
  Define
  \begin{equation}
    \label{e:k1_defs}
    K_1^{(1)} \equiv  \inner{\calL_+^{(0)} U_1^{(1)} }{U_1^{(1)} } 
  \end{equation}
  There exists $\sigma_3>.8$ such that
  \begin{equation*}
    K_1^{(1)}<0,\quad \text{for $.8 \leq \sigma< \sigma_3$}
  \end{equation*}
  as pictured in Figure \ref{f:nls3d_k1_1_var}, where
  \begin{equation}
    \label{e:nls3d_sigma3}
    \sigma_3 = 1.12092
  \end{equation}
\end{prop}
\begin{proof}
  This is solved in the same manner as Propositions \ref{p:nls3d_k0}
  and \ref{p:nls3d_j0}.
\end{proof}

\begin{figure}
  \includegraphics[width=2.4in]{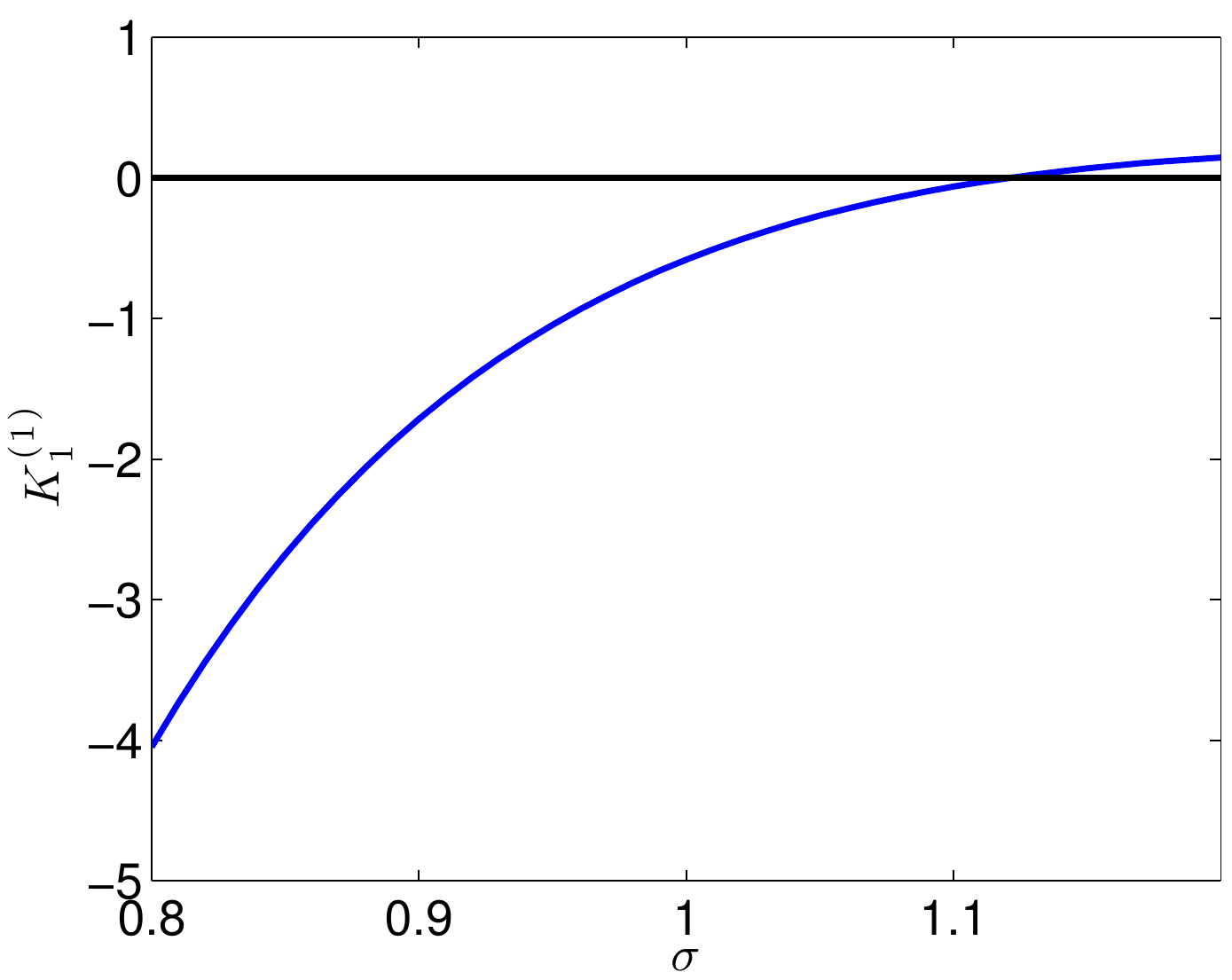}
  \caption{$K_1^{(1)}$ as a function of $\sigma$ for 3D supercritical
    NLS.}
  \label{f:nls3d_k1_1_var}
\end{figure}

\subsubsection{Proof of the Spectral Property}
\label{s:nls3d_proof}
Subject to the acceptance of these computations, we are now in the
position to prove the spectral property for
\[
\sigma_2<\sigma< \sigma_3
\]
where $\sigma_2$ and $\sigma_3$ are as defined by
\eqref{e:nls3d_sigma1} and \eqref{e:nls3d_sigma3}.  We first shall
show that the bilinear form on each harmonic,
$\bar{\mathcal{B}}_\pm^{(k)}$, is positive.  We first prove the case
of $\bar{\mathcal{B}}_+^{(0)}$.


Fixing $\sigma$, for a sufficiently small $\delta_0$, in the sense of
Propositions \ref{p:idx_stability} and \ref{p:inverse_stability},
\[
\frac{1}{\bar K_2^{(0)}}\paren{\bar K_1^{(0)}\bar K_2^{(0)} - (\bar
  K_3^{(0)})^2}<0.
\]
where $\bar K_j^{(0)}$ is the inner product associated with the
perturbed operator, $\bar \calL_+^{(0)}$.

Given that $f$ is orthogonal to $R$ and $\phi_2$, it is orthogonal to
any linear combination, including
\[
\bar q = R - \frac{\bar K_3^{(0)}}{\bar K_2^{(0)}}\phi_2
\]
Let $\bar Q$ solve
\begin{equation}
  \bar{\mathcal{L}}_+^{(0)} \bar{Q} = \bar q.
\end{equation}
By construction,
\begin{equation}
  \label{e:nls3d_spec_proof_BQQ_sign}
  \bar{\mathcal{B}}_+^{(0)}(\bar{Q} , \bar{Q} ) = \frac{1}{\bar
    K_2^{(0)}}\paren{\bar K_1^{(0)}\bar K_2^{(0)} - (\bar K_3^{(0)}) ^2}< 0.
\end{equation} 
To complete the proof of positivity, suppose $\bar{Q} \in H^1_\rad$,
even thought it is not; it decays too slowly for $L^2$.  We could
decompose $H^1_\rad$ as
\begin{equation*}
  H^1_\rad = \spn\set{\bar Q}\oplus \spn\set{\bar Q}^\perp
\end{equation*} 
where the orthogonal decomposition is done {\it with respect to
  $\bar{\mathcal{B}}_+^{(0)}$.}  We can do this because
\eqref{e:nls3d_spec_proof_BQQ_sign} implies the form is
non-degenerate.  Since $\ind (\calB_+^{(0)}) =1$, $\calB_+^{(0)}\geq
0$ on $\spn\set{\bar Q}^\perp$.  Were this not the case, it would
yield a second negative direction, independent of $\bar Q$,
contradicting the index computation.

Given $f\in H^1_\rad$, $f \perp R$ and $f \perp \phi_2$, with respect
to $L^2$ orthogonality, $f$ is also $L^2$ orthogonal to $\bar q$.  We
might then decompose $f$ as proposed in the preceding paragraph
\[
f = c \bar{Q} + f^\perp, \quad \bar{\calB}_+^{(0)} (f^\perp, \bar Q) =
0.
\]
If $c=0$, then $f = f^\perp$ resides in a subspace of $H^1_\rad$ where
$\bar{\calB}_+^{(0)}\geq 0$ , completing the proof.  We now show
$c=0$.

Taking the $L^2$ inner product of $f$ and $\bar q$,
\begin{equation*}
  \begin{split}
    0 &= c \inner{\bar Q}{q} + \inner{f^\perp}{\bar q} = c
    \bar\calB_+^{(0)}(\bar Q, \bar Q) +
    \inner{u^\perp}{\bar\calL_+^{(0)}
      \bar Q} \\
    &= c \bar\calB_+^{(0)}(\bar Q, \bar Q) + \calB_+^{(0)}(f^\perp,
    \bar Q) = c \bar\calB_+^{(0)}(\bar Q, \bar Q) + 0.
  \end{split}
\end{equation*}
Since we have computed $\calB_+^{(0)}(\bar Q, \bar Q) \neq 0$, $c=0$.
Hence, $L^2$ orthogonality to $\bar q$ induces $\bar\calB_+^{(0)}$
orthogonality to $\bar Q$. Thus, $f$ lies in the positive subspace of
$\bar\calB_+^{(0)}$.

However, $\bar Q$ is {\it not} in $L^2$.  To make this argument work,
one can introduce a cutoff function and take an appropriate limit.  We
omit these details and refer the reader to \cite{FMR,
  Marzuola:2010p5770}.

Positivity of $\bar \calB_-^{(0)}$ is proven similarly.  In this case,
we use the $L^2$ orthogonality of $g$ to $\tfrac{1}{\sigma} R + rR'$
and $\phi_1$.  Indeed, we construct a new $\bar Q$, the solution to
\[
\bar \calL_-^{(0)} \bar Q = \tfrac{1}{\sigma} R + rR' - \frac{\bar
  J_3^{(0)}}{\bar J_2^{(0)}}\phi_1
\]
However, this is only successful for $\sigma_2 < \sigma < 1.2$.

Positivity of $\bar \calB_+^{(1)}$ is somewhat easier, as there is no
need to form a linear combination of elements; there is only one
direction, $r R$, to project away from in this harmonic.
$K_1^{(1)}<0$ for all the values of $.8\leq \sigma < \sigma_3$.  We
conclude that $L^2$ orthogonality of $f$ to $r R$ yields positivity.
Since all other forms have index zero, there is nothing to prove for
them.

We have now established that
\begin{equation}
  \bar\calB({\bf z}, {\bf z}) \geq 0
\end{equation}
since each form on each harmonic is positive. It is trivial to see
that
\begin{equation}
  \calB({\bf z}, {\bf z})  \geq \delta_0 \int e^{-\abs{\bf x}} \abs{\bf z}^2 d
  {\bf x}.
\end{equation}
This is almost the desired expression.  Consider,
\begin{equation}
  \calB({\bf z}, {\bf z})  \geq \theta \paren{\int \abs{\grad {\bf
        z}}^2 - \abs{\calV_+} \abs{f}^2 - \abs{\calV_-}\abs{g}^2 d
    {\bf x}} + (1-\theta)\delta_0\int e^{-\abs{\bf x}} \abs{{\bf z}}^2 d
  {\bf x}
\end{equation}
for all $\theta\in (0,1)$.  Both potentials satisfy the estimate
$\abs{\calV_\pm}\lesssim e^{-2\sigma\abs{\bf x}}$, $\sigma>\tfrac{1}{2}$, as $\abs{\bf x}\to
\infty$, so taking $\theta$ sufficiently small,
\begin{equation}
  \calB({\bf z}, {\bf z}) \geq  \theta \int \abs{\grad {\bf z}}^2 d {\bf
    x}  + \frac{(1-\theta)\delta_0}{2}\int e^{-\abs{\bf x}} \abs{{\bf z}}^2 d
  {\bf x}.
\end{equation}
Hence,
\[
\calB({\bf z}, {\bf z}) \gtrsim \int \paren{ \abs{\nabla {\bf z} }^2 +
  e^{-\abs{\bf x}} \abs{{\bf z}}^2 }d {\bf x}
\]

This proves the Spectral Property from which we then immediately get
Theorem \ref{t:3dnls} for purely imaginary eigenvalues.  As resonances
in $d=3$ have sufficient decay, we can also rule them out.

\subsection{3D CQNLS}
\label{s:3dcqnls}

In this section we prove the Spectral Property for the 3D CQNLS
equation.  It is quite similar to 3D NLS, though we are now concerned
with the values of $\gamma$ in \eqref{e:cqnls} for which it holds.

\subsubsection{Indexes}
\label{s:cqnls_idx}
\begin{prop}
  \label{p:cqnls_idx}
  For $0\leq\gamma\leq .012$,
  \begin{gather*}
    \ind \calL_+^{(0)} = \ind \calL_-^{(0)} = \ind \calL_+^{(1)} = 1\\
    \ind \calL_+^{(2)} = \ind \calL_-^{(1)} = 0
  \end{gather*}
\end{prop}
\begin{proof}
  Examining Figures \ref{f:cqnls_idx}, we see that for three computed
  values of $0\leq\gamma\leq .012$, inclusive, we have the
  corresponding number of zero crossings.  We argue by continuity that
  this should hold at all points in the interval.
\end{proof}

\begin{figure}
  \subfigure[Harmonic
  $k=0$]{\includegraphics[width=2.4in]{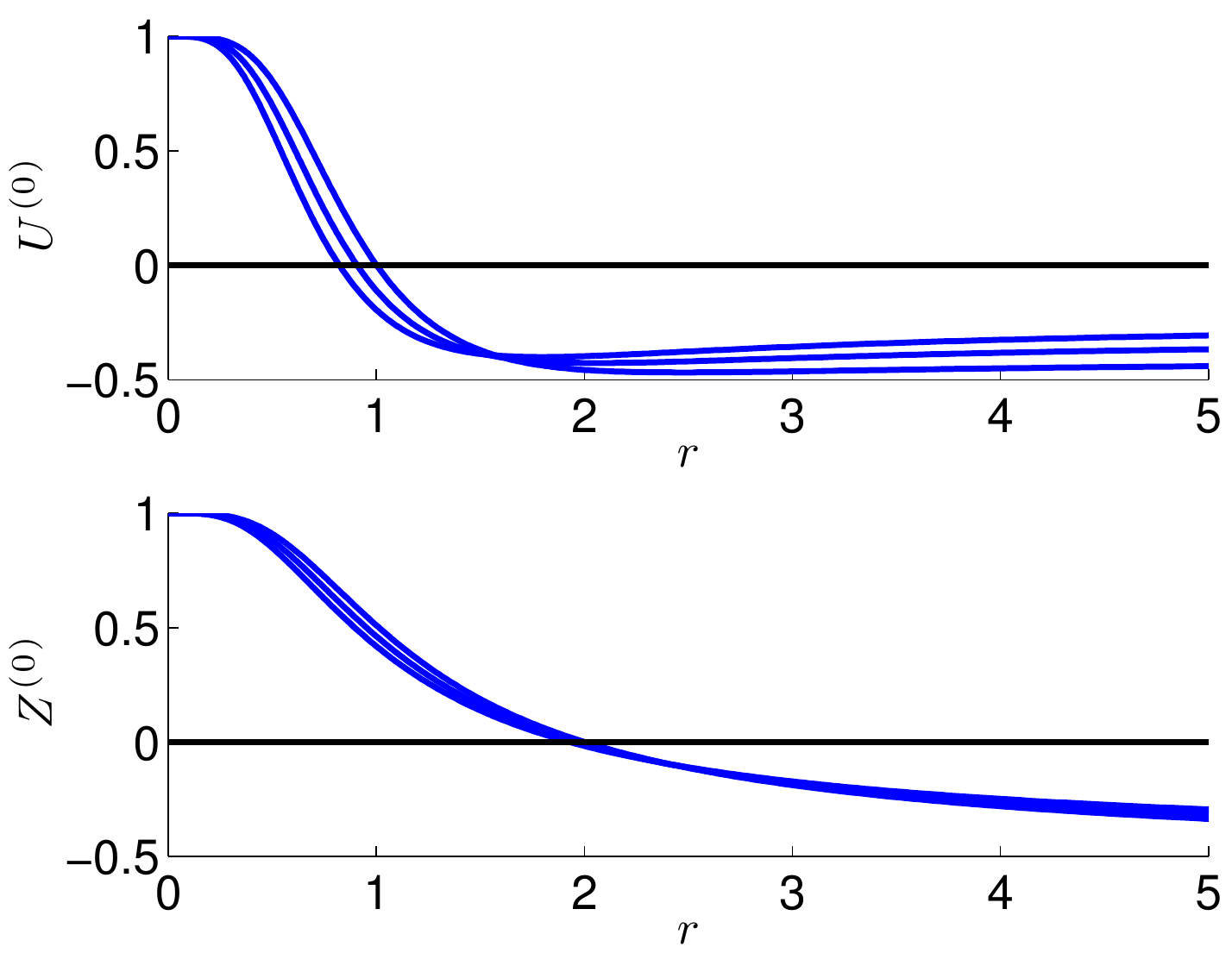}}
  \subfigure[Harmonic
  $k=1$]{\includegraphics[width=2.4in]{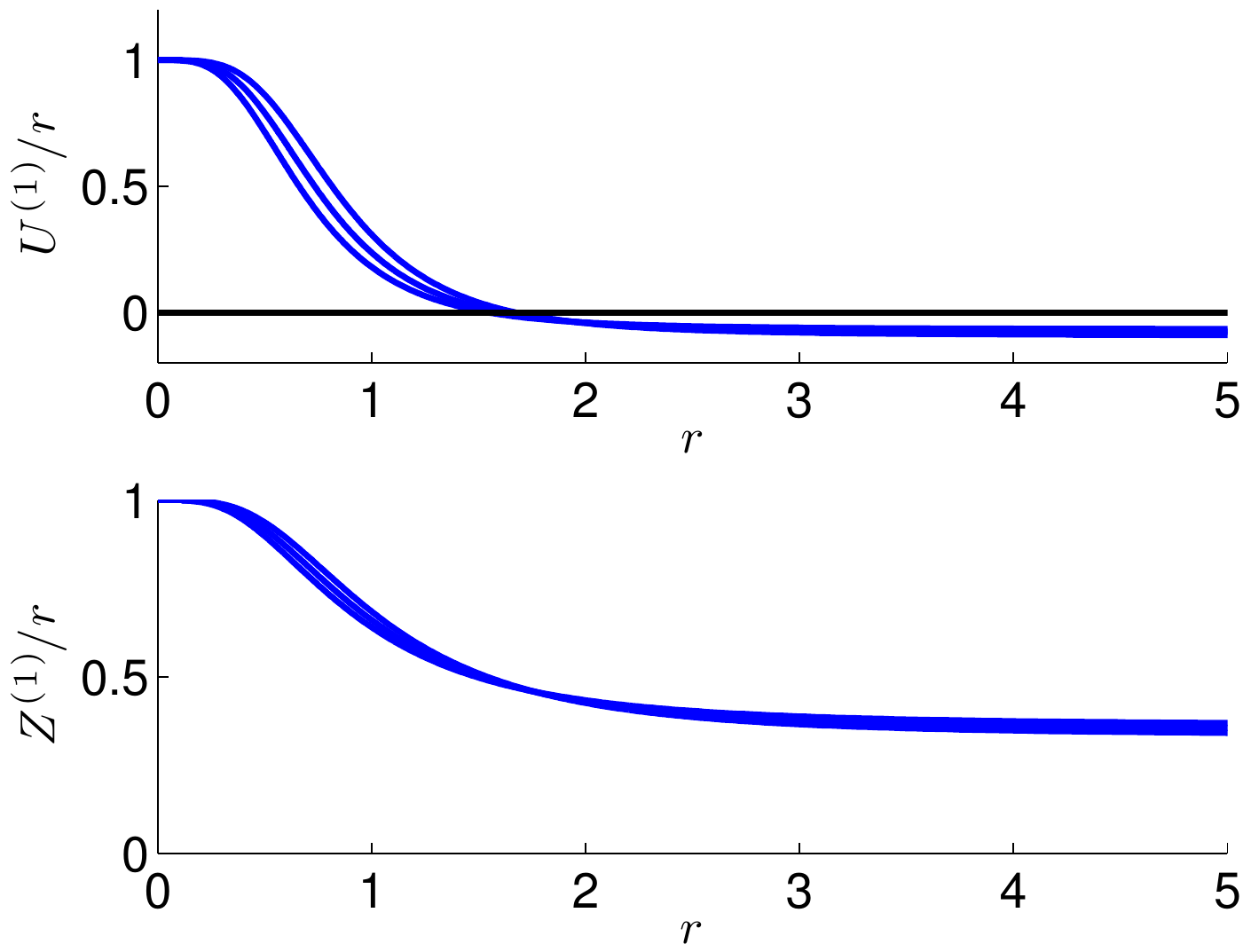}}

  \subfigure[Harmonic
  $k=2$]{\includegraphics[width=2.4in]{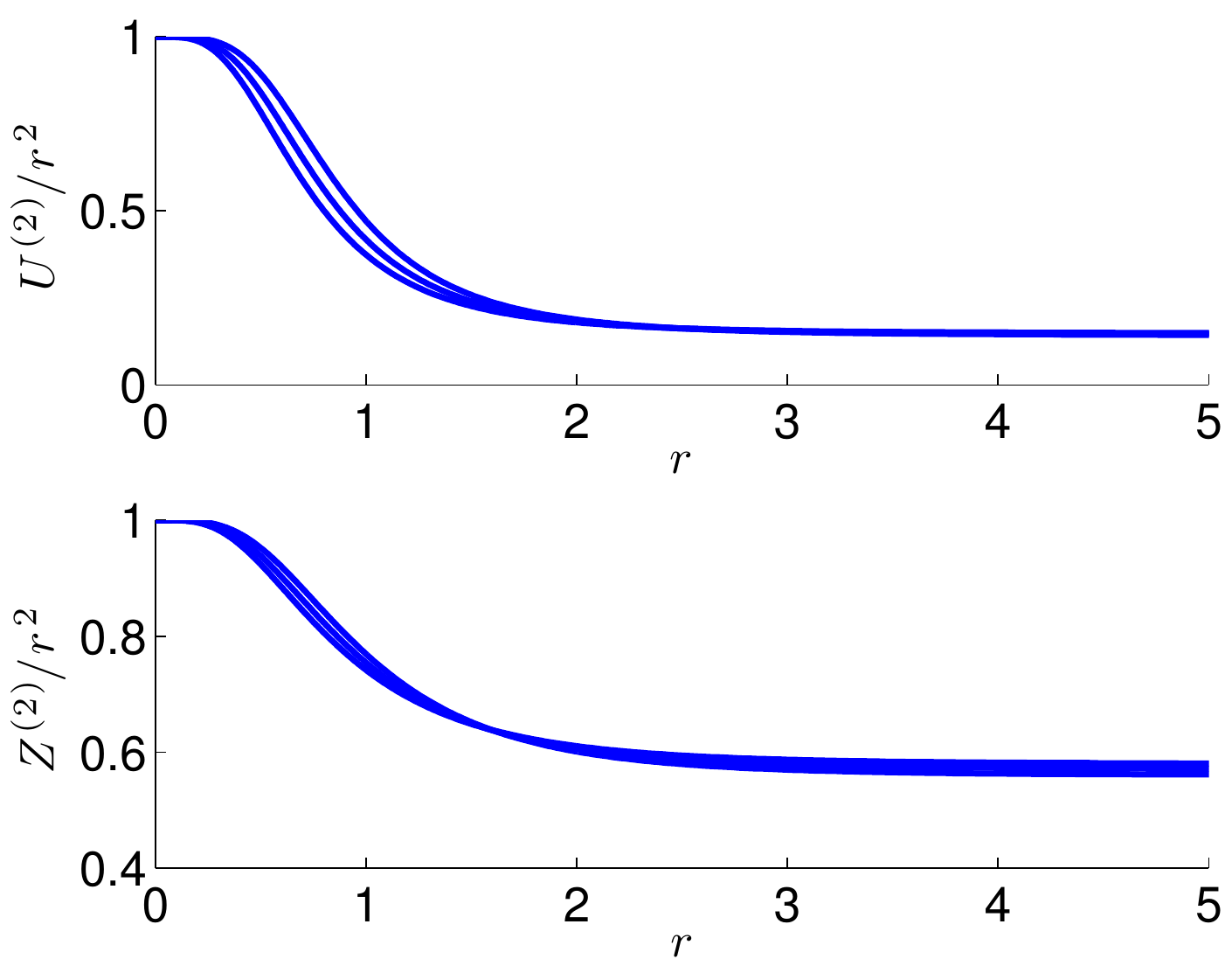}}
  \caption{Index functions for the 3D CQNLS equation computed at three
    values of $0\leq\gamma\leq .012$, inclusive.}
  \label{f:cqnls_idx}
\end{figure}

\subsubsection{Inner Products}

\begin{prop}
  \label{p:cqnls_k0}
  Let $U_1^{(0)}$ and $U_2^{(0)}$, solve
  \begin{subequations}
    \begin{align}
      \calL_+^{(0)} U_1^{(0)} = R \\
      \calL_+^{(0)} U_2^{(0)} = \phi_2
    \end{align}
  \end{subequations}
  Define $K_j^{(0)}$ as in \eqref{e:nls3d_k0_defs}.
  The $K_j^{(0)}$ have the values indicated in Figure
  \ref{f:cqnls_k_0_var}.  Moreover, for $0\leq \gamma\leq 0.012$,
  \begin{equation}
    \label{e:cqnls_k0_prod1_def}
    (K_1^{(0)}  K_2^{(0)} - (K_3^{(0)})^2)/K_1^{(0)}<0
  \end{equation}
  as pictured in Figure \ref{f:cqnls_kprod1_0_var}.

\end{prop}
\begin{proof}
  As before, we prove this by direct computation, at twenty five
  values of $\gamma$ between $0$ and $.012$, inclusive.
\end{proof}

\begin{figure}
  \subfigure[]{\includegraphics[width=2.4in]{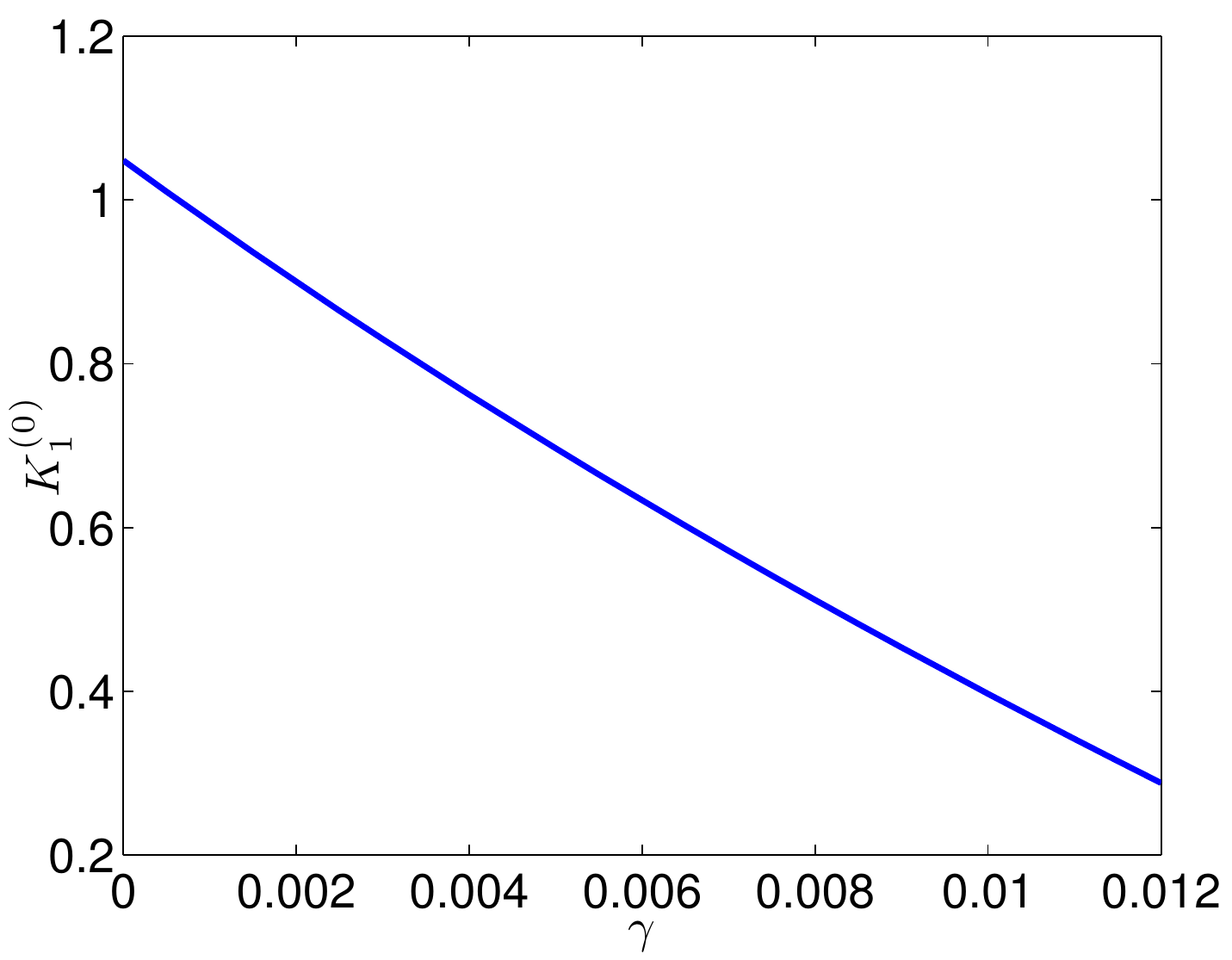}}
  \subfigure[]{\includegraphics[width=2.4in]{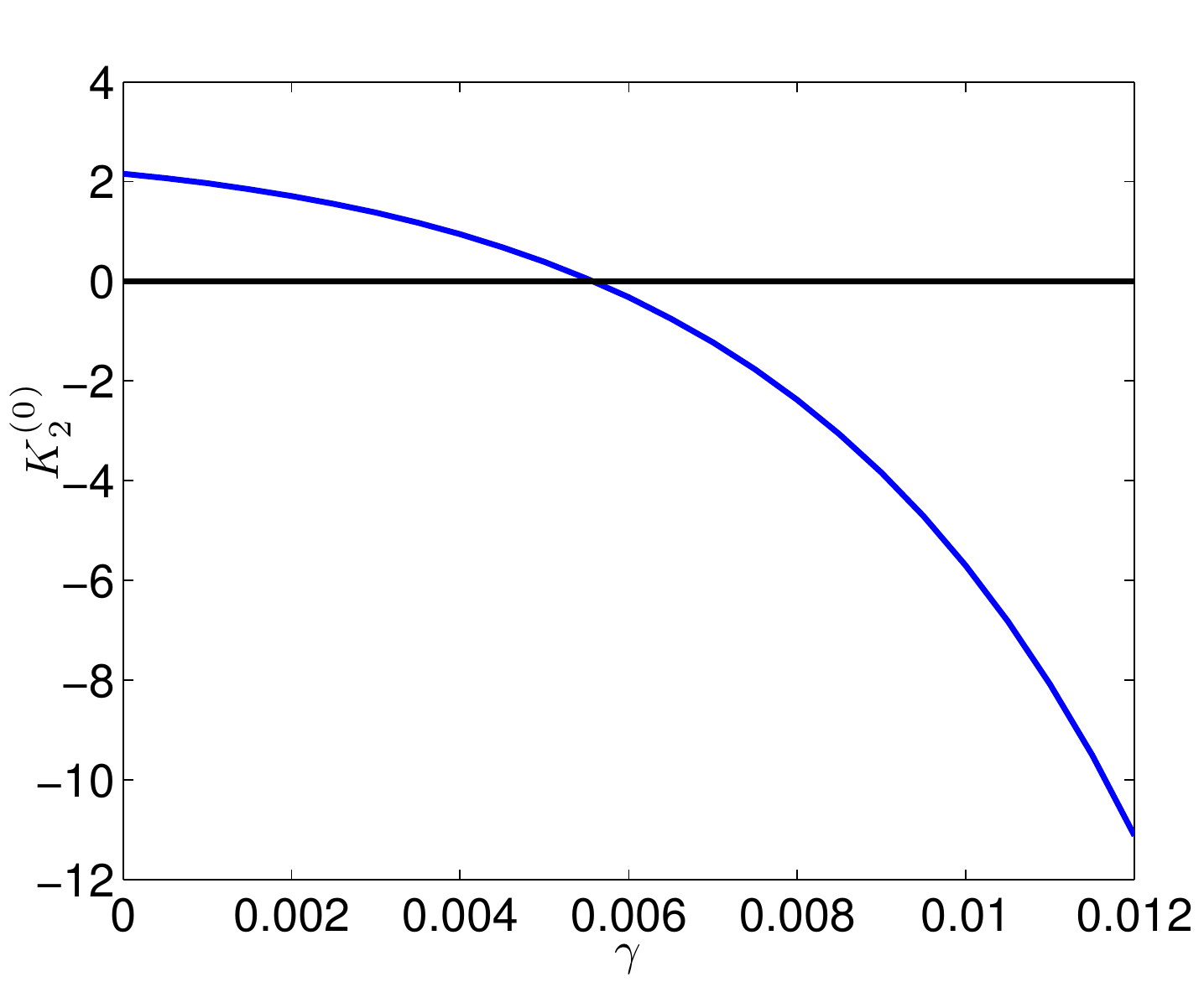}}
  \subfigure[]{\includegraphics[width=2.4in]{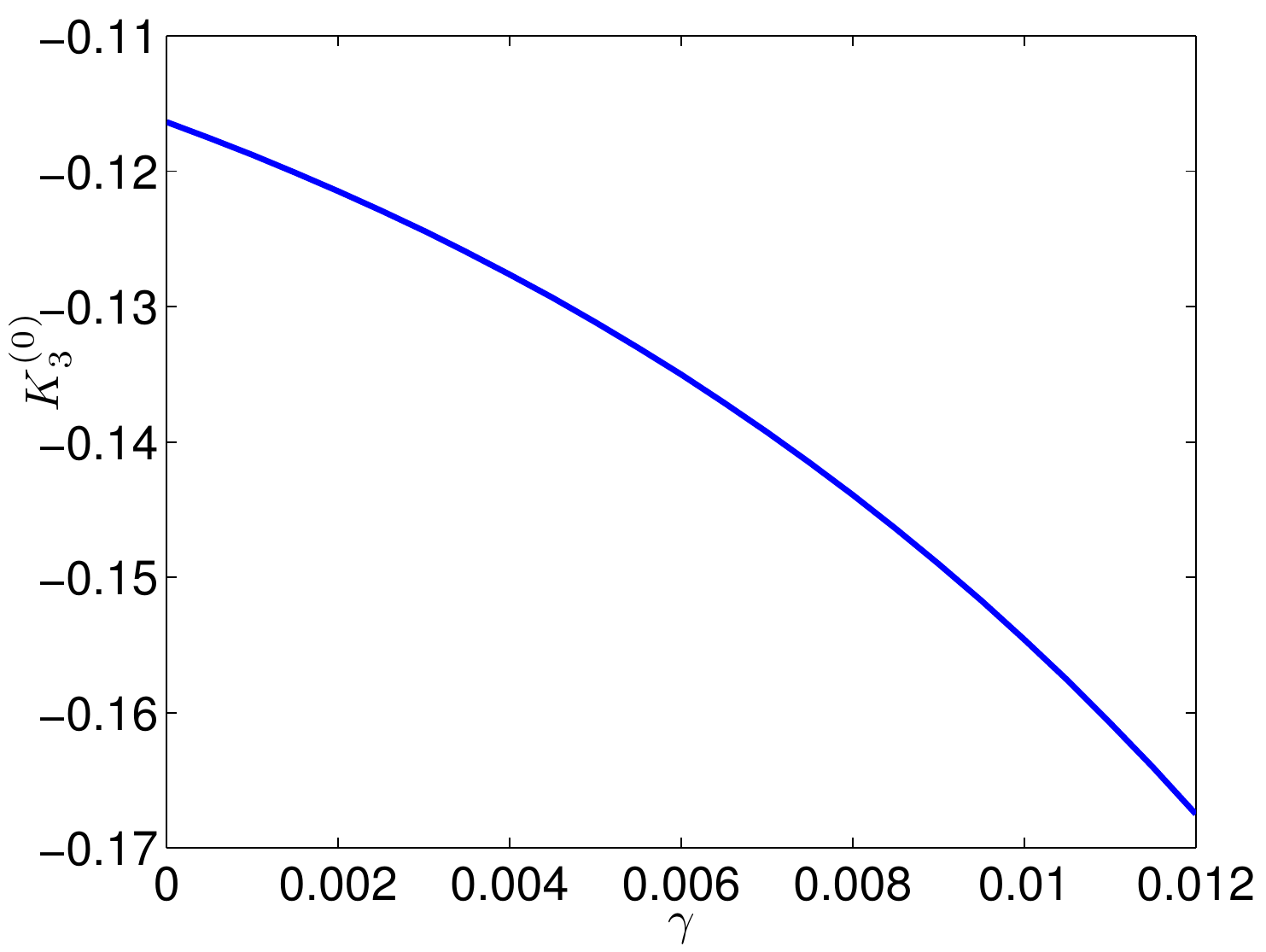}}
  \caption{\eqref{e:nls3d_k0_defs} as functions of $\gamma$ for 3D
    CQNLS.}
  \label{f:cqnls_k_0_var}
\end{figure}

\begin{figure}
  \includegraphics[width=2.4in]{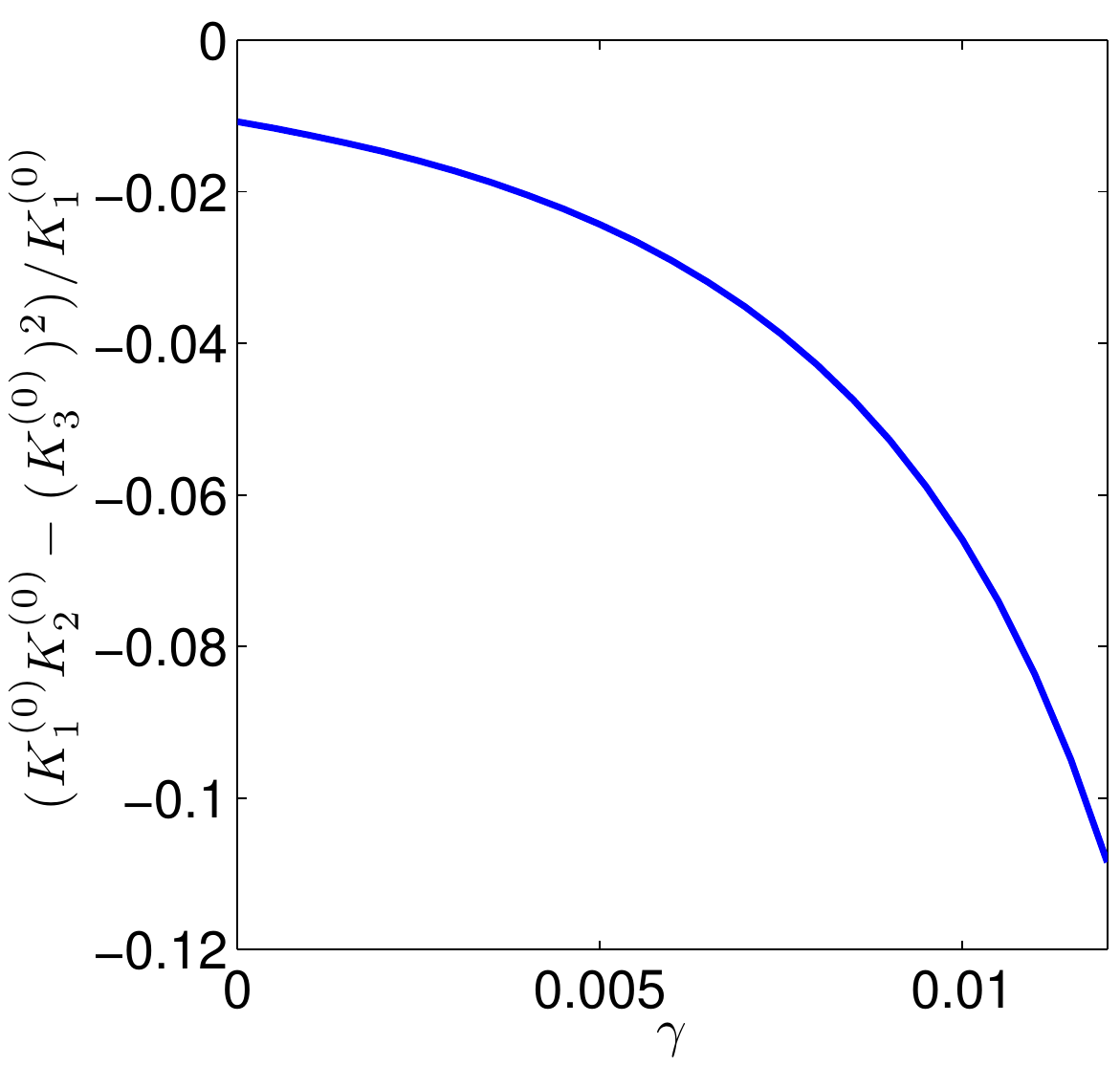}
  \caption{\eqref{e:cqnls_k0_prod1_def} as a function of $\gamma$ for
    3D CQNLS.}
  \label{f:cqnls_kprod1_0_var}
\end{figure}

\begin{prop}
  \label{p:cqnls_j0}
  Let $Z_1^{(0)}$ and $Z_2^{(0)}$, solve
  \begin{subequations}
    \begin{align}
      \calL_-^{(0)} Z_1^{(0)} = \partial_\omega R \\
      \calL_-^{(0)} Z_2^{(0)} = \phi_1
    \end{align}
  \end{subequations}
  Define $J_j^{(0}$ as in \eqref{e:nls3d_j0_defs}.
  The $J_j^{(0)}$ have the values indicated in Figure
  \ref{f:cqnls_j_0_var}.  Moreover there exist $\gamma_1>0$ and
  $\sigma_2>0$ such that
  \begin{gather}
    J_1^{(0)} <0,\quad \text{for $0 \geq \gamma< \gamma_1$}\\
    \label{e:cqnls_j0_prod2_def}
    (J_1^{(0)} J_2^{(0)} - (J_3^{(0)})^2)/J_2^{(0)}<0,\quad \text{for
      $0\leq \gamma<\gamma_2$}
  \end{gather}
  as pictured in Figure \ref{f:cqnls_jprod2_0_var}, where
  \begin{subequations}
    \begin{align}
      \label{e:cqnls_gamma1}
      \gamma_1 &=0.00989115  \\
      \label{e:cqnls_gamma2}
      \gamma_2 &= 0.0109065
    \end{align}
  \end{subequations}

\end{prop}

\begin{figure}
  \subfigure[]{\includegraphics[width=2.4in]{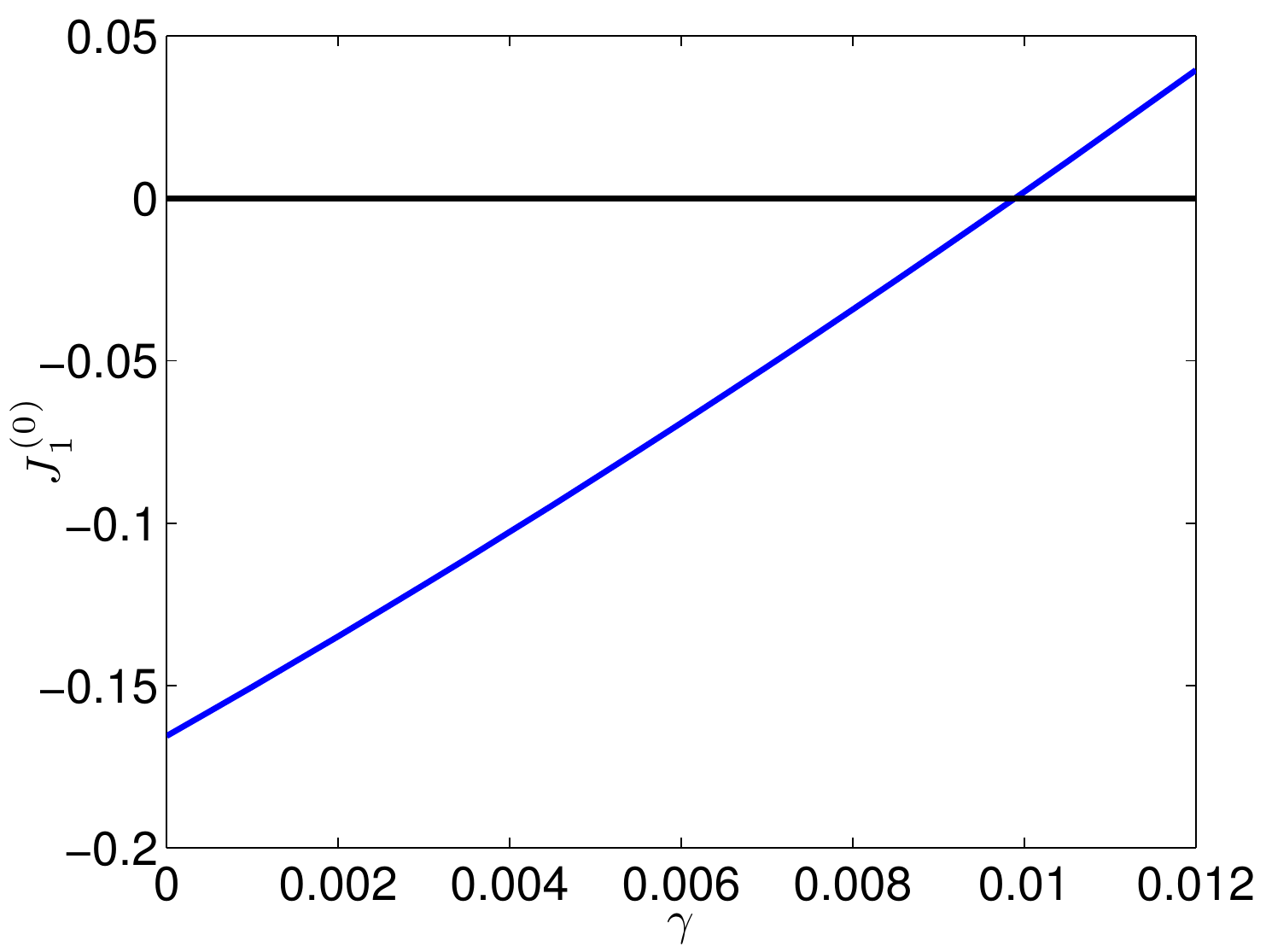}}
  \subfigure[]{\includegraphics[width=2.4in]{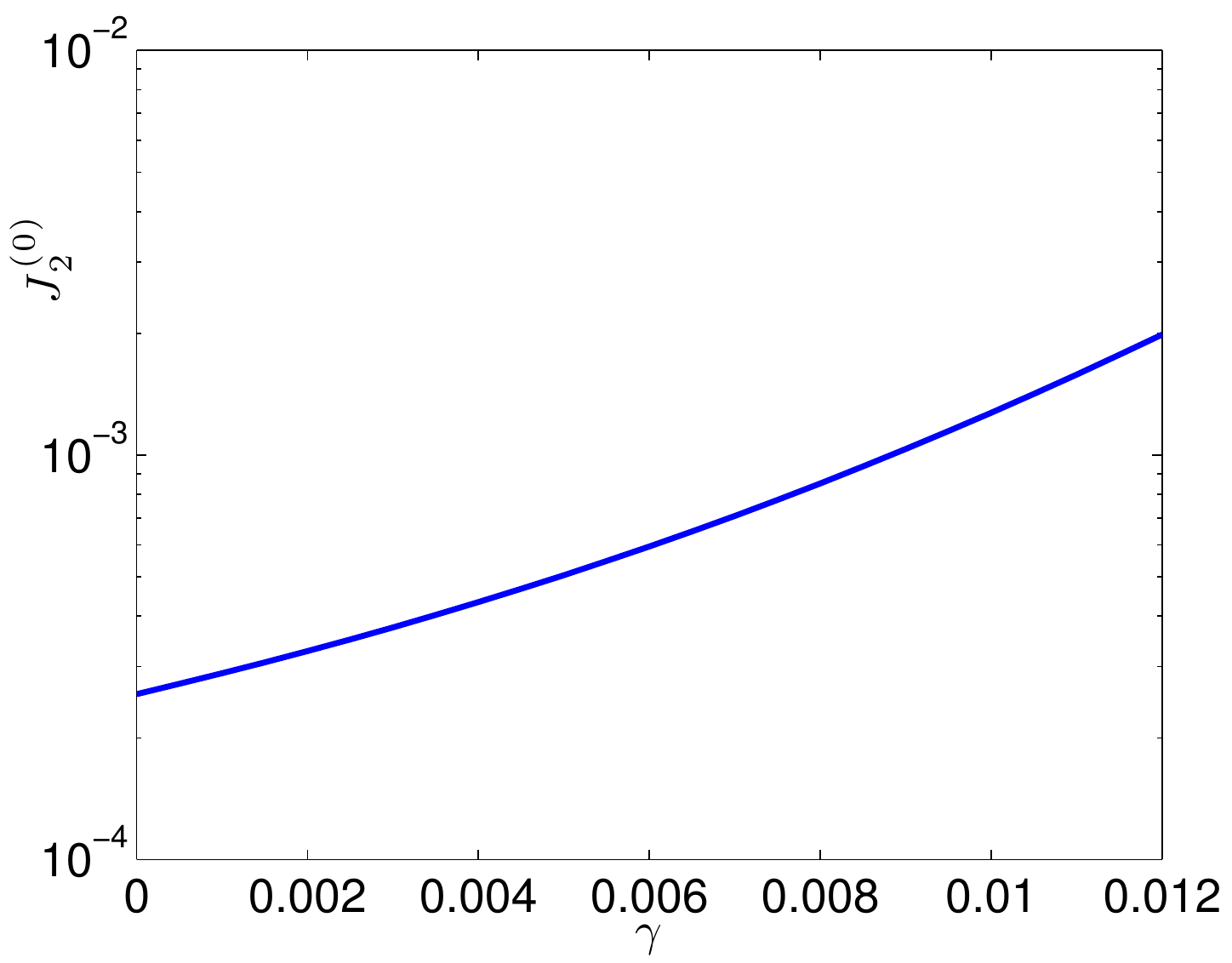}}
  \subfigure[]{\includegraphics[width=2.4in]{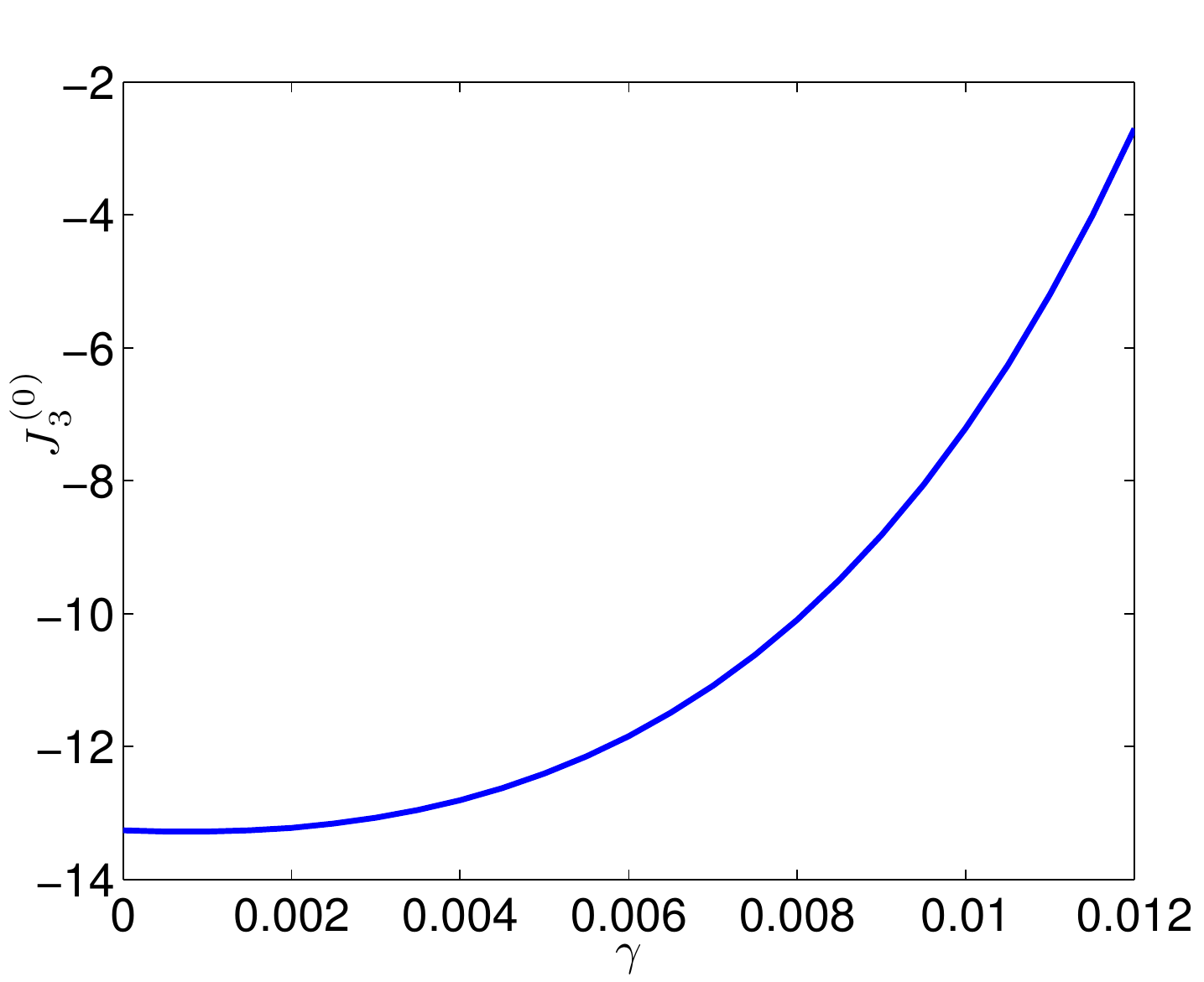}}
  \caption{The $J_j^{(0)}$ as functions of $\gamma$ for 3D CQNLS.}
  \label{f:cqnls_j_0_var}
\end{figure}

\begin{figure}
  \includegraphics[width=2.4in]{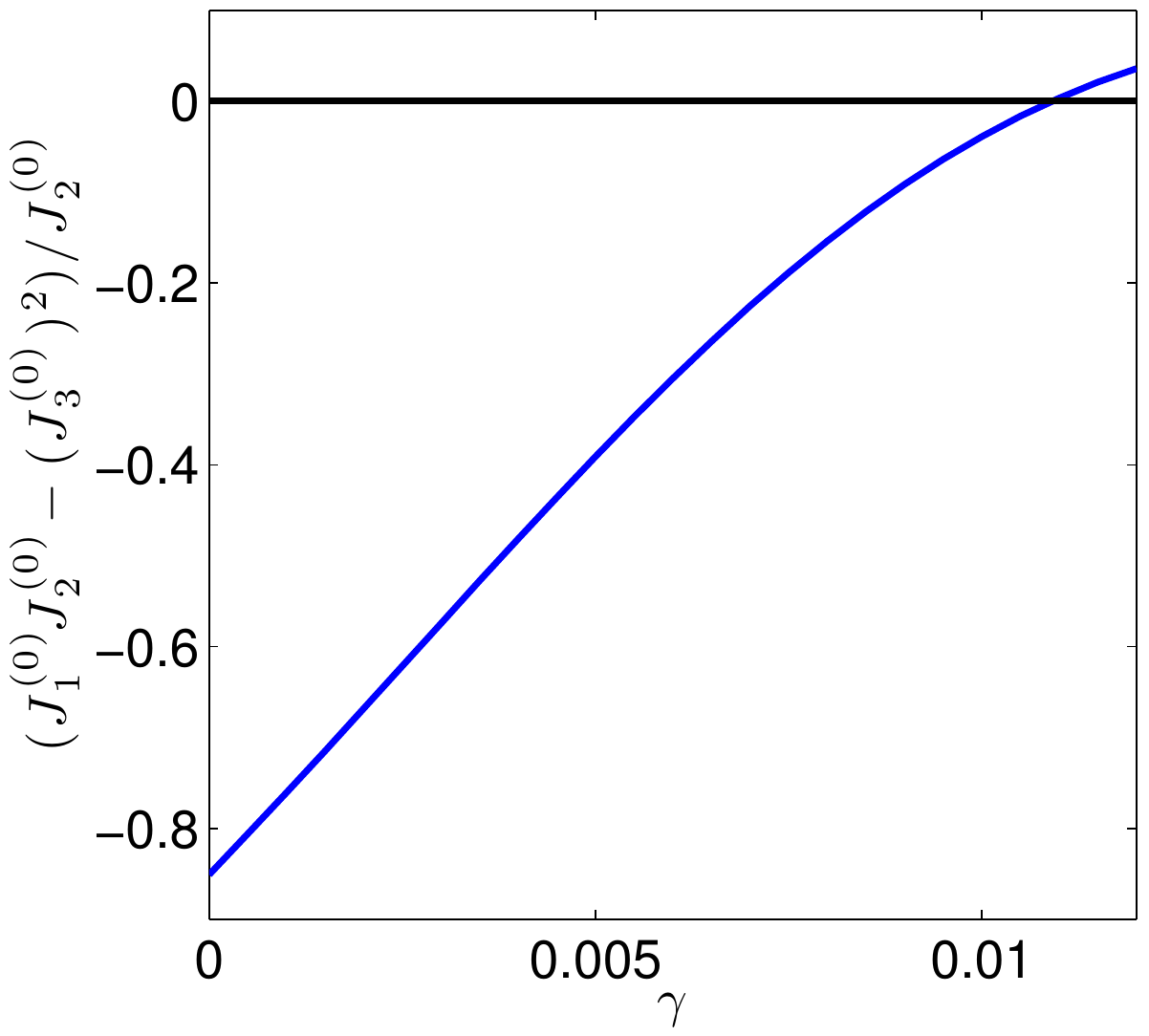}
  \caption{$(J_1^{(0)} J_2^{(0)} - (J_3^{(0)})^2 )/J_2^{(0)}$ as a
    function of $\gamma$ for 3D CQNLS.}
  \label{f:cqnls_jprod2_0_var}
\end{figure}

\begin{prop}
  \label{p:nls3d_k1}
  Let $U_1^{(1)}$ solve
  \begin{equation}
    \calL_+^{(1)} U_1^{(1)} = r R, \quad U_1^{(1)} \in L^\infty
  \end{equation}
  Define $K_1^{(1)}$ as in \eqref{e:k1_defs}.  Then
  \begin{equation*}
    K_1^{(1)}<0,\quad \text{for $0 \leq \gamma 0.12$}
  \end{equation*}
  as pictured in Figure \ref{f:cqnls_k1_1_var}.
\end{prop}

\begin{figure}
  \includegraphics[width=2.4in]{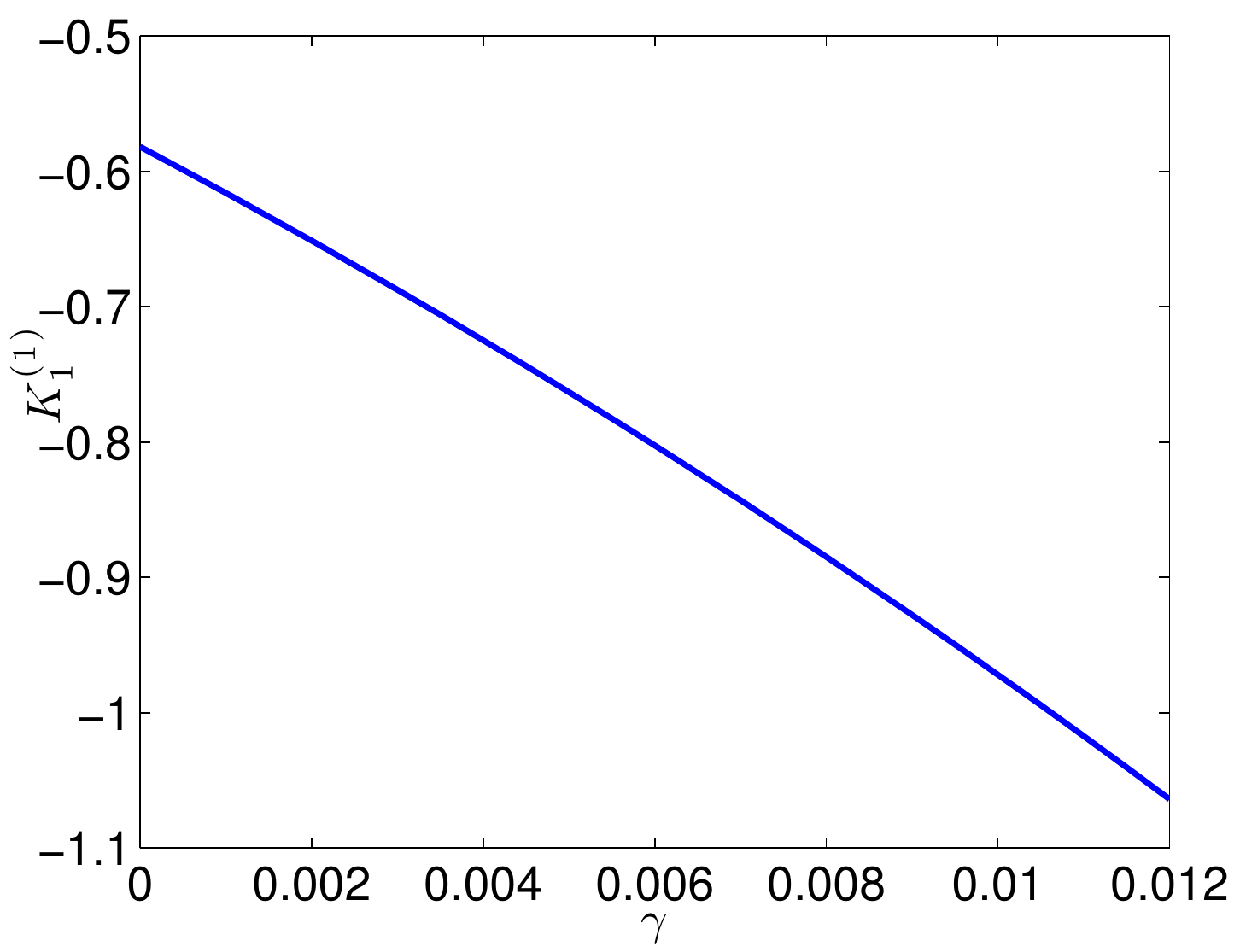}
  \caption{$K_1^{(1)}$ as a function of $\gamma$ for 3D CQNLS.}
  \label{f:cqnls_k1_1_var}
\end{figure}

\subsubsection{Proof of the Spectral Property}
The proof is quite similar to that of 3D NLS and we omit many of the
details.  For a fixed $\gamma$ within the allowable range, we take
$\delta_0$ sufficiently small so as not to alter the indexes or
appreciably change the inner products.

A notable difference is that for $\bar\calB_+^{(0)}$, we let $\bar Q$
solve
\begin{equation}
  \bar \calL_+^{(0)} \bar Q = \bar q = -\frac{\bar K_3^{(0)}}{\bar K_1^{(0)}} R + \phi_2.
\end{equation}
By construction and Proposition \ref{p:cqnls_k0},
\begin{equation}
  \bar\calB_+^{(0)}(\bar Q, \bar Q) =
  \frac{1}{K_1^{(0)}}\paren{K_1^{(0)} K_2^{(0)} - \paren{K_3^{(0)}}^2}< 0.
\end{equation}
If $f $ Is $L^2$ orthogonal to $R$ and $\phi_2$, then it is $L^2$
orthogonal to any linear combination.  As before, this will induce
$\calB_+^{()}$ orthogonality to $\bar Q$, ensuring that
$\calB_+^{(0)}$ is positive for such $f$.  This holds for all $0\leq
\gamma < 0.012$.

Positivity of $\bar\calB_-^{(0)}$ is proven as before, except it only
holds for $0 \leq \gamma < \gamma_2$ due to the results of Proposition
\ref{p:cqnls_j0}.  Positivity of $\bar \calB_+^{(1)}$ is the same, as
are the rest of the forms, since they have index zero.  This proves
positivity of $\bar \calB$ on $\calU$, and the coercivity of $\calB$
on $\calU$ in the form of \eqref{e:spec_prop} follows as before.  This
yields Theorem \ref{t:3dcqnls}.

\subsection{1D Supercritical NLS}
\label{s:1dnls}

In contrast to the 3D problems, where we decomposed into spherical
harmonics, in 1D we decompose into even and odd functions.

\subsubsection{Indexes}
\label{s:nls1d_idx}

\begin{prop}
  \label{p:nls1d_idx}
  For $2.3\leq\sigma\leq 6.3$,
  \begin{equation*}
    \ind (\calB_+^{(e)}) = \ind (\calB_-^{(e)}) = \ind (\calB_+^{(o)}) =
    1,\quad \ind (\calB_-^{(o)}) =0
  \end{equation*}
\end{prop}

\begin{proof}
  As before, we prove this by direct computation.  The index functions
  appear in Figure \ref{f:nls1d_idx}.
\end{proof}

\begin{figure}
  \subfigure[Even
  Functions]{\includegraphics[width=2.4in]{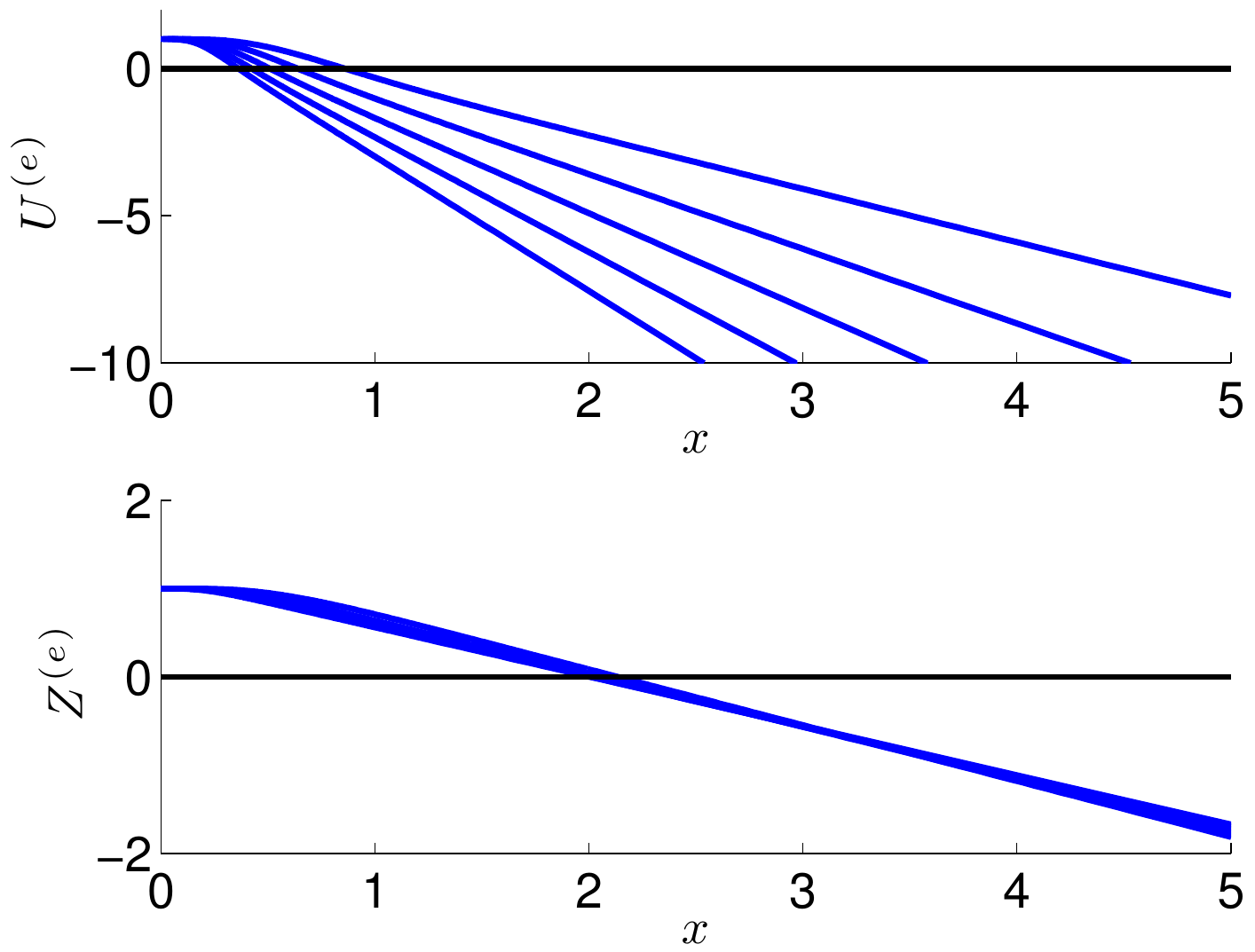}}
  \subfigure[Odd
  Functions]{\includegraphics[width=2.4in]{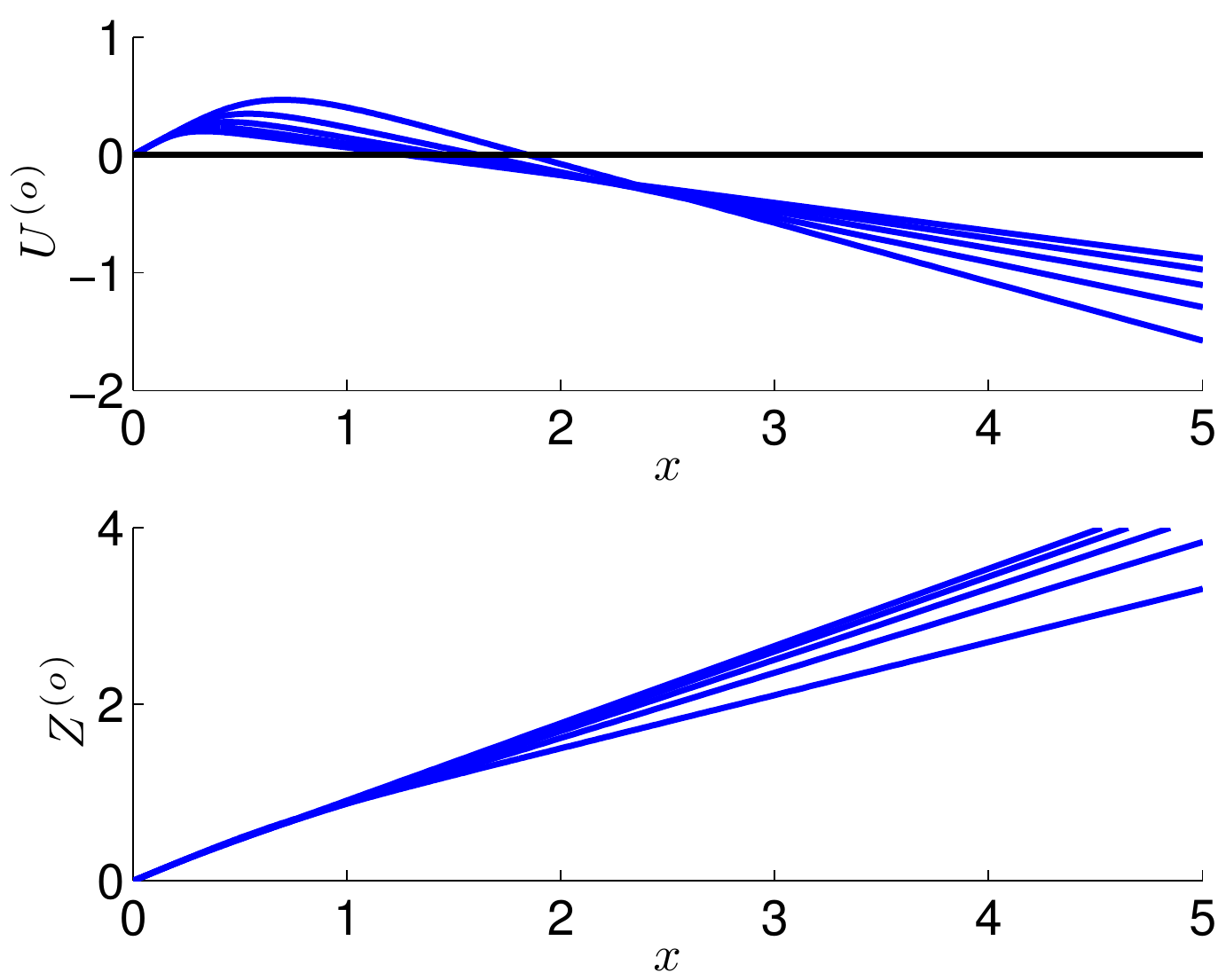}}
  \caption{Index functions for the 1D supercritical NLS equation
    computed at five values of $2.3\leq\sigma\leq 6.3$, inclusive.}
  \label{f:nls1d_idx}
\end{figure}

\subsubsection{Inner Products}
\label{s:nls1d_ip}

\begin{prop}
  \label{p:nls1d_ke}
  Let $U_1^{(0)}$ and $U_2^{(0)}$, solve
  \begin{subequations}
    \begin{align}
      \calL_+^{(e)} U_1^{(e)} = R \\
      \calL_+^{(e)} U_2^{(e)} = \phi_2
    \end{align}
  \end{subequations}
  and define:
  \begin{subequations}
    \label{e:nls1d_ke_defs}
    \begin{align}
      K_1^{(e)} &\equiv  \inner{\calL_+^{(e)} U_1^{(e)}}{U_1^{(e)}},\\
      K_2^{(e)} &\equiv \inner{\calL_+^{(e)} U_2^{(e)}}{U_2^{(e)}}.
    \end{align}
  \end{subequations}
  The $K_j^{(e)}$ have the values indicated in Figure
  \ref{f:nls1d_k_e_var}.  Moreover,
  \begin{align*}
    K_1^{(e)} &<0, \quad \text{for $2.3\leq \sigma<  \sigma_4$}\\
    K_2^{(e)}& <0, \quad \text{for $2.3\leq\sigma \leq 6.3$}
  \end{align*}
  where
  \begin{equation}
    \label{e:nls1d_sigma4}
    \sigma_4 = 3.49928679909
  \end{equation}
\end{prop}

\begin{figure}
  \subfigure[]{\includegraphics[width=2.4in]{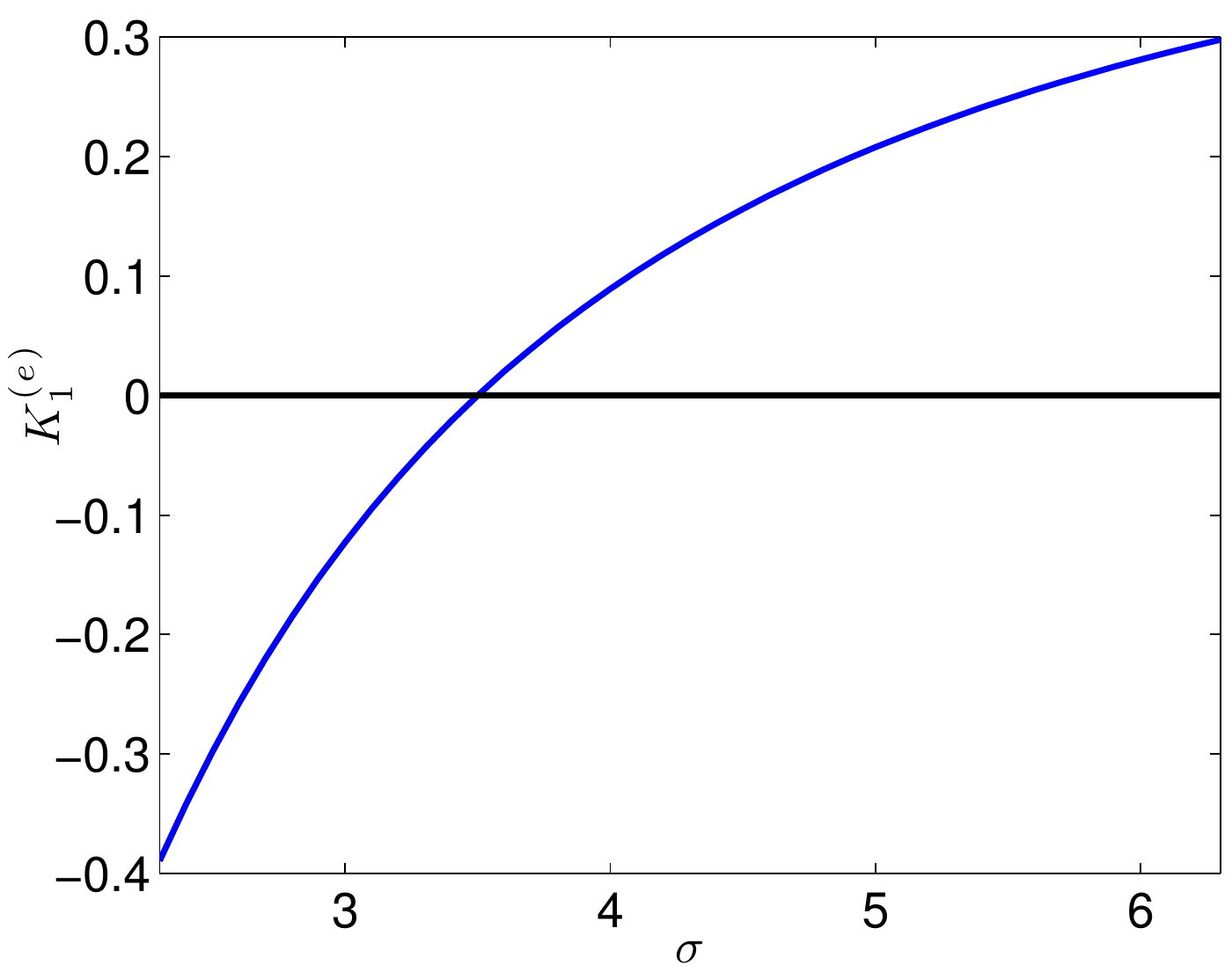}}
  \subfigure[]{\includegraphics[width=2.4in]{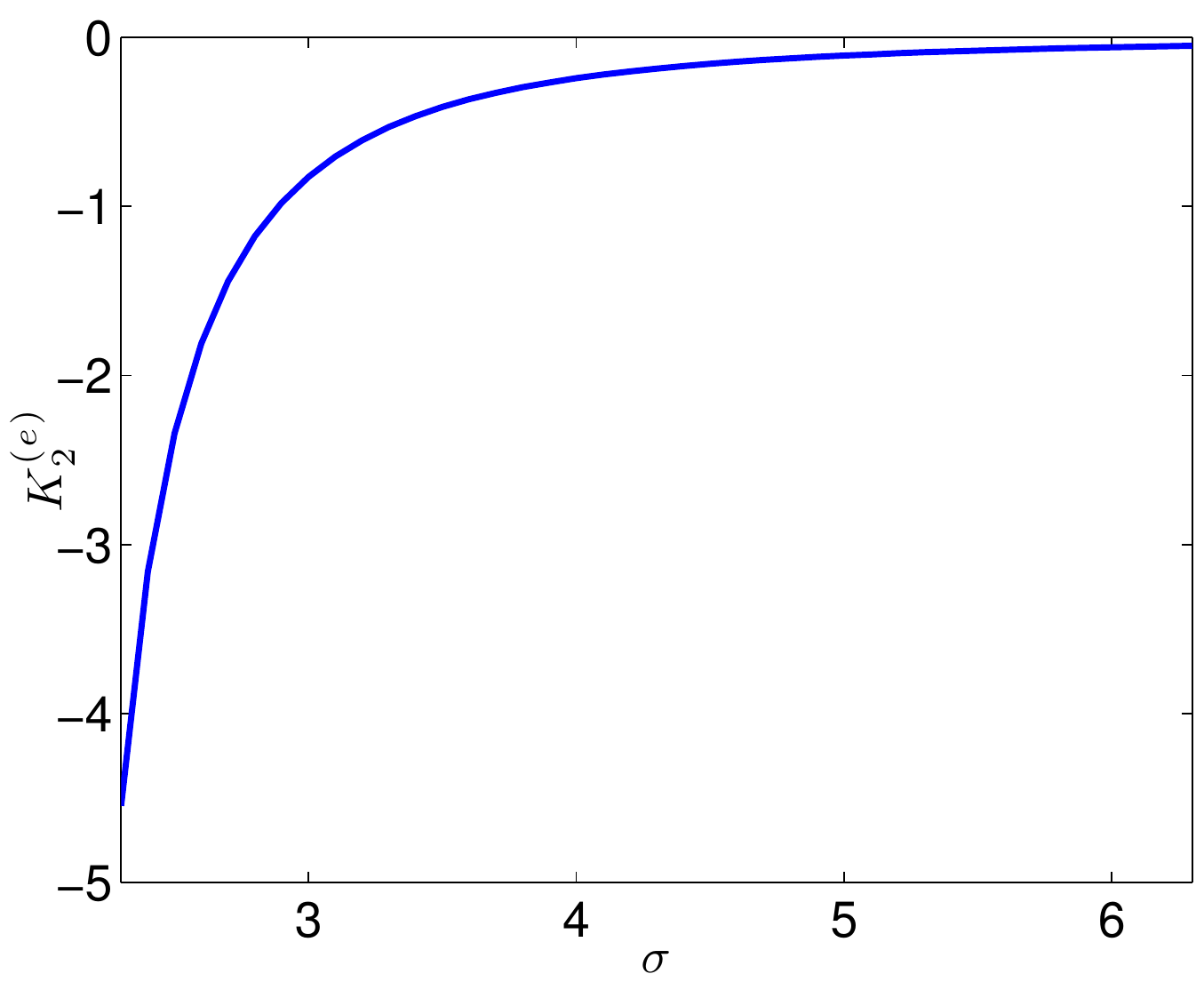}}
  \caption{The $K_j^{(e)}$ as functions of $\sigma$ for 1D
    supercritical NLS.}
  \label{f:nls1d_k_e_var}
\end{figure}

\begin{prop}
  \label{p:nls1d_je}
  Let $Z_1^{(e)}$ and $Z_2^{(e)}$, solve
  \begin{subequations}
    \begin{align}
      \calL_-^{(e)} Z_1^{(e)} &= \tfrac{1}{\sigma}R + x R' \\
      \calL_-^{(e)} Z_2^{(e)} &= \phi_1
    \end{align}
  \end{subequations}
  and define:
  \begin{subequations}
    \label{e:nls1d_je_defs}
    \begin{align}
      J_1^{(e)} &\equiv  \inner{\calL_-^{(e)} Z_1^{(e)}}{Z_1^{(e)}},\\
      J_2^{(e)} &\equiv  \inner{\calL_-^{(e)} Z_2^{(e)}}{Z_2^{(e)}},\\
      J_3^{(e)} &\equiv \inner{\calL_-^{(e)} Z_2^{(e)}}{Z_1^{(e)}}.
    \end{align}
  \end{subequations}
  The $J_j^{(0)}$ have the values indicated in Figure
  \ref{f:nls1d_j_e_var}.  Moreover,
  \begin{subequations}
    \begin{align}
      J_1^{(e)} &<0, \quad \text{for $\sigma_5\leq\sigma \leq 6.3$},\\
      \label{e:nls1d_je_prod2_def}
      (J_1^{(e)} J_2^{(e)} - (J_3^{(e)})^2)/J_2^{(e)}&<0,\quad
      \text{for $\sigma_6\leq\sigma \leq 6.3$},
    \end{align}
  \end{subequations}
  where
  \begin{subequations}
    \begin{align}
      \label{e:nls1d_sigma5}
      \sigma_5 & = 2.45649878,\\
      \label{e:nls1d_sigma6}
      \sigma_6 & = 2.45379561.
    \end{align}
  \end{subequations}
\end{prop}

\begin{figure}
  \subfigure[]{\includegraphics[width=2.4in]{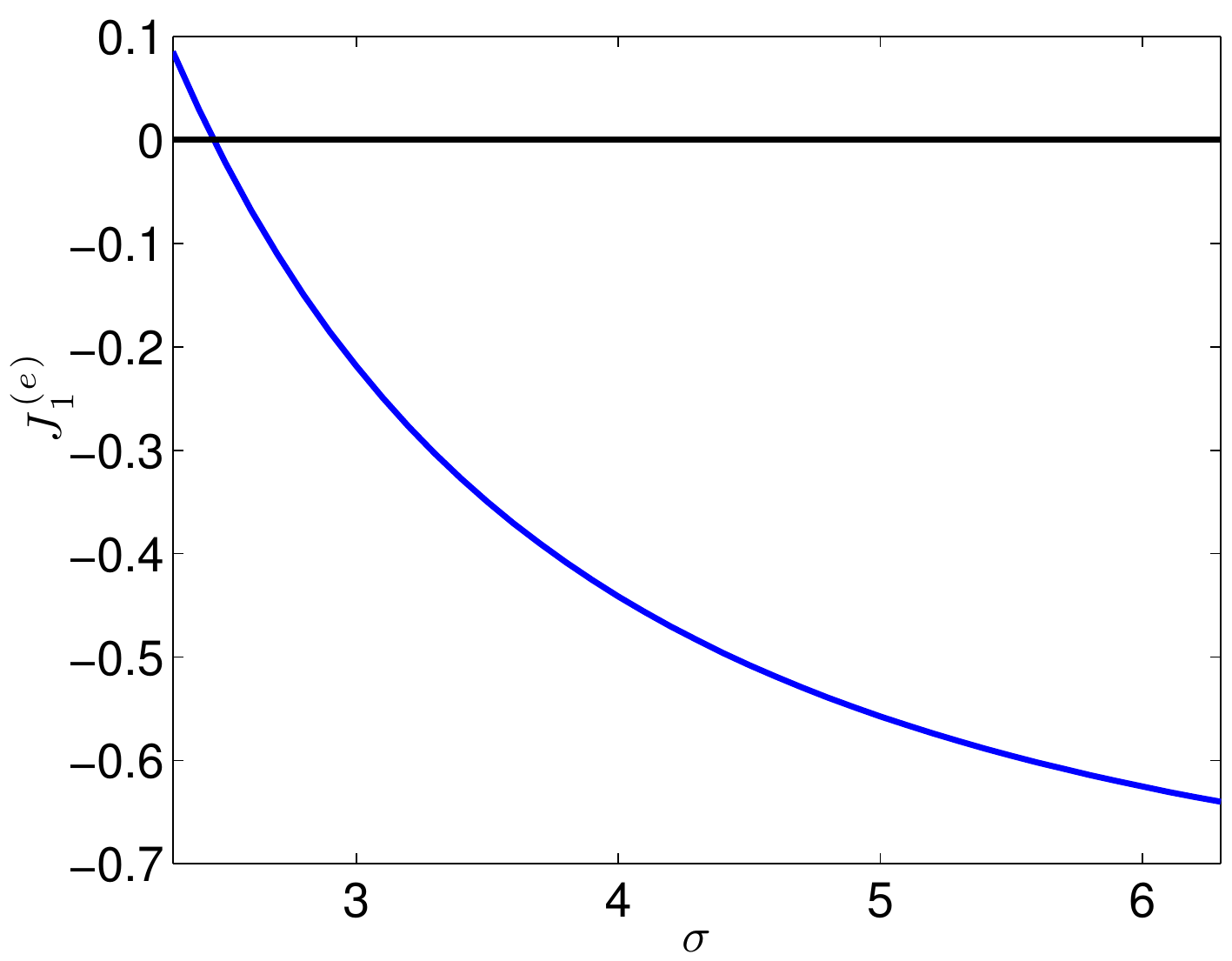}}
  \subfigure[]{\includegraphics[width=2.4in]{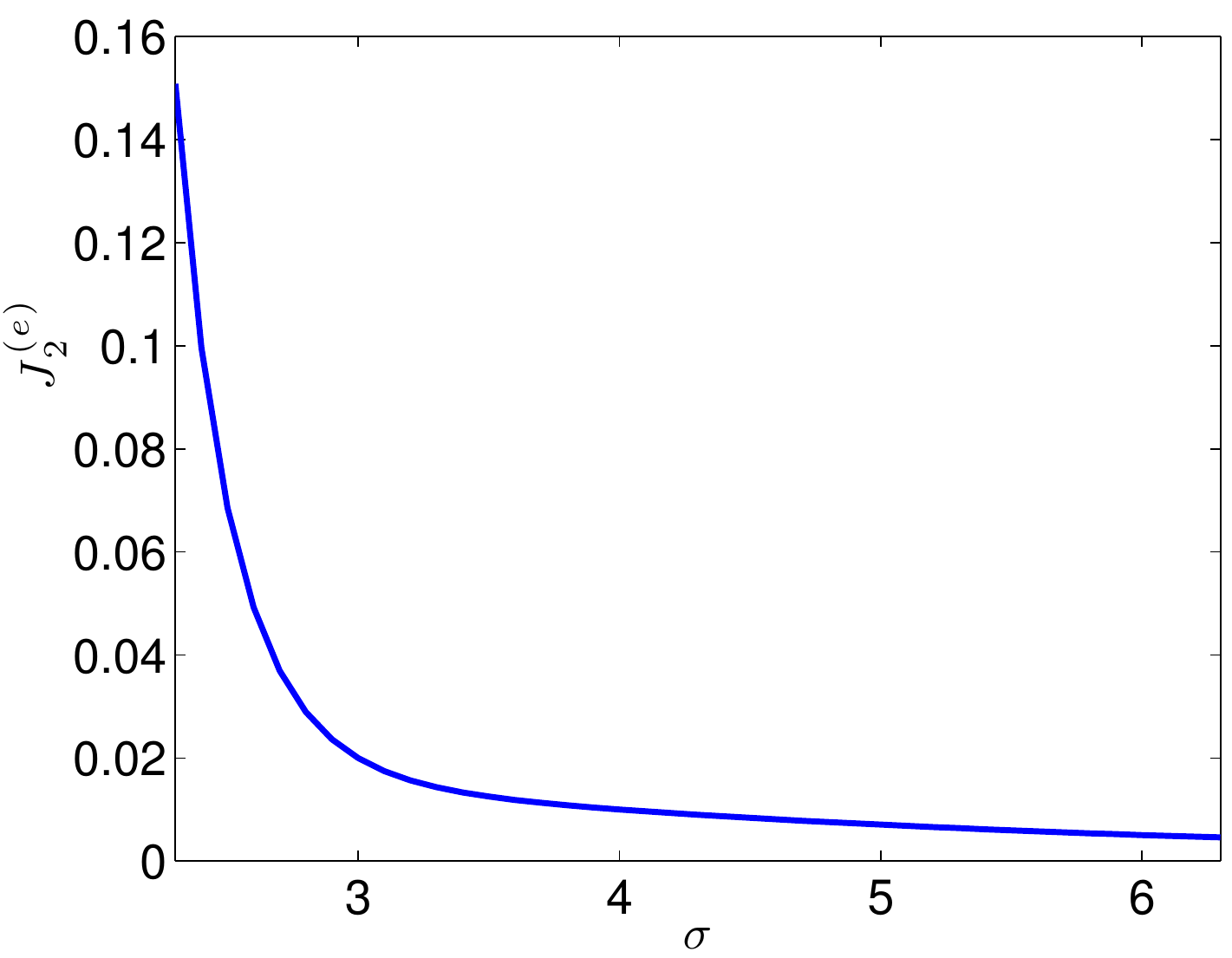}}
  \subfigure[]{\includegraphics[width=2.4in]{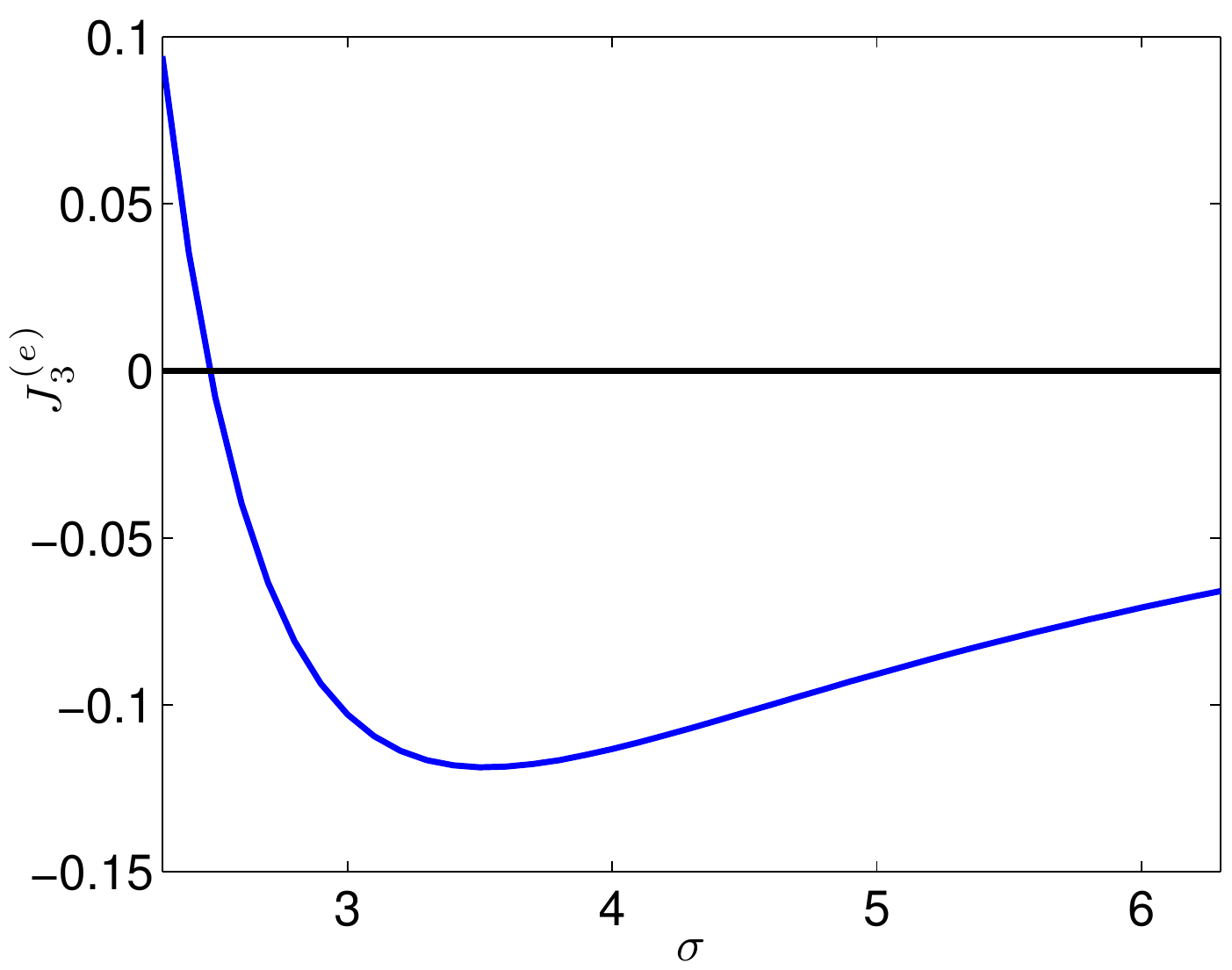}}
  \caption{\eqref{e:nls1d_je_defs} as functions of $\sigma$ for 1D
    supercritical NLS.}
  \label{f:nls1d_j_e_var}
\end{figure}

\begin{figure}
  \includegraphics[width=2.4in]{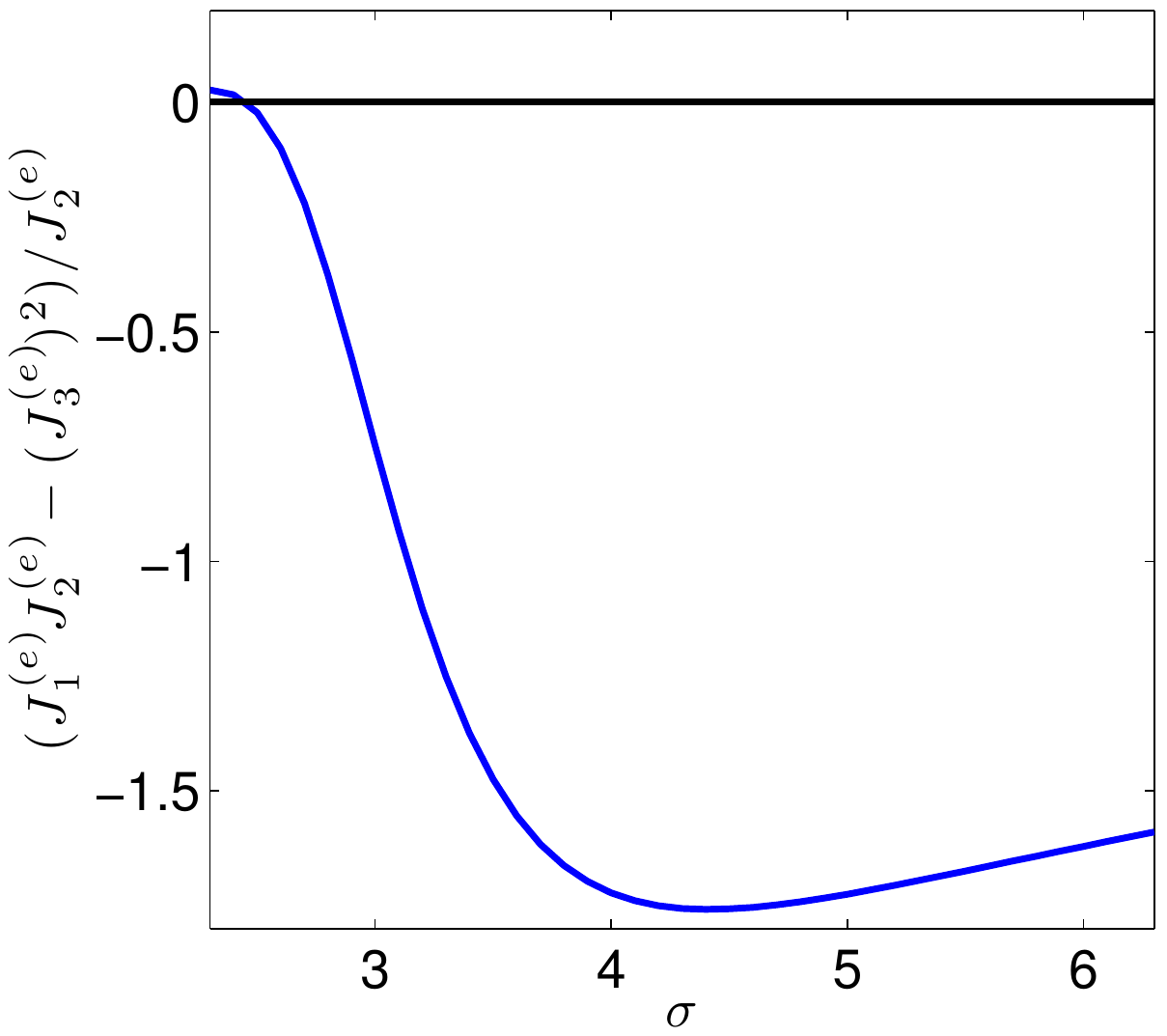}
  \caption{\eqref{e:nls1d_je_prod2_def} as a function of $\sigma$ for
    1D supercritical NLS.}
  \label{f:nls1d_jprod2_e_var}
\end{figure}

\begin{prop}
  \label{p:nls1d_ko}
  Let $U_1^{(o)}$ solve

    \begin{equation}
      \calL_+^{(o)} U_1^{(o)} =  x R
    \end{equation}
    and define
    \begin{equation}
      \label{e:nls1d_ko_defs}
      K_1^{(o)} \equiv  \inner{\calL_+^{(o)} U_1^{(o)}}{U_1^{(o)}}.
    \end{equation}
    $K_1^{(o)}$ has the values indicated in Figure
    \ref{f:nls1d_k_o_var}.  Moreover,
    \begin{equation*}
      K_1^{(o)}  <0, \quad \text{for $2.3\leq\sigma \leq \sigma_7$}\\
    \end{equation*}
    where
    \begin{equation}
      \label{e:nls1d_sigma7}
      \sigma_7 = 6.1288520139
    \end{equation}
  \end{prop}

  \begin{figure}
    \includegraphics[width=2.4in]{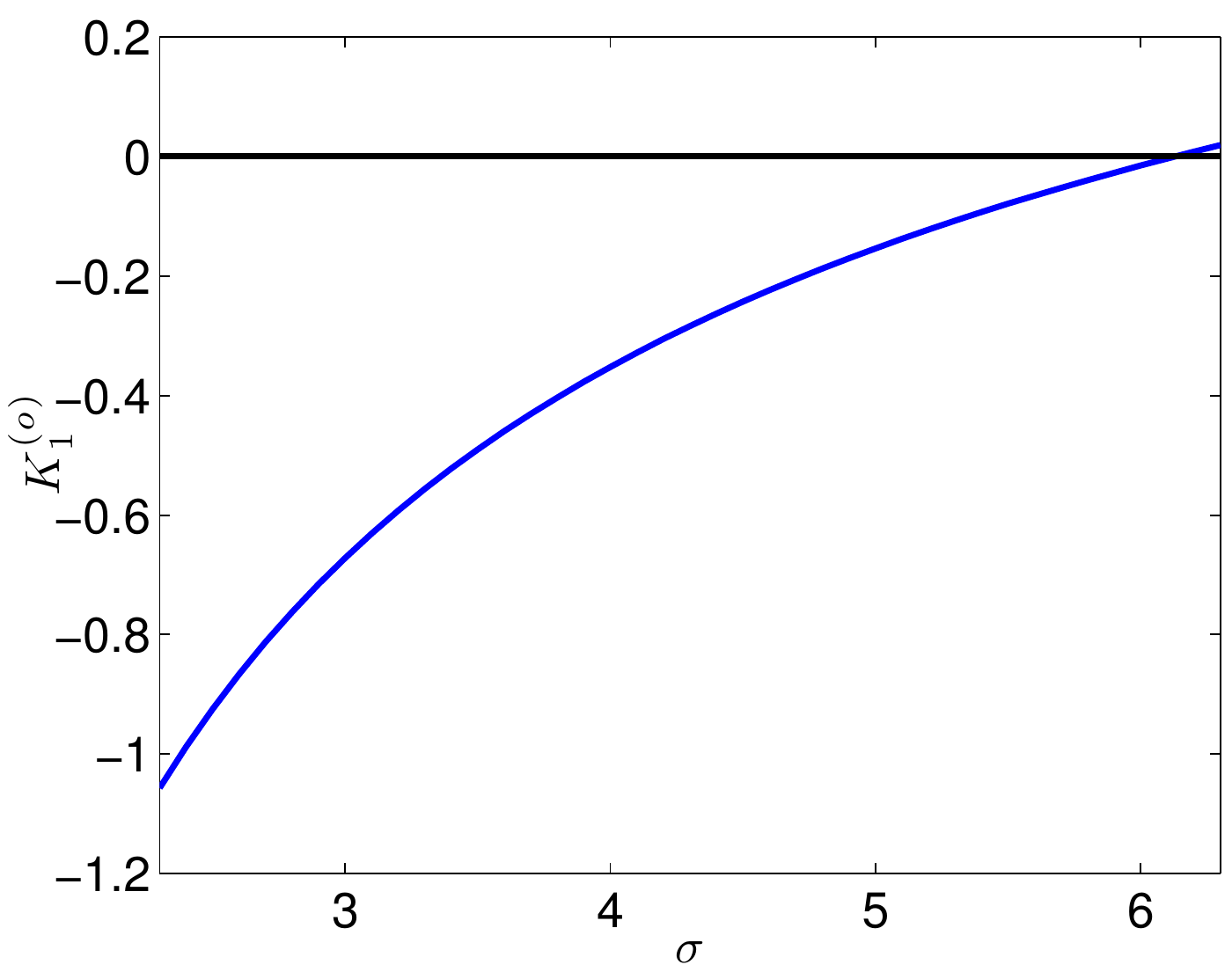}
    \caption{$K_1^{(1)}$ as a function of $\sigma$ for 1D
      supercritical NLS.}
    \label{f:nls1d_k_o_var}
  \end{figure}

  \subsubsection{Proof of the Spectral Property}
  To prove the spectral property in this case, we can use the
  orthogonality of $f$ to $\phi_2$ to induce positivity of
  $\calB_+^{(e)}$ for all $2.3\leq \sigma \leq 6.3$, via Proposition
  \ref{p:nls1d_ke}. For $\calB_-^{(e)}$, $L^2$ orthogonality to
  $\tfrac{1}{\sigma}R + xR'$ and $\phi_1$ induces positivity for
  $\sigma_6< \sigma \leq 6.3$ through Proposition \ref{p:nls1d_je}.
  Finally, Proposition \ref{p:nls1d_ko} yields positivity of
  $\calB_+^{(o)}$ for $2.3\leq \sigma < \sigma_7$.

  \section{Discussion}
  \label{s:discussion}

  We have successfully proven there are no purely imaginary
  eigenvalues, both in the spectral gap and embedded in the essential
  spectrum, for a collection of a orbitally unstable solitary wave
  solutions of NLS.  However, there is still much to be done.  First,
  the range of orbitally unstable solitary waves has not been
  exhausted.  We know there are no purely imaginary eigenvalues for
  $\sigma >2$ in the 1D NLS equation, and we anticipate that there are
  none for $\sigma > 2/3$ in 3D.  We expect similar results for the
  entire unstable branch of CQNLS.  Second, we have shown that the
  threshold found in \cite{DeSc}, given by \eqref{e:desc_bound}, is
  suboptimal; our range, \eqref{e:3dnls_bound}, is larger, but it too
  is likely suboptimal.

  This work also raises questions of what is satisfactory for a
  rigorous proof.  We relied on the computer in the following steps:
  \begin{itemize}
  \item Computing the solitary waves for the 3D, a nonlinear boundary
    value problem on an unbounded domain;
  \item Computing a finite number of index functions, linear initial
    value problems on unbounded domains;
  \item Solving a finite collection of linear elliptic problems on
    unbounded domains;
  \item Computing inner products on an unbounded intervals.
  \end{itemize}
  To deal with these obstacles we truncated the intervals to $[0,
  \rmax]$ ($[0, \xmax]$ in 1D), and introduced a finite number of
  discretization points within the intervals.  Even with these
  computations, we have only established the results for isolated
  values of $\sigma$ or $\gamma$; we argue by continuity that the
  results should extend to continuous intervals.

  However, we are able to produce a variety of consistency and {\it a
    posteriori} checks, such as verifying that the artificial boundary
  conditions are being satisfied sufficiently well.  Despite our
  confidence in our computations, it would still be desirable to avoid
  the use of the computer.  Indeed, an alternative method may permit
  us to move beyond the seemingly artificial restrictions on $\sigma$
  and $\gamma$ of our results.

  \clearpage

  \appendix

  \section{Numerical Methods}
  \label{s:numerics}

  The codes used to produce the results appearing in this work are
  available at
  \url{http://www.math.toronto.edu/simpson/files/spec_prop_asad_simpson_code.zip}
  for examination, verification, and experimentation.

  The numerical methods we use are the same as those used in
  \cite{Marzuola:2010p5770}, and separately in Simpson \& Zwiers,
  \cite{Simpson:2010p8489}.  For the 3D problems, we use the Fortran
  90/95 based boundary value problem solver of Shampine, Muir, \& Xu,
  \cite{shampine2006user}.  For the 1D problems, we found it
  sufficient to use the closely related {\tt bvp4c} routine in
  \matlab, \cite{shampine2003singular,shampine2003solving}.  These
  algorithms require us to specify the problems as first order systems
  of the form
  \begin{equation*}
    \frac{d}{dr}{\bf y} = \tfrac{1}{r}S {\bf y} + {\bf f}(r, {\bf y})
  \end{equation*}
  where $S$ is a constant matrix and ${\bf f}$ contains all
  non-singular terms, both linear and nonlinear.

  This automatically accommodates problems involving $r^{-1}$ type
  singularities at the origin, as found in the equation for the 3D
  solitons and the $k=0$ harmonic boundary value problems.  For higher
  harmonics, we use a point transformation $W(r) = r^k \tilde{W} (r)$
  to transform
  \[
  \bracket{-\frac{d^2}{dr^2 }- \frac{d-1}{r} \frac{d}{dr} +
    \frac{k(k+d-2)}{r^2}}W\Rightarrow r^k \bracket{-\frac{d^2}{dr^2 }-
    \frac{d-1+2k }{r} \frac{d}{dr} } \tilde{W}.
  \]

  \subsection{Index Functions}
  \label{s:idx_numerics}
  For the 3D problems, we computed the index functions on the domain
  $[0, 200]$ with a tolerance setting of $10^{-13}$.  The 1D problems
  were computed on the domain $[0, 100]$ with absolute and relative
  tolerances of $10^{-13}$.  While Figures \ref{f:nls3d_idx},
  \ref{f:cqnls_idx}, and \ref{f:nls1d_idx} show the indicated numbers
  of roots, they obviously do not prove that there is no root at some
  larger value of $x$ or $r$.

  To mitigate this uncertainly, there are two things we can check.
  First, we can examine the solution on a much larger domain.  A more
  systematic approach is to note that since the potentials are highly
  localized, asymptotically, the index functions satisfy the estimates
  \begin{subequations}
    \begin{align}
      U^{(k)}(r) &\approx C_0^{(k)} r^k + C_1^{(k)} r^{2-d-k}\\
      Z^{(k)}(r) &\approx D_0^{(k)} r^k + D_1^{(k)} r^{2-d-k}
    \end{align}
  \end{subequations}
  where $C_j^{(k)}$ and $D_j^{(k)}$ are constants.  Estimating these
  constants, and seeing that they have the correct signs and
  magnitudes, indicates that there cannot be further roots.  For
  example, in Figure \ref{f:idx_example}, we see that
  $C_0^{(0)}\approx -0.3668$ and $C_1^{(0)}\approx -0.2393$.  The only
  root of
  \[
  -0.3668 -0.2393 /r
  \]
  is near $-0.65$.  Thus, there should be no zeros in the ``far
  field''.  Similar analysis works for the $D_j^{(0)}$ constants.

\begin{figure}
  \subfigure{\includegraphics[width=2.4in]{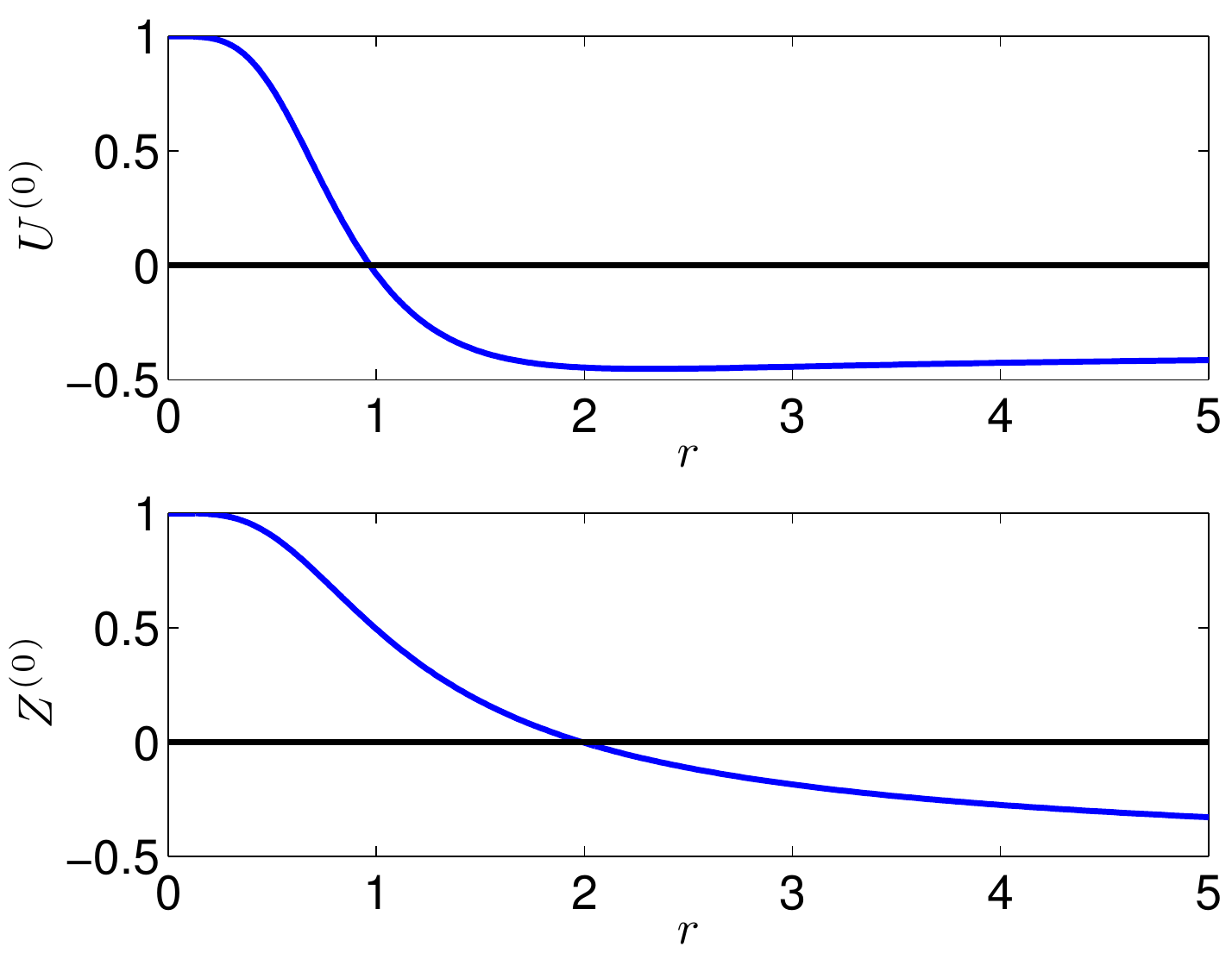}}
  \subfigure{\includegraphics[width=2.4in]{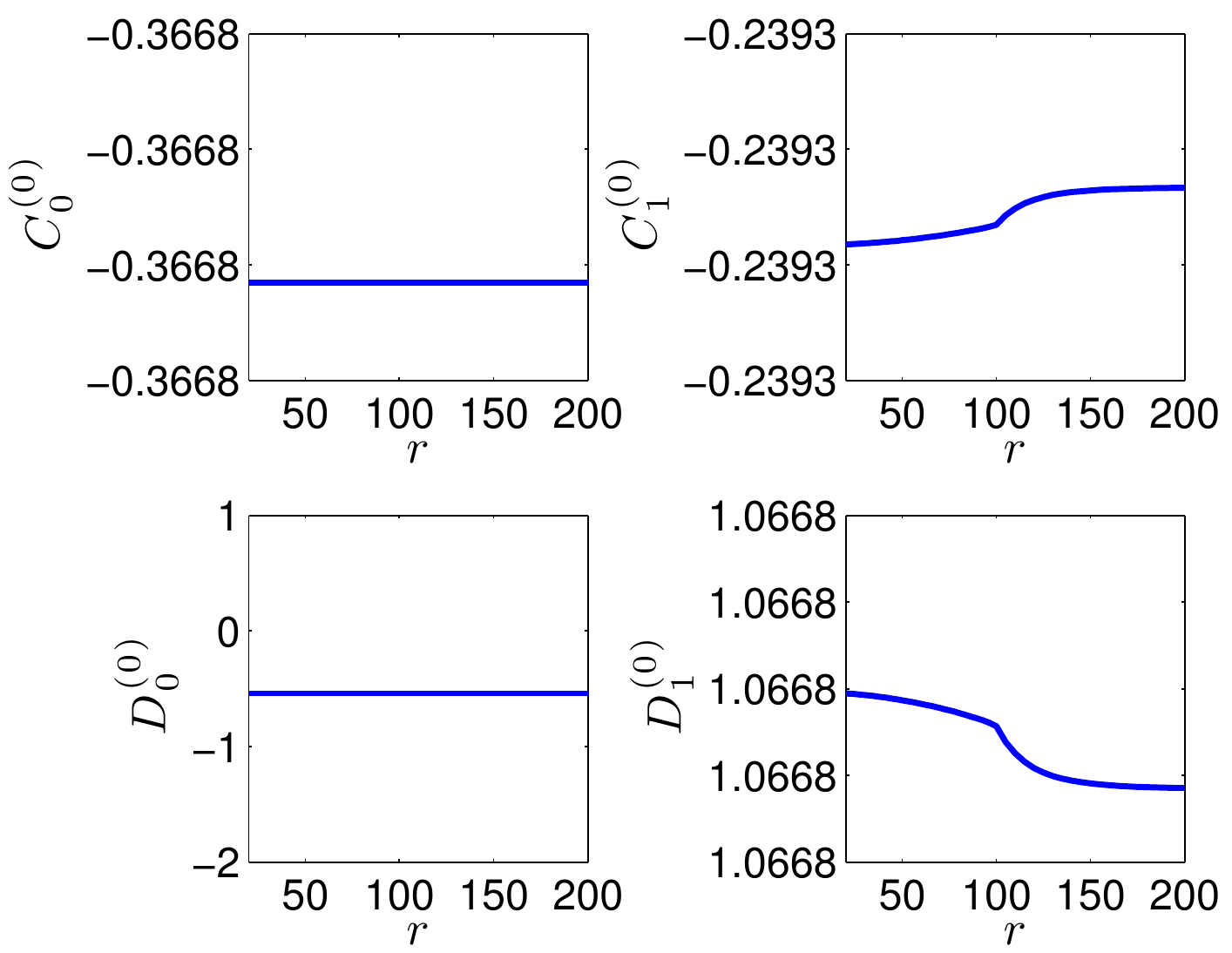}}
  \caption{The index functions and the asymptotic constants for the
    $\gamma = .01$ case of CQNLS in harmonic $k=0$.}
  \label{f:idx_example}
\end{figure}

\subsection{Boundary Value Problems \& Artificial Boundary Conditions}
\label{s:abcs}
Since this algorithm requires us to compute on a finite domain, $[0,
{\rmax}]$, we must introduce artificial boundary conditions at
$\rmax$.  These are discussed in the above references, but include
\begin{gather*}
  R(\rmax) +\frac{\rmax}{1+\rmax} R'(\rmax)=0\\
  \partial_\omega|_{\omega=1} R(\rmax) +\partial_\omega|_{\omega=1}
  R'(\rmax)=0
\end{gather*}
for the 3D problems.  Again, additional details are given in
\cite{Marzuola:2010p5770, Simpson:2010p8489}.  An example of an {\it a
  posteriori} check of these being satisfied is given in Figure
\ref{f:cqnls_absc}, which includes computations of the soliton, the
unstable eigenstate, and the solutions of the boundary value problems
for Proposition \ref{p:cqnls_k0}.

\begin{figure}
  \includegraphics[width=2.4in]{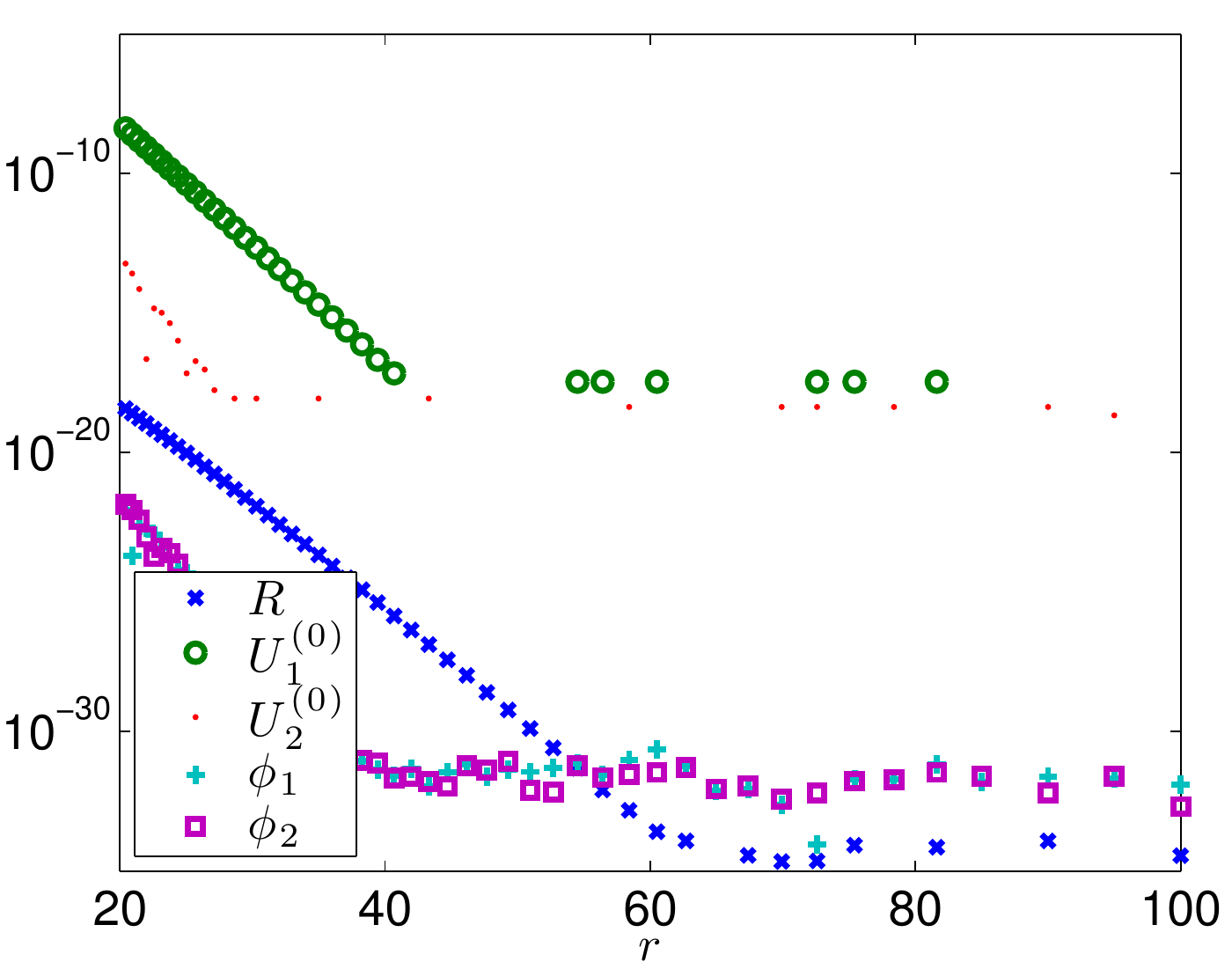}
  \caption{Verification of the artificial boundary conditions for the
    soliton, $R$, the eigenstate, $\boldsymbol{\phi}$, and the
    solutions of the boundary value problems $U^{(0)}_j$ for the
    $\gamma=0.01$ case of CQNLS.}
  \label{f:cqnls_absc}
\end{figure}

The boundary value problems were solved on the domain $[0, 100]$ in
all cases.  For the 3D problems, the tolerance was set at $10^{-12}$.
for the 1D problems, the absolute tolerance was $10^{-8}$ and relative
tolerance was $10^{-10}$.

\subsection{Root Finding}
\label{s:roots}
To compute roots of the inner products, or combinations of them, we
wrap our routines within the {\tt fzero} routine, with default
settings, of \matlab; we define a function that, for a given value of
$\sigma$ or $\gamma$, returns the inner product, and pass this to {\tt
  fzero}.  We found this could provide approximately six digits of
accuracy in our 3D computations, notably for the bounds
\eqref{e:3dnls_bound} and \eqref{e:cqnls:bound}, while for the 1D
computations, we were able to get eleven digits.

There appears to be a loss of sensitivity to the input $\sigma$ or
$\gamma$ for the 3D problems.  An example of this appears in Figure
\ref{f:nls3d_j0_slow}, where we see that for a range of $\sigma$
values near the root, $J_1^{(0)}$ is piecewise constant.  Refinements
to our approach, with higher order corrections to the artificial
boundary conditions, may improve the accuracy.

\begin{figure}
  \includegraphics[width=2.4in]{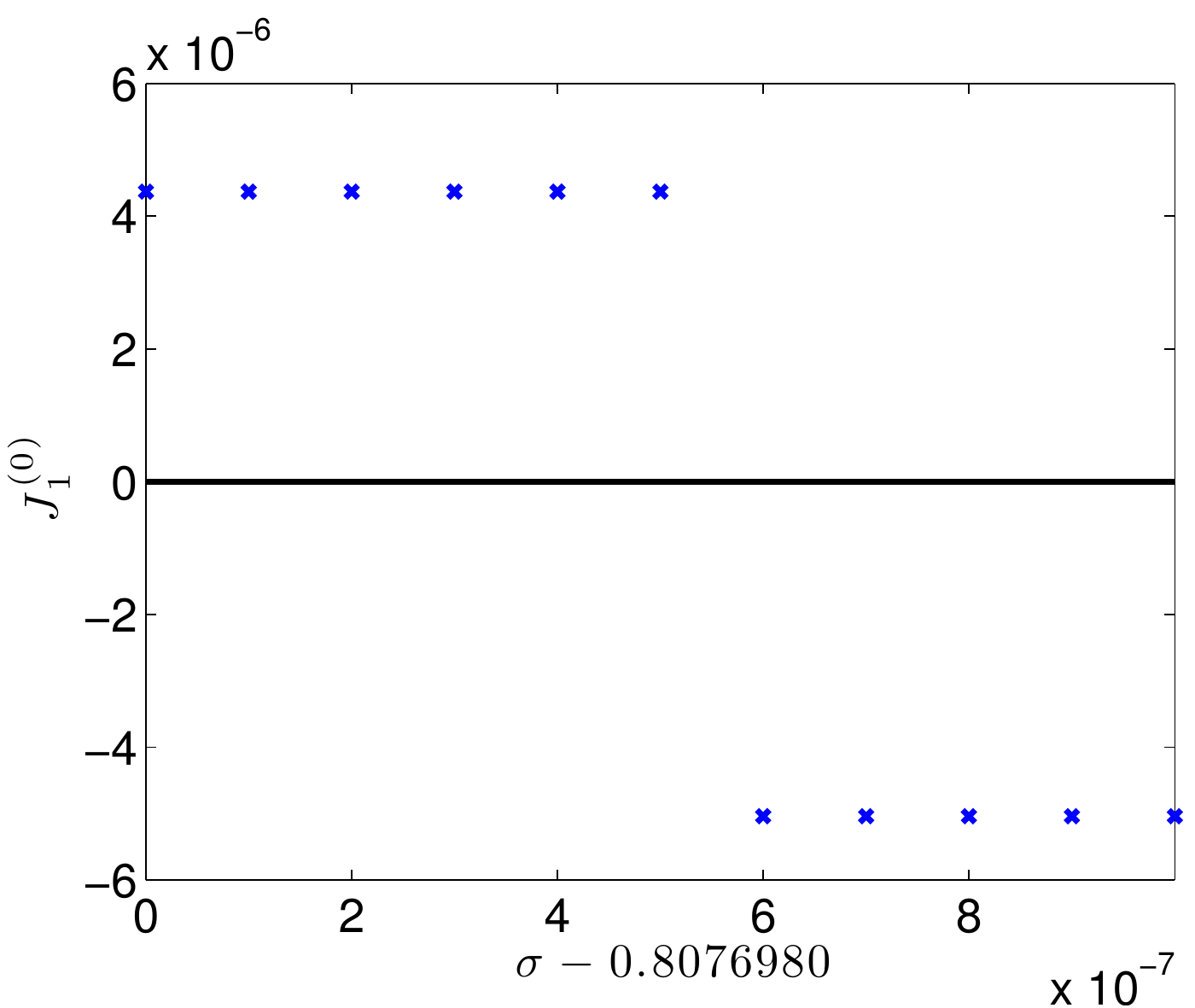}
  \caption{Zooming in on the root of $J_1^{(0)}$ inner product of the
    3D NLS equation, the solution is appears piecewise constant.  This
    is a closeup of Figure \ref{f:nls3d_j_0_var}(a). }
  \label{f:nls3d_j0_slow}
\end{figure}

\clearpage

\bibliographystyle{abbrv} \bibliography{spec_prop_paper}

\end{document}

%% file: gsmacros.tex
\newcommand{\abs}[1]{\lvert#1\rvert}
\newcommand{\norm}[1]{\lVert#1\rVert}
\newcommand{\paren}[1]{\left(#1\right)}
\newcommand{\bracket}[1]{\left[#1\right]}
\newcommand{\set}[1]{\left\{#1\right\}}

\newcommand{\matlab}{\textsc{Matlab} }

\newcommand{\ind}{{\mathrm{ind}}}
\newcommand{\rad}{{\mathrm{rad}}}

\newcommand{\spn}{{\mathrm{span}}}
\newcommand{\rmax}{r_{\max}}
\newcommand{\xmax}{x_{\max}}
\newcommand{\grad}{\nabla}

\newcommand{\ess}{\mathrm{ess}}

\newcommand{\bx}{\mathbf{x}}

\newcommand{\inner}[2]{\left\langle #1,#2\right\rangle}

\newcommand{\R}{\mathbb{R}}

\newcommand{\calU}{\mathcal{U}}

\newcommand{\calL}{{\mathcal{L}}}
\newcommand{\calB}{{\mathcal{B}}}
\newcommand{\calV}{\mathcal{V}}

\newtheorem{thm}{Theorem}
\newtheorem{prop}{Proposition}[section]
\newtheorem{cor}[prop]{Corollary}

\newtheorem{defn}[prop]{Definition}